\documentclass[a4paper,11pt]{article}
\usepackage[utf8]{inputenc}
\usepackage[T1]{fontenc}
\usepackage[english]{babel}

\usepackage[tbtags]{amsmath}
\usepackage{amssymb}
\usepackage{amsfonts}
\usepackage{amsthm}
\usepackage{calrsfs}
\usepackage{array}
\usepackage{bbm}
\usepackage{upgreek}
\usepackage{pdfsync}

\usepackage{mathtools}
\usepackage[svgnames]{xcolor}
\usepackage{hyperref}

\numberwithin{equation}{section}

\newcommand\B{\q {\bar B}}
\newcommand\barc{\bar c}
\newcommand\barcc{\tilde c}

\newcommand\bD{{\bs D}}
\newcommand\bd{{\bs d}}
\newcommand{\bs}[1]{{\boldsymbol{#1}}}

\newcommand\cd{{\hat d}}
\newcommand\chr{{R}}
\newcommand\cs{{S}}
\newcommand\cy{{Y}}
\newcommand\Ddir{{\hat D_{dir}}}

\newcommand{\ds}{\displaystyle}
\newcommand\dv{\mathop{\rm div}}
\newcommand\E{\q {\bar E}}
\newcommand\EE{\q {\bar E}_0}

\newcommand{\eb}{{\epsilon_0}}
\newcommand{\ed}{{\hat e_{down}}}
\newcommand{\ee}{{\bar\epsilon}}
\newcommand{\edir}{{\hat e_{dir}}}
\newcommand{\el}{{\hat e_{left}}}
\newcommand{\er}{{\hat e_{right}}}
\newcommand{\eup}{{\hat e_{up}}}
\newcommand{\eun}{{\epsilon_1}}

\renewcommand{\H}{{\cal H}}

\newcommand\hr{{\hat r}}
\renewcommand\iint{\displaystyle\int_{|y|<1}}

\newcommand\ind{{(i,\theta)}}

\newcommand{\kd}{\kappa_{down}}
\newcommand{\kdir}{\kappa_{dir}}
\newcommand{\kl}{\kappa_{left}}
\newcommand{\kr}{\kappa_{right}}
\newcommand{\kup}{\kappa_{up}}

\newcommand{\m}[1]{\mathbbm{#1}}
\newcommand{\md}{{\hat \mu_{down}}}
\newcommand{\mdir}{{\hat \mu_{dir}}}
\newcommand{\ml}{{\hat \mu_{left}}}
\newcommand{\mr}{{\hat \mu_{right}}}
\newcommand{\mup}{{\hat \mu_{up}}}
\newcommand\N{{\cal N}_\sigma}
\newcommand{\nd}{{\hat \nu_{down}}}
\newcommand{\ndir}{{\hat \nu_{dir}}}
\newcommand{\nl}{{\hat \nu_{left}}}
\newcommand{\nr}{{\hat \nu_{right}}}
\newcommand{\nup}{{\hat \nu_{up}}}
\newcommand{\pnu}{\partial_\nu}
\newcommand{\pp}{\Pi}
\newcommand{\ps}{{\partial_s}}
\newcommand{\py}{{\partial_y}}
\newcommand{\q}[1]{\mathcal{#1}}

\newcommand\RR{{\cal R}}
\newcommand{\sd}{{\hat s_{down}}}
\newcommand{\sdir}{{\hat s_{dir}}}
\newcommand\sigd{\sigma_{down}}
\newcommand\sigdir{\sigma_{dir}}
\newcommand\sigl{\sigma_{left}}

\newcommand\sigu{\sigma_{up}}

\renewcommand{\sl}{{\hat s_{left}}}
\newcommand{\smin}{{\hat s_{min}}}

\newcommand{\srb}{{\hat s_{right,+}}}
\newcommand{\su}{{\hat s_{up}}}
\newcommand{\sur}{{\hat s_{up, rel}}}

\newcommand{\Sdir}{{\hat S_{dir}}}
\newcommand{\Sl}{{\hat S_{left}}}
\newcommand{\Smin}{{\hat S_{min}}}

\newcommand{\Srb}{{\hat S_{right,+}}}
\renewcommand\SS{{\cal S}}
\newcommand{\Su}{{\hat S_{up}}}
\newcommand{\Sur}{{\hat S_{up, rel}}}

\newcommand{\td}{{\hat \theta_{down}}}
\newcommand{\tdir}{{\hat \theta_{dir}}}

\newcommand{\tl}{{\hat \theta_{left}}}

\newcommand{\tr}{{\hat \theta_{right}}}
\newcommand{\ts}{{s}}
\newcommand\tT{{{\cal T}}}

\newcommand{\tup}{{\hat \theta_{up}}}
\newcommand{\tx}{{x}}
\newcommand{\ty}{{y}}

\newcommand{\vc}[2]{\begin{pmatrix} #1\\#2\end{pmatrix}}

\DeclareMathOperator{\sgn}{\mathrm{sgn}}

\theoremstyle{plain}
\newtheorem{thm}{Theorem}
\newtheorem*{thm*}{Theorem}

\newtheorem{prop}{Proposition}[section]
\newtheorem{cor}[prop]{Corollary}
\newtheorem{lem}[prop]{Lemma}
\newtheorem{defi}[prop]{Definition}
\newtheorem{cl}[prop]{Claim}

\theoremstyle{definition}

\theoremstyle{remark}
\newtheorem*{nb}{Remark}

\makeatletter
\def\blfootnote{\xdef\@thefnmark{}\@footnotetext}
\makeatother

\title{\bf Blow-up solutions to the semilinear wave equation with a stylized pyramid as a blow-up surface
}
\author{Frank Merle\footnote{Both authors are supported by the ERC Advanced Grant no. 291214, BLOWDISOL. H.Z. is also supported by ANR project ANA\'E ref. ANR-13-BS01-0010-03.}\\
{\it \small Universit\'e de Cergy Pontoise and IHES}\\
Hatem Zaag{\small *}\\
{\it \small Universit\'e Paris 13, Sorbonne Paris Cit\'e},\\
{\it \small LAGA, CNRS (UMR 7539), F-93430, Villetaneuse, France}
}
\allowdisplaybreaks

\begin{document}

% With AMS-LaTeX, \maketitle follows the abstract
\maketitle   

\begin{abstract}
We consider the semilinear wave equation with subconformal power nonlinearity in two space dimensions. We construct a finite-time blow-up solution with an isolated characteristic blow-up point at the origin, and a blow-up surface which is centered at the origin and has the shape of a stylized pyramid, whose edges follow the bisectrices of the axes in $\m R^2$.
The blow-up surface is differentiable outside the bisectrices. 
As for the asymptotic behavior in similariy variables, the solution converges to the classical one-dimensional soliton outside the bisectrices. On the bisectrices outside the origin, it converges (up to a subsequence) to a genuinely two-dimensional stationary solution, whose existence is a by-product of the proof.
 At the origin, it behaves like the sum of 4 solitons localized on the two axes, with opposite signs for neighbors.\\
This is the first example of a blow-up solution with a characteristic point in higher dimensions, showing a really two-dimensional behavior.
 Moreover, the points of the bisectrices outside the origin give us the first example of non-characteristic points where the blow-up surface is non-differentiable.

\end{abstract}

\medskip


% 35B40    	Asymptotic behavior of solutions
% 35B44    	Blow-up
% 35L05    	Wave equation
% 35L67    	Shocks and singularities
% 35L71    	Semilinear second-order hyperbolic equations

\medskip


%%%%%%%%%%%%%%%%%%%%%%%%%%%%%%%%%%%
%%%%%%%%%%%%%%%%%%%%%%%%%%%%%%%%%%%
\section{Introduction}
%%%%%%%%%%%%%%%%%%%%%%%%%%%%%%%%%%%
%%%%%%%%%%%%%%%%%%%%%%%%%%%%%%%%%%%
We consider the subconformal semilinear wave equation in 2 space dimensions:
\begin{equation}\label{equ}
\left\{
\begin{array}{l}
\partial^2_{t} u =\Delta u+|u|^{p-1}u,\\
u(0)=u_0\mbox{ and }\partial_t u(0)=u_1,
\end{array}
\right.
\end{equation}
where $u(t):x\in{\m R}^2 \rightarrow u(x,t)\in{\m R}$, $1<p<5$, 
$(u_0, u_1) \in H^1\times L^2(\m R^2)$.
The Cauchy problem is locally wellposed, and we have the existence of blow-up solutions from Levine \cite{Ltams74}.

\medskip

Equation \eqref{equ} can be considered as a lab model for blow-up in hyperbolic equations, because it captures features common to a whole range of blow-up problems arising in various nonlinear physical models, in particular in general relativity (see Donninger, Schlag and Soffer \cite{DSScmp12}), and also for self-focusing waves in nonlinear optics (see Bizo\'n, Chmaj and Szpak \cite{BCSjmp11}). 

\medskip

In this paper, our aim is to construct the first example of a blow-up solution with a truely two-dimensional behavior. In particular, our solution will be non-radial and will not depend only on a one-dimensional variable. In fact, we will construct a multi-soliton solution here, since we will have 4 decoupled solitons in some backward light cone centered at the origin. As a matter of fact, there has been many papers addressing the question of multi-solitons in the literature, for various PDEs: for the generalized KdV equation, see Martel \cite{Majm05, Msema05}, Martel, Merle and Tsai \cite{MMTcmp02}; for NLS, see Merle \cite{Mcmp90}, Martle and Merle \cite{MMihp06}, Martel, Merle and Tsai \cite{MMTdmj06}, C\^ote, Martel and Merle \cite{CMMrmia11} as well as Martel and Rapha\"el \cite{MRbub15}; for water waves, see Ming, Rousset and Tzvetkov \cite{MRTsjma15};  for the Yamabe flow, see Daskalopoulos, Del Pino and Sesum; for the subcritical wave equation, see Merle and Zaag \cite{MZdmj12} as well as Côte and Zaag \cite{CZcpam13}; for the critical wave equation, see Duyckaerts, Kenig and Merle \cite{DKMcjm13}.

\medskip

More generally, constructing a solution to some 
PDE
with a prescribed behavior (not necessarily multi-solitons solutions) is an important question. That question was solved for (gKdV) by C\^ote \cite{Cjfa06,Cdmj07}, and also for parabolic equations exhibiting blow-up, like the semilinear heat equation by Bressan \cite{Biumj90, Bjde92} (with an exponential source), Merle \cite{Mcpam92}, Bricmont and Kupiainen \cite{BKnonl94}, Merle and Zaag in \cite{MZdmj97, MZcras96}, Schweyer \cite{Sjfa12} (in the critical case), Mahmoudi, Nouaili and Zaag \cite{MNZna16} (in the periodic case), the complex Ginzburg-Landau equation by Zaag \cite{Zihp98} and also by Masmoudi and Zaag in \cite{MZjfa08}, a complex heat equation with no gradient structure by Nouaili and Zaag \cite{NZcc14}, a gradient perturbed heat equation in the subcritical case by Ebde and Zaag in \cite{EZsema11}, then by Tayachi and Zaag in the critical case in \cite{TZ15} (see also \cite{TZnormandie15}), or a strongly perturbed complex-valued heat equation in Nguyen and Zaag \cite{NZcs14}. Other examples are available for Schr\"odinger maps (see Merle, Rapha\"el and Rodnianski \cite{MRRim13}), and also fo the Keller-Segel model (see Rapha\"el and Schweyer \cite{RSma14}, and also Ghoul and Masmoudi \cite{GM16}).

\medskip

If $u$ is a blow-up solution of equation \eqref{equ}, we define (see for example Alinhac \cite{Apndeta95}) a 1-Lipschitz 
graph $x\mapsto T(x)$ such that the domain of definition of $u$ is written as 
\begin{equation}\label{defdu}
D=\{(x,t)\;|\; x\in \m R^2\mbox{ and } 0\le t< T(x)\}.
\end{equation}
The graph of $T$ is called the blow-up surface (or curve if $N=1$) of $u$ and will be denoted by $\Gamma$. 
A point $x_0\in{\m R}^2$ is a non-characteristic point if there are 
\begin{equation}\label{nonchar}
\delta_0\in(0,1)\mbox{ and }t_0<T(x_0)\mbox{ such that }
u\;\;\mbox{is defined on }{\cal C}_{x_0, T(x_0), \delta_0}\cap \{t\ge t_0\}
\end{equation}
where 
\begin{equation}\label{defcone}
{\cal C}_{\bar x, \bar t, \bar \delta}=\{(x,t)\;|\; t< \bar t-\bar \delta|x-\bar x|\}.
\end{equation}
If not, we say that $x_0$ is a characteristic point.
 We denote by $\RR\subset {\m R}^2$ the set of non-characteristic points and by $\SS$ the set of characteristic points. 

\bigskip

The one-dimensional case in equation \eqref{equ} has been understood completely, in a series of papers by the authors \cite{MZjfa07, MZcmp08, MZxedp10, MZajm12, MZdmj12} and in C\^ote and Zaag \cite{CZcpam13} (see also the note \cite{MZxedp10}). This includes the first example of a solution with  a characteristic point in \cite{MZajm12} (in dimension one). For a general blow-up solution, we also proved that $\SS$ is made of isolated points (see \cite{MZdmj12}), and that the blow-up curve is of class $C^1$ on $\RR$ (see \cite{MZajm12}).
See Caffarelli and Friedman in \cite{CFtams86,CFarma85} for earlier results.

\bigskip

In higher dimensions $N\ge 2$, the situation is not as clear.\\
In fact, the blow-up rate is known (see \cite{MZajm03}, \cite{MZma05} and \cite{MZimrn05}; see also the extensions by Hamza and Zaag in \cite{HZnonl12} and \cite{HZjhde12} including the superconformal case in \cite{HZdcds13} also treated in Killip, Stoval and Vi\c san \cite{KSVma14}).\\
For the asymptotic behavior and the regularity of the blow-up surface, 
the only known result is at non-characteristic points, where we show in \cite{MZtams14} and \cite{MZcmp14} that $\Gamma$ is $C^1$, under a reasonable assumption on the profile.
The radial case outside the origin is also completely understood in \cite{MZbsm11}, since it reduces to a perturbation of the one-dimensional case.\\
Concerning the behavior of radial solutions at the origin, Donninger and Sch{\"o}rkhuber were able to prove the stability of the space-independent solution 
with respect to perturbations in initial data, in the Sobolev subcritical range \cite{DSdpde12} and also in the supercritical range in \cite{DStams14}. Some numerical results are available in a series of papers by Bizo\'n and co-authors (see \cite{Bjnmp01}, \cite{BCTnonl04}, \cite{BZnonl09}).  
See also Killip and Vi\c san \cite{KV11}. 

\bigskip

In this paper, we address the question of the existence of blow-up solutions to equation \eqref{equ} with $\SS\neq \emptyset$. As asserted above, the first example of such a solution was given in one space dimension in \cite{MZajm12}. Later, C\^ote and Zaag \cite{CZcpam13} constructed other examples showing multi-solitons. Both approaches extend to the radial case and to perturbations of equation \eqref{equ} with lower order terms (see Merle and Zaag \cite{MZbsm11}, Hamza and Zaag \cite{HZbsm13}). Of course, all these one-dimensional examples can be considered as trivial $2$-dimensional solutions, where $\SS$ is either a line, or a circle. From the finite speed of propagation, we may have parallel lines or concentic circles, and the local blow-up behavior is always rigorously one dimensional.
In particular, no example is known in higher dimensions, with $\SS$ locally reduced to an isolated point. The aim of this paper is precisely to provide such an example. Moreover, we will give a sharp description of the blow-up behavior and the blow-up surface, locally near the characteristic point (this is related to an explicit description of the instabilities of the 4-soliton solution we construct at the origin).

\bigskip

Before stating our result, 
let us introduce the following similarity variables, for any $(x_0,T_0)$ such that $0< T_0\le T(x_0)$:
\begin{equation}\label{defw}
w_{x_0,T_0}(y,s)=(T_0-t)^{\frac 2{p-1}}u(x,t),\;\;y=\frac{x-x_0}{T_0-t},\;\;
s=-\log(T_0-t).
\end{equation}
If $T_0=T(x_0)$, we write $w_{x_0}$ for short. 
The function $w_{x_0,T_0}$ (we write $w$ for simplicity) satisfies the 
following equation for all $|y|<1$ and $s\ge -\log T_0$:
\begin{equation}\label{eqw}
\partial^2_sw-\q L w+\frac{2(p+1)}{(p-1)^2}w-|w|^{p-1}w=
-\frac{p+3}{p-1}\partial_sw-2y\cdot\nabla \partial_s w
\end{equation} 
where
\begin{equation}\label{defro}
\q L w =\frac 1\rho \dv \left(\rho \nabla w-\rho(y\cdot \nabla w)y\right),\;\; \rho(y)=(1-|y|^2)^\alpha\mbox{ and }
\alpha=\frac{5-p}{2(p-1)}>0.
\end{equation}
Equation \eqref{eqw} is studied in the energy space 
\begin{equation}\label{defnh}
\H=\H_0\times L^2_\rho\mbox{ where }\|q_1\|_{\H_0}^2\equiv \iint \left(q_1^2+|\nabla q_1|^2-|y\cdot \nabla q_1|^2)\right)\rho dy.
\end{equation}
We also introduce for all $|\bs d|<1$ the following stationary solutions of \eqref{eqw} (or solitons) depending only on the one-dimensional variable $y\cdot \bs d$ (if $\bs d\neq 0$) and defined for all $|y|<1$ by 
\begin{equation}\label{defkd}
\kappa(\bs d,y)=\kappa_0 \frac{(1-|\bs d|^2)^{\frac 1{p-1}}}{(1+\bs d\cdot y)^{\frac 2{p-1}}}\mbox{ where }\kappa_0 = \left(\frac{2(p+1)}{(p-1)^2}\right)^{\frac 1{p-1}},
\end{equation}
and
\begin{equation}\label{solpart}
\bar d(s) = - \tanh \bar \zeta(s)\mbox{ where }\bar \zeta(s) = \left(\frac{p-1}4\right) \log s -\frac{(p-1)}4 \log \left(\frac{p-1}{4\barc}\right)
\end{equation}
which is an explicit solution to the ODE
\[
\frac 1{\barc} \frac{d\bar\zeta}{ds} = e^{-\frac 4{p-1}\bar \zeta}
\]
where $\barc(p)>0$ is 
defined 
in Lemma \ref{propdyn}.
Note that we have for some $C_0(p)>0$,
\begin{equation}\label{devdbar}
1+\bar d(s) = C_0s^{-\frac{p-1}2}+O(s^{-(p-1)})\sim C_0s^{-\frac{p-1}2}\mbox{ as }s\to \infty.
\end{equation}

\medskip
Let
$(e_1, e_2)$ 
be
the canonical basis of $\m R^2$.
This is the statement of our result:
%%%%%%%%%%%%%%%%%%%%%%%%%%%%%%%%%%%%
%%%%%%%%%%%%%%%%%%%%%%%%%%%%%%%%%%%%
\begin{thm}[Existence of a blow-up solution with 
an isolated characteristic blow-up point and a blow-up surface which is a pyramid at the first order]\label{mainth} There exists $u(x,t)$ a solution to equation \eqref{equ},
which is symmetric with respect to the axes and anti-symmetric with respect to bisectrices,
with the following properties:\\ 
{\bf  (A) (Blow-up with an isolated characteristic point)}. $u(x,t)$ 
blows up on some blow-up graph $\Gamma = \{(x,T(x))\}$,
and for some $\delta>0$, we have $\SS\cap B(0, \delta) = \{0\}$.\\
 {\bf (B) (The blow-up surface is nearly a pyramid)}. $T$ is symmetric with respect to the axes and the bisectrices and $C^1$ outside the bisectrices. Moreover, when $0\le x_2<x_1\le\delta$, we have for some $C_0=C_0(p)>0$:
\begin{align*}
T(x)&=T(0)-x_1(1-C_0|\log x_1|^{-\frac{p-1}2})+o(x_1|\log x_1|^{-\frac{p-1}2})+o(x_2|\log x_1|^{-\frac{p-1}4}),\\
\nabla T(x)& =-e_1(1-C_0|\log x_1|^{-\frac{p-1}2})+o(|\log x_1|^{-\frac{p-1}2})e_1+o(|\log x_1|^{-\frac{p-1}4})e_2.
\end{align*}
{\bf (C) (Blow-up behavior of the solution)}.
We have the following behavior for $w_x$ for $0\le x_2\le x_1\le \delta$
as $s\to \infty$:\\
(i) if $x=0$, then 
\begin{equation}\label{cprofile0}
\left\|w_{0}(y,s)-\left(\kappa(\bar d(s)e_1,y)+\kappa(-\bar d(s)e_1,y) - \kappa(\bar d(s)e_2,y) - \kappa(-\bar d(s)e_2,y)\right)\right\|_{\q H}\to 0
\end{equation}
where $1+\bar d(s)\sim C_0s^{-\frac{p-1}2}$ as $s\to \infty$;\\
(ii) if 
$x_2<x_1$, 
then $w_x(s)$ converges as $s\to +\infty$ to 
$\kappa(d(x)e_1)$, 
with
\begin{equation*}%\label{behdx0}
d(x)+1\sim C_0|\log x_1|^{-\frac{p-1}2}\mbox{ as }x\to 0. 
\end{equation*}
(iii)  if $x\neq 0$ with $x_1= x_2$, then $w_x(s_n)$ converges to some stationary solution $w^*_x$ for some sequence $s_n \to \infty$, where $w^*_x$ is a genuinely two-dimensional stationary solution of equation \eqref{eqw}.
\end{thm}
%%%%%%%%%%%%%%%%%%%%%%%%%%%%%%%%%%%%
%%%%%%%%%%%%%%%%%%%%%%%%%%%%%%%%%%%%
\begin{nb}
The existence of the new stationary solution of equation \eqref{eqw} just mentioned at the end of this theorem, follows from an indirect argument we use when $x$ is on the bisectrices. 
\end{nb}
\begin{nb}
From the symmetries of the solution, we see that $u(x,t)=0$ on the bisectrices. In one space dimension, such a property implies that $x$ is a characteristic point. Surprisingly, in our two-dimensional setting, only the origin is a characteristic point, and the other points on the bisectrices are non-characteristic blow-up points, showing a genuinely two-dimensional behavior.
\end{nb}
\begin{nb} 
Note that the behavior of 
$\bar d(s)$
given in \eqref{solpart} for this 4-soliton solution in 2 dimensions is the same as for the 2-soliton solution in one space dimension (see \cite{MZdmj12} and \cite{CZcpam13}).
\end{nb}
\begin{nb} Our estimates can be adapted to show that the origin is the only characteristic point in $\m R^2$. However, in order to keep this paper in a reasonable length, we did not include that argument in the paper.
\end{nb}
\begin{nb} When $0<x_2=x_1 \le \delta$, the estimate on $T(x)$ in part (B) does hold, by continuity of $T$. Moreover, we can compute upper and lower left derivatives for $T$ along any direction non parallel to the bisectrix $\{x_1=x_2\}$, and the same holds from the right. In particular, 
 if $|\omega|=1$ and $\omega_2-\omega_1>0$, then:\\
$\partial_{\omega,r,\pm} T(x) = (-1+C_0 |\log x_1|^{-\frac{p-1}2}+o(|\log x_1|^{-\frac{p-1}2}))\omega_1+o(|\log x_1|^{-\frac{p-1}4}) \omega_2$;\\
$\partial_{\omega,l,\pm} T(x) = (-1+C_0 |\log x_1|^{-\frac{p-1}2}+o(|\log x_1|^{-\frac{p-1}2}))\omega_2+o(|\log x_1|^{-\frac{p-1} 4}) \omega_1$
 as $x\to 0$, where the subscript $r$ and $l$ stands for ``right'' and ``left'', whereas the subscript $\pm$ stands for ``upper'' or ``lower''.\\
At the origin, $T$  has a right derivative with respect to $x_1$ whose value is  $\partial_{x_1,r}T(0)=-1$, with similar statements from the left and in the direction $x_2$.
\end{nb}
\begin{nb}
Our result can be generalized to other pyramids, with any regular polygon as a basis. 
In higher space dimensions $N\ge 3$, we naturally generalize our results to a pyramid with a hypercube as a basis.
\end{nb}
\begin{nb}
 Using a Lorentz transform near  some point of the bisectrices different from the origin, we can tilt the blow-up surface and obtain the existence of a blow-up solution of equation \eqref{equ}, with a tent-shaped (at the first order) blow-up surface, no characteristic point in some neighborhood, a slope approaching $\frac{\sqrt 2}2$ and an upper edge depending on $x_2$. This tent is in fact new and different from the one obtained by considering a solution depending only on $x_1$ with a characteristic point at the origin. Indeed, in two space dimensions, this ``naive'' tent has a line of characteristic points on its upper edge, a slope approaching $1$, and an upper edge that does not depend on $x_2$.
\end{nb}

This paper is organized in 3 sections:\\
- In Section \ref{SecConsU}, we construct a blow-up solution to equation \eqref{equ} defined only in the backward light cone with vertex $(0,T(0))$ and showing 4 solitons for $w_0$ at the origin as in \eqref{cprofile0}.
Using the finite speed of propagation, we derive from the latter a blow-up solution to the Cauchy problem of equation \eqref{equ}.\\
- In Section \ref{secwx}, we leave the origin and focus on the behavior of $w_x$ for $x\neq 0$. Using the decomposition into 4 solitons together with the upper blow-up bounds from \cite{MZimrn05}, we first derive some rough estimates on the blow-up surface, showing in particular that it is under some pyramid, with a flatter slope line. Then, we find the behavior of $w_x(y,s)$ for large $s$, which turns to be different from the behavior of $w_0(y,s)$. This shows in particular that the 4-soliton solution is unstable.\\
- In Section \ref{secreg}, we make the link between dynamics of $w_x$ from the previous section and the local behavior of $T(x)$. This is in fact the new feature of our paper which makes its originality. In particular, we use families of moving non-characteristic cones together with subtle elementary geometric methods to derive the pyramid shape of the blow-up surface. The delicate case is the case where $x$ is on the bisectrices, since this is a new situation, not encountered in dimension 1. It is worth noticing that the moving cone technique simplifies the ``moving plane'' technique we use earlier in one space dimension in \cite{MZcmp08}.
\begin{nb}
Let us notice that the construction step (given in Section \ref{SecConsU}) follows the classical scheme of a ``construction with a prescribed behavior'', and has only a technical difficulty. In fact, the very heart of our argument lays in the following section, where we aim at understanding the instability of the 4-soliton solution we have for $w_0$, when we move outside the origin to consider the behavior of $w_x$. We did that already in one space dimension in \cite{MZdmj12}, but the situation is much more delicate here, mainly because of the dynamics at the bisectrices. 
\end{nb}
{\bf Acknowledgment}:
We are indebted to the referee for his extremely careful reading and his numerous suggestions which undoubtedly improved both the content and the shape of the paper. 
%%%%%%%%%%%%%%%%%%%%%%%%%%%%%%%%%%%%
%%%%%%%%%%%%%%%%%%%%%%%%%%%%%%%%%%%%
\section{Construction of a solution to equation \eqref{equ} with a prescribed 4-soliton behavior}
\label{SecConsU}
%%%%%%%%%%%%%%%%%%%%%%%%%%%%%%%%%%%%
%%%%%%%%%%%%%%%%%%%%%%%%%%%%%%%%%%%%
We proceed in 2 sections:\\
- In Section \ref{seccons}, we construct a 4-soliton blow-up soliton to equation \eqref{eqw} obeying the behavior \eqref{cprofile0}.\\
- In Section \ref{subprop0}, we derive from that solution a blow-up  to the Cauchy problem of equation \eqref{equ}.
%%%%%%%%%%%%%%%%%%%%%%%%%%%%%%%%%%%%
%%%%%%%%%%%%%%%%%%%%%%%%%%%%%%%%%%%%
\subsection{Construction of a 4-soliton solution in similarity variables}\label{seccons}
%%%%%%%%%%%%%%%%%%%%%%%%%%%%%%%%%%%
%%%%%%%%%%%%%%%%%%%%%%%%%%%%%%%%%%%
In this section, we aim at constructing a 4-soliton blow-up solution to equation \eqref{equ} obeying the behavior \eqref{cprofile0}. In fact, we will adopt the similarity variables' framework \eqref{defw}, working only for $|y|<1$, which corresponds to the backward light cone with vertex $(0, T(0))$ in the original setting $u(x,t)$. Later in Section \ref{subprop0}, we extend the solution outside the light-cone.

\medskip

Thus, our aim becomes to construct a 4-soliton solution to equation \eqref{eqw} defined in the unit ball and
obeying the behavior \eqref{cprofile0}. 

\medskip

In fact, we will consider some $s_0 \ge 0$ and find suitable initial data for equation \eqref{eqw} at time $s=s_0$ so that \eqref{cprofile0} is satisfied. A natural choice would be to take $s_0$ large enough, and initial data close to the sum of the 4 solitons expected in \eqref{cprofile0}. For that reason, the introduction of a new family of generalized solitons close to $\kappa(\bs d)$ \eqref{defkd} will be very useful. 

\medskip

Let us then introduce for all $|\bs d|<1$ and $\nu >-1+|\bs d|$, $\kappa^*(\bs d,\nu,y) = (\kappa_1^*, \kappa_2^*)(\bs d,\nu,y)$ where 
\begin{align}\label{defk*}
\kappa_1^*(\bs d,\nu, y) & =
\kappa_0\frac{(1-|\bs d|^2)^{\frac 1{p-1}}}{(1+\bs d\cdot y+\nu)^{\frac 2{p-1}}}, \\
\kappa_2^*(\bs d,\nu, y) & = \nu \pnu \kappa_1^*(\bs d,\nu, y) =
-\frac{2\kappa_0\nu}{p-1}\frac{(1-|\bs d|^2)^{\frac 1{p-1}}}{(1+\bs d\cdot y+\nu)^{\frac {p+1}{p-1}}}.
\end{align}
We refer to these functions as ``generalized solitons'' or solitons for short. Notice that for any $\mu\in\m R$, $\kappa^*(\bs d,\mu e^s,y)$ is a solution to equation \eqref{eqw} (when written as a first order equation) which depends on time $s$ (if $\mu\neq 0$) and  only on one space direction, namely the coordinate along $\bs d$ (if $\bs d\neq 0$). Notice also that $\kappa^*(\bs d,\mu e^s,y)$ are in fact time translations (in the $u(x,t)$ setting) of $\kappa(d,y)$. More precisely,  $\kappa^*(\bs d,\mu e^s,y)$  can be obtained from $\kappa(d,y)$ by applying the similarity variables' transformation \eqref{eqw} backwards with a scaling time $T=1$, then applying it again forward with a scaling time $T=1-\mu$.
 Moreover:\\
- $\kappa^*(\bs d,\mu e^s)\to (\kappa(\bs d),0)$ in $\H$ as $s\to -\infty$;\\
- when $\mu=0$, we recover the stationary solutions $(\kappa(\bs d),0)$ defined in \eqref{defkd};\\
- when $\mu>0$, the solution exists for all $|y|<1$ and $s\in \m R$
and converges to $0$ in $\H$ as $s\to \infty$ (it is a heteroclinic connection for equation \eqref{defw} between $(\kappa(\bs d),0)$ and $0$);\\
- when $\mu<0$, the solution exists for all $|y|<1$ and $s<\log\left(\frac {|\bs d|-1}\mu\right)$
and blows up at time $s=\log\left(\frac {|\bs d|-1}\mu\right)$.

\medskip

Given $s_0$ large enough, our initial data for equation \eqref{eqw} will be of the form
\begin{equation}\label{w0}
\vc{w(y,s_0)}{\partial_s w(y,s_0)} =
\sum_\ind
(-1)^{i+1}\kappa^*(\theta \bar d(s_0) e_i,\nu_0,y)\mbox{ where }
\nu_0 \in I_0\equiv [-s_0^{-\frac 12 - \frac{p-1}2}, s_0^{-\frac 12 - \frac{p-1}2}]
\end{equation}
and $\bar d(s_0)$ is given in \eqref{solpart}. If $s_0$ is large enough, we will show that for some well-chosen parameter $\nu_0$, the solution of equation \eqref{eqw} with initial data \eqref{w0} behaves as expected in \eqref{cprofile0}. This is the aim of the section:
%%%%%%%%%%%%%%%%%%%%%%%%%%%%%%%%%%%%
%%%%%%%%%%%%%%%%%%%%%%%%%%%%%%%%%%%%
\begin{prop}[A 4-soliton solution in the $w(y,s)$ setting]\label{propw} 
For any $s_0$ large enough, there exists $\nu_0\in\m R$ such that
equation \eqref{eqw} with initial data (at $s=s_0$) given by \eqref{w0} is defined for all $(y,s)\in(-1,1)\times [s_0, \infty)$, satisfies $(w(s), \ps w(s))\in\H$ for all $s\ge s_0$, and 
\begin{equation}\label{cprofile}
\left\|\vc{w(s)}{\ps w(s)} - 
\sum_\ind
(-1)^{i+1}\vc{\kappa(\theta \bar d(s) e_i,y)}{0}\right\|_{\H} \to 0\mbox{ as }s\to \infty,
\end{equation}
where $\bar d(s)$ is introduced in \eqref{solpart}. 
\end{prop}
%%%%%%%%%%%%%%%%%%%%%%%%%%%%%%%%%%%%
%%%%%%%%%%%%%%%%%%%%%%%%%%%%%%%%%%%%
\begin{nb} Note that equation \eqref{eqw} is wellposed in $H^1\times L^2(|y|<1)$, locally in time. Indeed, thanks to the transformation \eqref{defw}, this is a direct consequence of the wellposedness of equation \eqref{equ} in $H^1\times L^2(\m R^2)$ together with the finite speed of propagation. Since $w(s_0)\in H^1\times L^2(|y|<1)$ from \eqref{w0} for $s_0$ large enough (this comes from \eqref{devdbar} which implies that
\begin{equation}\label{prep}
\bar d(s_0)+1\sim C_0s_0^{-\frac{p-1}2} > s_0^{-\frac{p-1}2-\frac 12}\ge |\nu_0|),
\end{equation}
 the maximal solution of equation \eqref{eqw} is continuous in time with values in $H^1\times L^2(|y|<1)$, hence in $\H$.
\end{nb}

\medskip

In order to prove Proposition \ref{propw}, we proceed in two steps:\\
- in Section \ref{subseclin}, we recall from \cite{MZtams14} the properties of the  linearized operator of equation \eqref{eqw} around the soliton $\kappa(\pm d e_i,y)$, given $d\in (-1,1)$ and $i=1,2$. This will be crucial for our analysis, since it happens that the solitons in our goal \eqref{cprofile0} will appear to be decoupled in the proof;\\
- in Section \ref{subsecinit}, we 
give a modulation technique adapted to the properties of the linearized operator studied in the previous section and conclude the proof of Proposition \ref{propw}. 

\bigskip

 We use a classical strategy of construction of solutions to PDEs with some prescribed behavior. In particular, our strategy relies on two steps: 

\medskip

- Thanks to a dynamical system formulation, we show that controlling
$w(y,s)$
around the expected behavior \eqref{cprofile0} reduces to the control of 2 unstable directions.
This dynamical system formulation is essentially the same formulation we used in \cite{CZcpam13} for multi-solutions in one space dimension, and also to show that all characteristic points are isolated in \cite{MZdmj12}. 

\medskip

- Then, we solve the finite dimensional problem thanks to a topological argument based on index theory.

%%%%%%%%%%%%%%%%%%%%%%%%%%%%%%%%%%%%
%%%%%%%%%%%%%%%%%%%%%%%%%%%%%%%%%%%%
\subsubsection{The linearized operator of equation \eqref{eqw} around $\kappa(\pm de_i)$}\label{subseclin}
%%%%%%%%%%%%%%%%%%%%%%%%%%%%%%%%%%%%
%%%%%%%%%%%%%%%%%%%%%%%%%%%%%%%%%%%%
In this section, we recall from Section 2.1 in \cite{MZtams14} spectral properties of $L_{\pm de_i}$,  the linearized operator of equation \eqref{eqw} around the stationary solution $\kappa(\pm de_i)$. From \eqref{eqw}, we see that 
\begin{equation}\label{defld}
\begin{array}{rcl}
L_{\pm de_i}\vc{q_1}{q_2}&=&\vc{q_2}{\q L q_1+\psi(\pm de_i,y)q_1-\frac{p+3}{p-1}q_2-2y.\nabla q_2},\\
\\
\mbox{where }\psi(\pm de_i,y)&=&p\kappa(\pm de_i,y)^{p-1}-\frac{2(p+1)}{(p-1)^2}=\frac{2(p+1)}{(p-1)^2}\left(p \frac{(1-d^2)}{(1\pm d y_i)^2}-1\right).
\end{array}
\end{equation}
From the invariance of equation \eqref{equ} under some symmetries (time translation together with the Lorentz transformation and rotations), it is easy to see 
%It happens 
that $\lambda=1$ and $\lambda=0$ are eigenvalues of the linear operator $L_{\pm de_i}$ with multiplicity $1$ and $2$ respectively, and the corresponding eigenfunctions are:\\
- for $\lambda=1$, coming from time translation invariance of equation \eqref{equ},\;\;
\begin{equation}\label{deffld}
F_0(\pm de_i,y)=(1-d^2)^{\frac p{p-1}}\vc{(1\pm dy_i)^{-\frac {p+1}{p-1}}}{(1\pm d y_i)^{-\frac {p+1}{p-1}}};
\end{equation}
- for $\lambda=0$, coming from the invariance of equation \eqref{equ} under the Lorentz transformation and rotations:
\begin{equation}\label{deffid}
\begin{cases}
F_1(\pm de_i,y)= (1-d^2)^{\frac 1{p-1}}\vc{\ds\frac{y_i\pm d}{(1\pm d y_i)^{\frac {p+1}{p-1}}}}{0},\\
F_2(\pm de_i,y)= (1-d^2)^{\frac {p+1}{2(p-1)}}\vc{\ds\frac{y_{3-i}}{(1\pm d y_i)^{\frac {p+1}{p-1}}}}{0}.
\end{cases}
\end{equation}
In order to compute the projectors on the eigenfunctions of $L_{\pm de_i}$, we consider its conjugate with respect to the natural inner product of $\H$ defined by $\phi(q,r)=$
\begin{align}
\phi\left(\vc{q_1}{q_2}, \vc{r_1}{r_2}\right)&= \int_{-1}^1 \left(q_1r_1+\nabla q_1\cdot\nabla r_1- (y\cdot\nabla q_1)(y\cdot\nabla r_1)+q_2r_2\right)\rho dy \label{defphi}\\
&= \int_{-1}^1 \left(q_1\left(-\q L r_1+r_1\right) +q_2 r_2\right)\rho dy. \label{defphi2}
\end{align}
Omitting the computation of $L_{\pm de_i}^*$, we simply note that naturally, $\lambda=1$ and $\lambda=0$ are also eigenvalues of $L_{\pm de_i}^*$, in the sense that
%%%%%%%%%%%%%%%%%%%%%%%%%%%%%%%%%%%%
%%%%%%%%%%%%%%%%%%%%%%%%%%%%%%%%%%%%
for $l=0,1,2$, there exists $W_l(\pm de_i)\in \q H$ such that $L_{\pm de_i}^* W_l(\pm de_i)=\lambda_l W_l(\pm de_i)$, with $\lambda_0=1$, $\lambda_1=\lambda_2=0$,
\begin{equation}\label{defWl2}
W_{l,2}(\pm de_i,y)=c_{\lambda_l}\left(\frac{1-|y|^2}{1-d^2}\right)^{\lambda_l}F_{l,1}(\pm d e_i,y),
\end{equation}
for some positive constants $c_1$ and $c_0$, and $W_{l,1}(\pm de_i)$ is uniquely determined as the solution $v_1$ of
\begin{equation}\label{eqWl1}
-\q L v_1 + v_1 = \left(\lambda_l - \frac{p+3}{p-1}\right)v_2 - 2 y\cdot \nabla v_2+ 4\alpha \frac{v_2}{1-|y|^2}
\end{equation}
with $v_2= W_{l,2}(\pm de_i)$.\\
Of course, we have the orthogonality relation: $\phi(F_l,W_k)=\delta_{l,k}$ for $l,k=0,1,2$.
%%%%%%%%%%%%%%%%%%%%%%%%%%%%%%%%%%%%
%%%%%%%%%%%%%%%%%%%%%%%%%%%%%%%%%%%%
\\
This gives us the following expression for the projector $\pp_l(\pm d e_i)$ on the eigenfunction $F_l(\pm d e_i)$ defined for all $r\in \H$ by
\begin{equation}\label{defpdi} 
\pp_l(\pm d e_i,r) =\phi\left(W_l(\pm d e_i), r\right).
\end{equation}

%%%%%%%%%%%%%%%%%%%%%%%%%%%%%%%%%%%%
%%%%%%%%%%%%%%%%%%%%%%%%%%%%%%%%%%%%
\subsubsection{
A modulation technique}\label{subsecinit}
%%%%%%%%%%%%%%%%%%%%%%%%%%%%%%%%%%%%
%%%%%%%%%%%%%%%%%%%%%%%%%%%%%%%%%%%%

As one can see from the expression of initial data given in \eqref{w0}, at the initial time $s=s_0$, $w(y,s_0)$ is a pure sum of generalized solitons. From the continuity of the flow associated with equation \eqref{eqw} in $\H$ (see the remark following Proposition \ref{propw}), $w(y,s)$ will stay close to a sum of solitons, at least for a short time after $s_0$. In fact, we can do better, and impose some orthogonality conditions, killing the zero and expanding directions of the linearized operator of equation \eqref{eqw} around the sum of solitons. This is possible since the corresponding eigenvalues come from the invariance of equation \eqref{equ} under some symmetries, as we have just mentioned in section \ref{subseclin}. More precisely, the following modulation technique from our earlier work in \cite{MZtams14} (one soliton in higher dimensions) 
and \cite{MZdmj12} (multi-solitons in one dimension) 
is crucial for that.
%%%%%%%%%%%%%%%%%%%%%%%%%%%%%%%%%%%
%%%%%%%%%%%%%%%%%%%%%%%%%%%%%%%%%%%
\begin{lem}[A modulation technique for a function with symmetries, near 4 solitons]\label{lemode0}For all $A\ge 2$, there exist $E_1(A)$, $C_1(A)>0$ and $\eun(A)>0$ such that if 
$v\in \H$ is symmetric with respect to the axes and antisymmetric with respect to the bisectrices,
$\hat d\in (-1,1)$ and $\hat \nu\in (-1,\infty)$ satisfy
\begin{equation}\label{condnu}
-1+\frac 1A \le \frac{\hat \nu}{1-|\hat d|}\le A,\;
-\arg\tanh \hat d \ge E_1
\mbox{ and }
\|\hat q \|_{\H}\le \eun,
\end{equation}
where $\hat q=v-
\ds\sum_\ind
%\sum_{i=1,2,\theta=\pm 1}
 (-1)^{i+1}\kappa^*(\theta\hat d e_i,\hat \nu)$,
then, there exist $d\in(-1,1)$ and $\nu\in (-1,\infty)$ (both continuous as a function of $v$) such that
\begin{equation}\label{ortho}
\forall i=1,2,\;\theta=\pm1,\; l=0,1,2,\;\pp_l(\theta d^* e_i, q)=0,
\end{equation}
where $d^*=\frac d{1+\nu}$, $q=v-
\ds\sum_\ind
%\sum_{i=1,2,\theta=\pm 1}
 (-1)^{i+1}\kappa^*(\theta de_i,\nu)$  and the projector $\pp_l(\pm d^*e_i)$ is defined in \eqref{defpdi}. Moreover, we have 
\begin{equation}\label{proximite}
\|q \|_{\q H} 
+|\arg\tanh |d| - \arg\tanh |\hat d||+\left|\frac\nu{1-|d|} - \frac{\hat \nu}{1-|\hat d|}\right|
\le  C_1\|\hat q \|_{\q H}.
\end{equation}
\end{lem}
%%%%%%%%%%%%%%%%%%%%%%%%%%%%%%%%%%%
%%%%%%%%%%%%%%%%%%%%%%%%%%%%%%%%%%%
\begin{proof}
From the symmetries of $v$, it is enough to focus on the condition \eqref{ortho} only when $i=1$ and $\theta=-1$. 
Furthermore, we may also focus only on the cases $l=0$ and $l=1$, since condition \eqref{ortho} is always satisfied,
again from the symmetries of $v$. 
This way, the usual implicit function theorem one needs in modulation techniques reduces to the case of only one soliton, just in one space dimension. This latter case is of course simpler than the case of multi-solitons in one space dimension, already treated in Proposition 3.1 in \cite{MZdmj12}. For that reason, we omit the proof and refer the reader to our earlier work, in particular  \cite{MZdmj12} and \cite{MZtams14}.
\end{proof}

Let us apply this proposition with $v=(w(y,s_0), \partial_s w(y,s_0))$ \eqref{w0}, $\hat d=\bar d(s_0)$ and $\hat \nu=\nu_0$. Clearly, we have $\hat q=0$. Then, from \eqref{w0}, \eqref{solpart} and straightforward calculations, we see that
\[
\frac{|\hat \nu|}{1-|\hat d|}\le \frac C{\sqrt{s_0}} \quad \text{and} \quad 
-\arg\tanh \hat d\ge \frac{p-1}5 \log s_0
\]
for $s_0$ large enough. Therefore, Lemma \ref{lemode0} applies with $A=2$ and from the continuity of the flow associated with equation \eqref{eqw} in $\H$ (see the remark following Proposition \ref{propw}), we have a maximal $\bar s=\bar s(s_0,\nu_0)\in (s_0, +\infty]$ such that $w$ exists for all time $s\in [s_0, \bar s)$ and $w$ can be modulated in the sense that
\begin{equation}\label{defq}
\vc{w(y,s)}{\partial_s w(y,s)} = \sum_\ind (-1)^{i+1} \kappa^*(\theta d(s)e_i, \nu(s))+q(y,s)
\end{equation}
where the parameters $d(s)$ and $\nu(s)$ are such that for all $s\in[s_0, \bar s)$,
\begin{equation}\label{mod}
\forall i=1,2,\;
l=0,1,2,\;\pp_l(\pm d^*(s) e_i, q(s))=0,
\end{equation}
and
\begin{equation}\label{conmod}
\frac{|\nu(s)|}{1-|d(s)|}\le s_0^{-1/4},\;\;
\zeta(s) \ge \frac{(p-1)}8 \log s_0
\quad \text{and} \quad \|q(s)\|_{\H}\le \frac 1{\sqrt{s_0}},
\end{equation}
with $d^*(s) = \frac{d(s)}{1+\nu(s)}$ and $\zeta(s) = -\arg\tanh d(s)$.\\
Moreover, we have
\begin{equation}\label{paraminit}
q(s_0)=0,\;\;d(s_0)=\bar d(s_0)\mbox{ and }\nu(s_0) = \nu_0.
\end{equation}
From \eqref{defq}, we see that controlling the solution $(w(s), \partial_s w(s))$ in $\q H$ is equivalent to controlling $q(s)$ in $\q H$, $d(s)\in (-1,1)$ and $\nu(s)\in \m R$.
%%%%%%%%%%%%%%%%%%%%%%%%%%%%%%%%%%%
%%%%%%%%%%%%%%%%%%%%%%%%%%%%%%%%%%%
\subsubsection{Dynamics near the 4 solitons}
%%%%%%%%%%%%%%%%%%%%%%%%%%%%%%%%%%%
%%%%%%%%%%%%%%%%%%%%%%%%%%%%%%%%%%%
Since \eqref{conmod} holds for all $s\in [s_0, \bar s)$, we project equation \eqref{eqw} (more precisely, its first order version) to obtain equations satisfied by the components $q(s) \in\q H$, $d(s)\in (-1,1)$ and $\nu(s)\in \m R$.  Introducing
\begin{equation}\label{defpb}
 \bar p =
\begin{cases}
p & \text{ if } p <2, \\
2 - 1/100 & \text{ if } p =2, \\
2 & \text{ if } p >2,
\end{cases}
\end{equation}
we claim the following:
%%%%%%%%%%%%%%%%%%%%%%%%%%%%%%%%%%%%
%%%%%%%%%%%%%%%%%%%%%%%%%%%%%%%%%%%%
\begin{lem}[Dynamics of the parameters]\label{propdyn}
There exist $\barc=\barc(p)>0$ and $\delta>0$
such that for $s_0$ large enough and for all $s\in[s_0,\bar s)$, we have 
\begin{align}
 \frac{|\nu'-\nu|}{1-|d|^2}&\le  C\|q\|_{\q H}^2+Ce^{-\frac 4{p-1}\zeta}+C\frac{|\nu|}{1-|d|^2}\|q\|_{\q H},\label{est:nu}\\
\left|\zeta' -\barc e^{-\frac 4{p-1}\zeta}\right|&\le
C\|q\|_{\q H}^2+C \frac{|\nu|}{1-|d|^2} \left(e^{-\frac 4{p-1}\zeta}+\|q\|_{\q H}\right) 
+C e^{-\frac 4{p-1}\left(1+\delta\right)\zeta},\label{est:zeta} \\
\| q(s) \|_{\q H}^2 & \le C e^{-\delta(s-s_0)} \| q(s_0) \|_{\q H}^2 +C e^{-\frac{4\bar p \zeta}{p-1}},\label{est:q}
\end{align}
where $\zeta(s) = - \arg\tanh d(s)$.
% and $\barc(p)>0$.
\end{lem}
%%%%%%%%%%%%%%%%%%%%%%%%%%%%%%%%%%%%
%%%%%%%%%%%%%%%%%%%%%%%%%%%%%%%%%%%%
\begin{proof} The proof is the two-dimensional version of our earlier work on multi-solitons in one space dimension, in \cite{MZajm12}, \cite{MZdmj12} and \cite{CZcpam13}. Most of the estimates are the same as in one space dimension, some others are truely two-dimensional, hence more delicate, without involving new ideas. For that reason, we leave the proof to Appendix \ref{appdyn}.
\end{proof}

\bigskip

From the decomposition \eqref{defq}, our goal in Theorem \ref{mainth} will be achieved if we construct a solution such that $q\to 0$, $\nu\to 0$ and $\zeta(s) \sim \bar \zeta(s)$ defined in \eqref{solpart}. For that reason, it is natural to make the following change of variables defined by
\begin{equation}\label{defxi0}
\xi(s) = \frac 4{p-1}(\zeta(s) -\bar \zeta(s)).
\end{equation}
Therefore, we reduce the control of $w$ to the control of $(q,\xi,\nu)$ near zero. Assuming that 
\begin{equation}\label{xi1}
|\xi(s)|\le 1,
\end{equation}
in addition to \eqref{conmod}, we derive the following corollary from Lemma \ref{propdyn}:
%%%%%%%%%%%%%%%%%%%%%%%%%%%%%%%%%%%%%%
%%%%%%%%%%%%%%%%%%%%%%%%%%%%%%%%%%%%%%
\begin{cor}\label{cordyn} For $s_0$ large enough and for all $s\in [s_0, \bar s)$, if \eqref{xi1} holds, then we have
\begin{align*}
|\nu'-\nu|&\le Cs^{-\frac{p-1}2}(\|q\|_{\q H}^2 +s^{-1})+C|\nu|\|q\|_{\q H},\\
|\xi'+s^{-1}\xi|&\le Cs^{-1} \xi^2+C\|q\|_{\q H}^2 +C|\nu|s^{\frac{p-1}2}(\|q\|_{\q H}+s^{-1})+Cs^{-1-\delta},\\
\| q(s) \|_{\q H}^2 & \le C e^{-\delta(s-s_0)} \| q(s_0) \|_{\q H}^2 +Cs^{-\bar p}.
\end{align*}
\end{cor}
%%%%%%%%%%%%%%%%%%%%%%%%%%%%%%%%%%%%%%
%%%%%%%%%%%%%%%%%%%%%%%%%%%%%%%%%%%%%%
\begin{proof} The proof is straightforward, if one uses the following trivial estimate : $\frac 1C \le e^{2\zeta}(1-|d|)\le C$.
\end{proof}
From this corollary, we see that the control of $\|q\|$ and $\xi$ near zero is natural, since the linear parts of the differential inequalities they satisfy show a decreasing property. However, the situation may seem hopeless for $\nu$, which shows an expanding behavior. In fact, bearing in mind that a simple application of the intermediate-value theorem gives a solution going to zero at infinity for the following model equation:
\[
z'(s)=z(s)+\frac 1s,\;\;z(s_0)=z_0,
\]
for some well-chosen $z_0$, 
the same intermediate-value theorem  will allow us to show the existence of a solution $(q,\xi,\nu)$ to the system given in the previous corollary, for a well-chosen parameter $\nu_0 \in I_0$ defined in \eqref{w0}.

\medskip

In order to formalize the argument, we introduce the following shrinking set, that will allow the control of the set of parameters $(q,\xi,\nu)$ towards zero as $s\to \infty$:
%%%%%%%%%%%%%%%%%%%%%%%%%%%%%%%%%%%
%%%%%%%%%%%%%%%%%%%%%%%%%%%%%%%%%%%
\begin{defi}[A shrinking set to zero]\label{defVa} For any $s\ge s_0$, we defined $\q V(s)$ as the unit ball of $\q H \times \m R^2$ equipped with the norm
\begin{equation}\label{defN}
N(q,\xi,\nu) = \max \left\{s^{\frac 12 +\eta}\|q\|_{\q H},\;s^\eta|\xi|,\;s^{\frac 12 +\frac{p-1}2}|\nu|\right\},
\end{equation}
where
\begin{equation}\label{defv}
\eta = \frac 14\min\left(1,\delta, \frac{\bar p-1}2\right),
\end{equation}
$\delta$ is introduced in Lemma \ref{propdyn} and $\bar p$ defined in \eqref{defpb}.
\end{defi}
%%%%%%%%%%%%%%%%%%%%%%%%%%%%%%%%%%%
%%%%%%%%%%%%%%%%%%%%%%%%%%%%%%%%%%%
From \eqref{paraminit}, \eqref{w0} and the definition \eqref{defxi0} of $\xi$, we derive the existence of\\
 $s^*(s_0, \nu_0) \in [s_0,\bar s)$ (remember that $\bar s$ is defined just before \eqref{defq}), such that for all $s\in [s_0, s^*)$, $(q,\xi,\nu)(s)\in \q V(s)$ and:\\
- either $s^* = +\infty$, or\\
- $s^*<+\infty$ and from continuity, $(q,\xi,\nu)(s^*)\in \partial\q V(s^*)$, in the sense that one of the 3 quantities defining the maximum in \eqref{defN} is equal to $1$.

\medskip

With Lemma \ref{propdyn} at hand, we are in a position to prove the following, which directly implies Proposition \ref{propw}:
%%%%%%%%%%%%%%%%%%%%%%%%%%%%%%%%%%%%
%%%%%%%%%%%%%%%%%%%%%%%%%%%%%%%%%%%%
\begin{lem}[A solution such that $(q,\xi,\nu)(s) \in \q V(s)$] \label{reducw} For $s_0$ large enough, there exists $\nu_0\in I_0$ (defined in \eqref{w0})
such that equation \eqref{eqw} with initial data (at $s=s_0$) given by \eqref{w0} is defined for all $(y,s)\in B(0,1) \times [s_0,\infty)$ and satisfies $(q,\xi,\nu)(s) \in \q V(s)$ for all $s\ge s_0$ (or equivalently, $s^*(s_0,\nu_0)=+\infty$).
\end{lem}
%%%%%%%%%%%%%%%%%%%%%%%%%%%%%%%%%%%%
%%%%%%%%%%%%%%%%%%%%%%%%%%%%%%%%%%%%
\begin{proof} 

In fact, we started the proof of this lemma right after the statement of Lemma \ref{lemode0}.
 For the sake of clarity, we summarize here all the previous arguments, and conclude the proof thanks to a topological argument.

Let $s_0$ be large enough.
For all $\nu_0\in I_0$ defined in \eqref{w0},
 we consider the solution  $w(s_0,\nu_0,y,s)$ (or $w(y,s)$ for short) to equation \eqref{eqw}, with initial condition at time $s_0$ given by \eqref{w0}.\\
As we showed after the statement of Lemma \ref{lemode0}, $w(y,s)$ can be modulated (up to some time $\bar s = \bar s(s_0,\nu_0)>s_0$) into a triplet $(q(s), d(s),\nu(s))$. 
Performing the change of variables \eqref{defxi0},
 we reduce the control of $w(s)$ to the control of $(q(s), \xi(s), \nu(s))$ and we see from \eqref{paraminit} that 
\begin{equation}\label{initmod2}
q(s_0)=0,\;\;\xi(s_0)=0\mbox{ and }\nu(s_0)=\nu_0.
\end{equation}
As stated in Lemma \ref{reducw}, our aim is to find some $\nu_0\in I_0$ so that the associated $w \in \q C([s_0,\infty), \q H)$ is globally defined for forward times and satisfies
\[ 
\forall s \ge s_0, \quad N(q(s),\xi(s),\nu(s),s) \le 1, \quad \text{i.e.} \quad (q,\xi,\nu)(s) \in \q V(s). 
\] 
We argue by contradiction and assume that for all $\nu_0\in I_0$,
the exit time $s^*(s_0,\nu_0)$
is finite, where 
\begin{equation} \label{exit}
s^*(s_0,\nu_0)= \sup \{ s \ge s_0 \;|\; \forall \tau \in [s_0,s], \ N(q(\tau), \xi(\tau), \nu(\tau), \tau) \le 1 \}.
\end{equation}
Then by continuity, notice that 
\begin{equation} \label{Nbord}
N(q(s^*), \xi(s^*), \nu(s^*), s^*) =1\mbox{ with }s^*=s^*(s_0,\nu_0),
\end{equation}
and that the supremum defining $s^*(s_0,\boldsymbol \nu)$ is in fact a maximum.\\
We now consider the (rescaled) flow for $\nu(s)$, that is
\begin{equation}\label{defPhi}
\Phi: (s,\nu_0) \mapsto s^{\frac 12 +\frac{p-1}2} \nu(s).
\end{equation}
By the properties of the flow, $\Phi$ is a continuous function of $(s,\nu_0) \in [s_0, s^*(s_0,\nu_0)] \times  I_0$. 
By definition of the exit time $s^*(s_0,\nu_0)$, we have that for all $s \in [s_0, s^*(s_0,\nu_0)]$, $|\Phi(s,\nu_0)|\le 1$.
The following claim allows us to conclude:

\medskip

%%%%%%%%%%%%%%%%%%%%%%%%%%%%%%%%%%%%%%%%%%%%%
%%%%%%%%%%%%%%%%%%%%%%%%%%%%%%%%%%%%%%%%%%%%%
\noindent\emph{Claim.}
For $s_0$ large enough, we have:\\ 
(i) {\bf (Reduction to one dimension)}: For all $\nu_0\in I_0$, $|\Phi(s^*(s_0,\nu_0), \nu_0)|=1$.\\
(ii) {\bf (Transverse crossing)}: The flow $s \mapsto \Phi(s,\nu_0)$ is transverse outgoing at $s=s^*(s_0, \nu_0)$ (when it hits $\pm 1$).\\ 
  (iii) {\bf (Initialization)}: If $\nu_0\in \partial I_0$, that is $\nu_0=\theta_0 s_0^{-\frac 12-\frac{p-1}2}$ with $\theta_0=\pm 1$, then $s^*(s_0,\nu_0)=s_0$ and $\Phi(s^*(s_0,\nu_0), \nu_0)=\theta_0$.
%%%%%%%%%%%%%%%%%%%%%%%%%%%%%%%%%%%%%%%%%%%%%
%%%%%%%%%%%%%%%%%%%%%%%%%%%%%%%%%%%%%%%%%%%%%

\medskip

Indeed, from item (ii), $\nu_0 \to s^*(s_0,\nu_0)$ is continuous, hence from items (i) and (iii),
\[
\nu_0 \mapsto \Phi(s^*(s_0,\nu_0), \nu_0)
\]
is a continuous map from $I_0$ to $\{-1,1\}$, and the images $1$ and $-1$ are attained when $\nu_0 \in \partial I_0$. From the intermediate value theorem, this is a contradiction. Thus, there exists $\nu_0 \in I_0$ such that for all $s\ge s_0$, $N(s_0, \nu_0)\le 1$, hence $(q,\xi,\nu)(s_0, \nu_0, \cdot,s) \in \q V(s)$ which is the desired conclusion of Lemma \ref{reducw}. It remains to prove the Claim in order to conclude.
\begin{nb}
Note that we use item (ii) of the Claim either with $s=s^*$, in order to prove the continuity of the exit time, or with $\nu_0$ on the boundary of $I_0$ and $s=s_0$ to show item (iii) of the same Claim.
\end{nb}

\bigskip

\begin{proof}[Proof of the Claim]
 In the following, the constant $C$ stands for $C(s_0)$.\\
(i) Consider $s\in [s_0,s^*(s_0,\nu_0)]$. Since $N(q(s),\xi(s),\nu(s),s)\le 1$ by definition of the exit time $s^*(s_0, \nu_0)$, it follows that estimate \eqref{xi1} is satisfied provided that $s_0$ is large enough, hence, Corollary \ref{cordyn} applies. Using the initial condition \eqref{initmod2} and the definition \eqref{defv} of $\eta$, then taking $s_0$ large enough, we see that
\[
\|q(s)\|_{\q H}\le  C{s^{-\frac{\bar p}2}} \le \frac {s^{-\frac 12 -\eta}}2,
\]
hence from this,
\begin{equation}\label{est:nu2}
|\nu'-\nu|\le  Cs^{-\frac{p-1}2-1}\mbox{ and }
|\xi'+s^{-1}\xi|\le C s^{-1-2\eta}+Cs^{-\frac{1+\bar p}2}+Cs^{-\frac 32}+Cs^{-1-\delta}\le  C s^{-1-2\eta}.
\end{equation}
Integrating the last inequality on the interval $[s_0,s]$ and using again \eqref{initmod2}, we see that
\[
|\xi(s)| \le C{s^{-2\eta}}\le \frac {s^{-\eta}}2,
\]
for $s_0$ large enough.
Since $N(q(s^*),\xi(s^*),\nu(s^*),s^*)=1$ by \eqref{Nbord}, 
we see from the definition \eqref{defN} of $N$ that necessarily
\begin{equation*}%\label{nubord}
s^{1/2+\frac{p-1}2}|\nu(s)|=1.
\end{equation*}
Using the definition \eqref{defPhi} of the flow $\Phi$, we get to the conclusion of item (i). % of the Claim.

\medskip
 
 \noindent (ii) Assume that $\Phi(s,\nu_0)=s^{\frac 12+\frac{p-1}2}\nu(s)=\theta_0$ for some $s\in [s_0,s^*(s_0, \nu_0)]$ and $\theta_0=\pm 1$. 
Using \eqref{est:nu2} (which holds here), we write 
\begin{align*}
\MoveEqLeft 
\frac{d}{ds} s^{1/2+\frac{p-1}2} \nu(s) 
 = {s}^{1/2 + \frac{p-1}2} \left( \left( \frac{1}{2} + \frac{p-1}2 \right) \frac{\nu(s)}{s} + \nu'(s) \right) \\
& =  {s}^{1/2 + \frac{p-1}2} \left(  \nu(s) \left( 1 + \frac{1}{2s} + \frac{p-1}{2s} \right) + O \left( \frac{1}{{s}^{1+\frac{p-1}2}} \right) \right) 
=\theta_0+O\left(\frac 1{\sqrt s}\right).
\end{align*}
Hence, for $s_0$ large enough, $\frac{d}{ds} s^{1/2+\frac{p-1}2} \nu(s) $ has the sign of $\theta_0$, which shows that the flow is transverse outgoing and item (ii) holds.

\medskip

\noindent (iii) 
Let $\nu_0\in \partial I_0$, that is $\nu_0 = \theta_0 s_0^{-\frac 12 -\frac{p-1}2}$ for some $\theta_0=\pm 1$. From \eqref{paraminit} and the definition \eqref{defPhi} of the flow $\Phi$, we see that 
\begin{equation}\label{ini}
\nu(s_0)=\nu_0\mbox{ and }\Phi(s_0,\nu_0) = \nu_0s_0^{\frac 12 +\frac{p-1}2}=\theta_0,
\end{equation}
hence, the flow hits $\pm 1$ at $s=s_0$, and item (ii) applies, showing that the flow $\Phi$ is transverse outgoing. By the definition of the exit time, we see that 
$s^*(s_0, \nu_0) = s_0$.
Using \eqref{ini}, we get to the conclusion of item (iii). % of the Claim. 
This concludes the proof of the Claim.
\end{proof}
Since Lemma \ref{reducw} follows from the Claim and the intermediate value theorem, this concludes the proof of Lemma \ref{reducw} too.
\end{proof}

\bigskip

It remains to give the proof of Proposition \ref{propw} in order to conclude this section. 

\begin{proof}[Proof of Proposition \ref{propw}]
If $s_0$ is large enough and $w(y,s)$ is the solution constructed in Lemma \ref{reducw} (with initial data at $s=s_0$ given by \eqref{w0}), then we write by definition \eqref{defq} of $q$:
\begin{equation}\label{triangle}
\left\|\vc{w(s)}{\ps w(s)} - 
\sum_\ind
(-1)^{i+1}\vc{\kappa(\theta \bar d(s) e_i,y)}{0}\right\|_{\H}\le \|q(s)\|_{\q H} + \sum_\ind A_{i,\theta}
\end{equation}
where
\[
A_{i,\theta} = \|\kappa^*(\theta d(s)e_i,\nu(s))-(\kappa(\theta \bar d(s)e_i),0)\|_{\q H}.
\]
Since $(q,\xi,\nu) (s) \in \q V(s)$ for all $s \ge s_0$, we see from Definition \ref{defVa} and  \eqref{defxi0} that  
\begin{equation}\label{boundxi}
|\xi(s)|=\frac 4{p-1}|\zeta - \bar \zeta(s)|\le Cs^{-\eta}. 
\end{equation}
In particular, by definition \eqref{solpart} of $\bar \zeta(s)$, we see that
\[
1-|d(s)|\sim Ce^{-2\zeta(s)}\sim Ce^{-2\bar \zeta(s)}\sim C s^{-\frac{p-1}2}\mbox{ as }s\to \infty,
\]
hence, again from definition \ref{defVa} of $\q V(s)$, we have
\[
\forall s\ge s_0,\;\; \frac{|\nu(s)|}{1-|d(s)|}\le Cs^{-\frac 12}.
\]
Therefore, item (iv) of Lemma \ref{lemkd} applies and since $\kappa^*(\theta \bar d(s)e_i,0,y)=(\kappa(\theta \bar d(s)e_i,y),0)$ by definitions \eqref{defkd} and \eqref{defk*}, we write 
\[
A_{i,\theta} \le C \frac{|\nu(s)|}{1-|d(s)|}+C|\zeta(s) - \bar \zeta(s)|\le Cs^{-\eta}.
\]
As $\|q(s)\|_{\q H}\le  C{s^{-\frac 12 -\eta}}$, still by definition \ref{defVa} of $\q V(s)$, the conclusion of Proposition \ref{propw} follows from \eqref{triangle}.
\end{proof}

%%%%%%%%%%%%%%%%%%%%%%%%%%%%%%%%%%%%
%%%%%%%%%%%%%%%%%%%%%%%%%%%%%%%%%%%%
\subsection{Derivation of a blow-up solution to equation \eqref{equ}}\label{subprop0}
%%%%%%%%%%%%%%%%%%%%%%%%%%%%%%%%%%%%
%%%%%%%%%%%%%%%%%%%%%%%%%%%%%%%%%%%%
Using the previous section, we derive a blow-up solution to the Cauchy problem of equation \eqref{equ}. This is the aim of this section:
%%%%%%%%%%%%%%%%%%%%%%%%%%%%%%%%%%%%
%%%%%%%%%%%%%%%%%%%%%%%%%%%%%%%%%%%%
\begin{prop}[Existence of a blow-up solution for equation \eqref{equ} and first estimates on the blow-up time]\label{propyr}$ $\\
(i) For $s_0$ large enough, there exists $u(x,t)$ a solution to the Cauchy problem (at $t=-1$) of equation \eqref{equ} (with the similarity variables' version $w_0$ given by \eqref{w0}), symmetric with respect to the axes and anti-symmetric with respect to the bisectrices, which blows up on a 
surface $\Gamma = \{(x,T(x))\}$
satisfying $T(0)=0$. Moreover, $w_{0}$ shows 4 solitons as in \eqref{cprofile0}.\\
(ii) For any $\delta_0\in(0,1)$, there exists $\bar s_0$ such that for all $s_0\ge \bar s_0$, we have
\begin{equation*}%\label{but-1}
T(x) \ge -\frac{(1-\bar d(s_0))}2x_1, 
\end{equation*}
whenever
\begin{equation}\label{portion0}
x\in \q R, \;\;0\le x_2\le(1-\delta_0) x_1\mbox{ and }\delta_0\le x_1\le 1+\frac{(\lambda_0-1)^2}{100},
\end{equation}
where 
\begin{equation}\label{defl0}
\lambda_0\equiv-\frac{1+\nu_0}{\bar d(s_0)}>1, 
\end{equation}
 $\bar d(s_0)<0$ is defined in \eqref{solpart}, and $\nu_0$ is introduced in Proposition \ref{propw} satisfying \eqref{w0}.\\
By symmetry, a similar statement holds for all non-characteristic points outside a small neighborhood of the origin and the bisectrices. 
\end{prop}
%%%%%%%%%%%%%%%%%%%%%%%%%%%%%%%%%%%%
%%%%%%%%%%%%%%%%%%%%%%%%%%%%%%%%%%%%
\begin{nb}
In the previous statement, $u$ and $T$ depend on the choice of $s_0$, but we omit that dependence to simplify the notation.
The fact that $\lambda_0>1$ follows from \eqref{prep}, provided that $s_0$ is large enough.
\end{nb}
\begin{proof}$ $\\
(i) Consider $s_0$ large enough, $w(y,s)$ the solution of \eqref{eqw} constructed in Proposition \ref{propw} with initial data (at time $s=s_0$) given by \eqref{w0}, 
and the chosen parameter $|\nu_0|\le s_0^{-\frac 12 - \frac{p-1}2}$. Through the similarity variables' transformation \eqref{defw}, this yields a solution to the Cauchy problem (say, at time $t=-1$), for equation \eqref{equ} defined only in the backward light cone with vertex $(0,0)$ and basis at $t=-1$. 

\medskip

In order to get a solution defined for all $x\in\m R^2$, at least for small $t+1\ge 0$, we need first to extend the definition of $w(y,s_0)$ \eqref{w0} to all $y\in\m R^2$. Note that the expression we choose for $w(y,s_0)$ with the 4 solitons 
 is singular whenever $|y_i|=\lambda_0>1$ defined in \eqref{defl0}.
For that reason, we introduce the following non-increasing truncation $\chi_0\in C^\infty(\m R)$ such that
\[
\chi_0(\xi)=1
\mbox{ if }\xi<\lambda_0-(\lambda_0-1)^2
\mbox{ and }\chi_0(\xi)=0\mbox{ if }\xi>\lambda_0-\frac{(\lambda_0-1)^2}2
\]
(note that 
\begin{equation}\label{l1}
\lambda_0-(\lambda_0-1)^2 >1
\end{equation}
 from \eqref{defl0}, \eqref{w0} and \eqref{devdbar}, provided that $s_0$ is large enough).
 This way, we may extend the definition of $w(y,s_0)$ by introducing new initial data for equation \eqref{eqw} defined by
\begin{equation}\label{bw0}
\vc{\bar w(y,s_0)}{\partial_s \bar w(y,s_0)} =
\sum_\ind
(-1)^{i+1}\kappa^*(\theta \bar d(s_0) e_i,\nu_0,y)\chi_0(\theta y_i).
\end{equation}
From the choice of the truncation $\chi_0$, we see that 
$(\bar w, \partial_s \bar w)(s_0) \in H^1\times L^2(\m R^2)$.
Therefore, we may define $u(x,t)$ as the solution of equation \eqref{equ} with initial data (at time $t=-1$) in 
$H^1\times L^2(\m R^2)$ 
given by
\begin{equation}\label{initu}
u(x,-1)=\bar w(x,s_0)\mbox{ and }
\partial_t u(x,-1)=\partial_s \bar w(x,s_0)+\frac 2{p-1}\bar w(x,s_0)+x\cdot \nabla \bar w(x,s_0).
\end{equation}
From the choice we made in \eqref{bw0}, it is clear that $(\bar w, \partial_s \bar w)(y,s_0)$ has the same symmetries as $(w, \partial_s w)(y,s_0)$ \eqref{w0}, and so does $(u,\partial_ t u)(x,-1)$, hence, $(u,\partial_ t u)(x,t)$ also for any $t\ge -1$ wherever the solution is defined. Thus, 
$u(x,t)$ will be symmetric with respect to the axes and anti-symmetric with respect to the bisectrices too. 

\medskip

Since we choose the initial time in the $u(x,t)$ to be $t=-1$, it follows by definition of the similarity variables' transformation \eqref{defw} that the initial time for $w_{0,0}$ is $s=-\log(0-(-1))=0$. Using \eqref{l1}, \eqref{initu} and \eqref{bw0}, we see that
\[
 (w_{0,0}, \partial_s w_{0,0})(y,0)=(\bar w, \partial_s \bar w)(y,s_0)=(w, \partial_s w)(y,s_0), \mbox{ for all }|y|<1.
\]
Therefore, 
using the uniqueness of the solution to the Cauchy problem for equation \eqref{eqw} (discussed in the remark following Proposition \ref{propw}), we see that
\begin{equation}\label{partout}
(w_{0,0}, \partial_s w_{0,0})(y,s)=(w, \partial_s w)(y,s+s_0),\mbox{ for all }|y|<1\mbox{ and }s\ge 0.
\end{equation}
Using back the similarity variables' transformation \eqref{defw}, this yields the value of $u(x,t)$ in the backward light cone with vertex $(0,0)$ with basis at $t=-1$, in the sense that for all $t\in [-1,0)$ and $|x|<-t$, 
\[
u(x,t) =(-t)^{-\frac 2{p-1}}w\left(\frac x{-t}, s_0-\log(-t)\right).
\] 
In particular, $u(x,t)$ blows up when $(x,t) = (0,0)$, which means that $T(0)=0$, hence $w_{0,0}=w_0$ from the convention just stated after \eqref{defw}. Therefore, from \eqref{partout} and Proposition \ref{propw}, we see that $w_0$ shows 4 solitons, as in \eqref{cprofile0}.

\bigskip

\noindent (ii) Consider $\delta_0>0$ and $x\in \q R$ satisfying \eqref{portion0}. Our idea is to choose the largest backward light cone with vertex $(x,\bar T(x))$ such that its section at $t=-1$ does not encounter the truncation visible in \eqref{bw0}. By definition \eqref{initu} of initial data $u(x,-1)$, we easily see that
\begin{equation}\label{deftb}
\bar T(x) =\lambda_0-1-(\lambda_0-1)^2-x_1.
\end{equation}
Then, we will see that at the section at time $t=-1$ of that cone, the initial data $(u, \partial_t u)(\xi,-1)$ is in fact close to the ``$u(\xi,t)$ version'' of the right-soliton, namely to $(\bar u_{1,1}, \partial_t \bar u_{1,1})(\xi,-1)$ where 
\begin{equation}\label{defbu}
\bar u_{j,\eta} (\xi,t) = \kappa_0\frac{(1-|\bar d(s_0)|^2)^{\frac 1{p-1}}}{(\nu_0-t+\eta\bar d(s_0) \xi_j)^{\frac 2{p-1}}},
\end{equation}
More precisely, it holds that
\begin{equation}\label{prochex}
\|(u, \partial_tu)(-1) -(\bar u_{1,1}, \partial_t\bar u_{1,1})(-1)\|_{H^1\times L^2(|\xi-x|<1+\bar T(x))}\le C\delta_0^{-\frac{2(p+1)}{p-1}}s_0^{-1/2}
\end{equation}
(see below for a proof). Roughly speaking, this is our argument: from the continuity with respect to initial data, together with the finite speed of propagation, if we restrict ourselves to the intersection of the above-mentioned cone with the domain of $\bar u_{1,1}$, we will see that $u(\xi,t)$ remains close to $\bar u_{1,1}(\xi,t)$.
As a matter of fact, by definition \eqref{defbu}, $\bar u_{1,1}(\xi,t)$ blows up on the plane $\{t=T^*(\xi)\}$ with
\[
T^*(\xi)=\nu_0+\bar d(s_0)\xi_1. 
\]
Since the slope of this plane is $\bar d(s_0)>-1$ and the slope of the cone is $-1$, this would mean that
\begin{equation}\label{tmin}
T(x)\ge \min(\bar T(x), T^*(x)),
\end{equation}
if one can estimate the effect of the error term in \eqref{prochex}.
Since we have
\begin{equation}\label{eq}
1+\bar d(s_0) \sim C_0 s_0^{-\frac{p-1}2}\sim\lambda_0-1\mbox{ as }s_0 \to \infty
\end{equation}
(see \eqref{devdbar}, \eqref{defl0} and \eqref{w0}), it is easy to see that
\begin{align*}
\bar T(x) +\frac{1-\bar d(s_0)}2 x_1 &= C_0s_0^{-\frac{p-1}2}(1-\frac{x_1}2)+o(s_0^{-\frac{p-1}2}),\\
T^*(x) +\frac{1-\bar d(s_0)}2 x_1 &= \frac{C_0x_1}2 s_0^{-\frac{p-1}2}+o(s_0^{-\frac{p-1}2}).
\end{align*}
Since $\delta_0\le x_1 \le 1+\frac{(\lambda_0-1)^2}{10}$, taking $s_0$ large enough yields
\begin{equation}\label{hier}
 \min(\bar T(x), T^*(x))+\frac{1-\bar d(s_0)}2 x_1>0,
\end{equation}
and the conclusion of item (ii) of Proposition \ref{prop1} would follow, if \eqref{tmin} holds.\\
Unfortunately, we are unable to handle the effect of the error term in \eqref{prochex}, for that reason, we proceed differently.\\
Assume by contradiction that 
\begin{equation}\label{co}
T(x) < - \frac{1-\bar d(s_0)}2 x_1.
\end{equation}
From \eqref{hier} (which holds true unlike our informal argument) and \eqref{prochex}, we derive the following estimate for the similarity version $w_x$ defined in \eqref{defw}:
\begin{equation}\label{proxx}
\left\|(\partial_s w_x,w_x)(-\log(1+ T(x))) -\kappa^*\left(\bar d(s_0),\bar \nu\right)\right\|_{\q H} \le C \delta_0^{-\frac{2(p+1)}{p-1}}s_0^{-1/2},
\end{equation}
with
\[
\bar \nu = \frac{1+\nu_0+\bar d(s_0) x_1}{1+T(x)}\ge 
\frac{1+\nu_0+\bar d(s_0) x_1}{1 - \frac{1-\bar d(s_0)}2 x_1}
\ge
\frac{1+\nu_0+\bar d(s_0) \delta_0}{1 - \frac{1-\bar d(s_0)}2 \delta_0 } \to 1-\delta_0>0\mbox{ as }s_0\to\infty,
\]
where we used \eqref{co}, \eqref{portion0} and \eqref{devdbar}.
Using item (iii) of Lemma \ref{lemkd}, we see that
\[
\left\|\kappa^*\left(\bar d(s_0),\bar \nu\right)\right\|_{\q H} \le C\max(\bar \mu^{\frac 1{p-1}}, \bar \mu^{\frac 2{p-1}}),
\]
where
\[
\bar \mu =\left(1+\frac{\bar \nu}{1-|\bar d(s_0)|}\right)^{-1}\le \frac{2C_0}{1-\delta_0}s_0^{-\frac{p-1}2}
\]
for $s_0$ large enough, 
again by \eqref{devdbar}. Using \eqref{proxx}, we see that
\[
\left\|(\partial_s w_x,w_x)(-\log(1+ T(x)))\right\|_{\q H} \le C(1-\delta_0)^{-\frac 1{p-1}}s_0^{-1/2} +C \delta_0^{-\frac{2(p+1)}{p-1}}s_0^{-1/2}.
\]
Taking $s_0$ large enough, we contradict the following non-degeneracy of the blow-up limit result we proved at non-characteristic points in \cite{MZtams14}:
Now, we recall the following non-degeneracy result from \cite{MZtams14}:
%%%%%%%%%%%%%%%%%%%%%%%%%%%%%%%%%%%%%
%%%%%%%%%%%%%%%%%%%%%%%%%%%%%%%%%%%%%
\begin{prop}[Non-degeneracy of the blow-up limit for $w_{\bar x}$ when $\bar x\in \RR$; Corollary 3.2 page 35 in \cite{MZtams14}]\label{propnondeg}
There exists $\bar \epsilon_1>0$ such for all $\bar x\in \RR$ and $s\ge - \log T(\bar x)$, we have $\|(w_{\bar x}(s), \partial_s w_{\bar x}(s))\|_{\q H}\ge \bar \epsilon_1$.
\end{prop}
%%%%%%%%%%%%%%%%%%%%%%%%%%%%%%%%%%%%%
%%%%%%%%%%%%%%%%%%%%%%%%%%%%%%%%%%%%%
\begin{nb}
The constant $\bar \epsilon_1$ is universal and does not depend on the solution $u(x,t)$.
\end{nb}
Of course, it remains for us to prove \eqref{prochex}.

\medskip

{\it Proof of \eqref{prochex}}: 
From \eqref{portion0} and the defintion \eqref{deftb} of $\bar T(x)$, we see that whenever $\xi \in B(x,1+\bar T(x))$, we have $|\xi_i|\le |x_i|+1+\bar T(x)\le x_1+1+\bar T(x)=\lambda_0-(\lambda_0-1)^2$, therefore, $\chi_0(\theta \xi) =1$, for any $i=1,2$ and $\theta =\pm 1$. In particular, identity \eqref{bw0} holds with $\chi_0(\theta y_i)$ replaced by $1$.
Using \eqref{initu}, we see that
\begin{align}
&\|(u, \partial_t u)(-1) - (\bar u_{1,1}, \partial_t \bar u_{1,1})(-1)\|_{H^1\times L^2(|\xi-x|<1+\bar T(x))}\nonumber\\
\le& \sum_{(j,\eta)\neq(1,1)}\|(\bar u_{j,\eta}, \partial_t \bar u_{j,\eta})(-1) \|_{H^1\times L^2(|\xi-x|<1+\bar T(x))}\nonumber\\
\le &C(1-|\bar d(s_0)|^2)^{\frac 1{p-1}}(1+\bar T(x)) \sum_{(j,\eta)\neq(1,1)}
[\min_{|\xi-x|<1+\bar T(x))}(1+\nu_0+\eta\bar d(s_0) \xi_j)]^{-\gamma_{j,\eta}},\label{inter0}
\end{align}
where $\bar u_{j,\eta}$ is defined in \eqref{defbu} and $\gamma_{j,\eta}$ is equal to $\frac 2{p-1}$ if the minimum is larger than $1$, and equal to $\frac{p+1}{p-1}$ otherwise. Using \eqref{portion0}, \eqref{deftb} and \eqref{eq}, we see that when $|\xi-x|<1+\bar T(x)$, we have
\begin{align*}
\xi_1 &\ge x_1-(1+\bar T(x))\ge -1 +2x_1 -(\lambda_0-1)+(\lambda_0-1)^2\ge  -1+2\delta_0+O(s_0^{-\frac{p-1}2}),\\
\xi_2 &\le x_2 +(1+\bar T(x)) \le (1-\delta_0)x_1+1-x_1+ (\lambda_0-1)-(\lambda_0-1)^2\le 1-\delta_0^2 +O(s_0^{-\frac{p-1}2})\\
\xi_2 &\ge x_2 - (1+\bar T(x)) \ge -1 -(\lambda_0-1)+(\lambda_0-1)^2+x_1 
\ge -1+\delta_0 +O(s_0^{-\frac{p-1}2})
\end{align*}
as $s_0\to \infty$. Using again \eqref{eq} and \eqref{portion0}, we see that \eqref{prochex} follows from \eqref{inter0}.
This concludes the proof of Proposition \ref{propyr}. 
\end{proof}

%%%%%%%%%%%%%%%%%%%%%%%%%%%%%%%%%%%%%
%%%%%%%%%%%%%%%%%%%%%%%%%%%%%%%%%%%%%
\section{Behavior of $w_x$ for $x\neq 0$}\label{secwx}
%%%%%%%%%%%%%%%%%%%%%%%%%%%%%%%%%%%%%
%%%%%%%%%%%%%%%%%%%%%%%%%%%%%%%%%%%%%
Here, we leave the origin and consider $w_x$ for $x\neq 0$. We proceed in several steps:\\
- In Section \ref{subprop1}, we derive some rough upper and lower bound on the blow-up surface, which follow from the decomposition \eqref{cprofile0} together with the upper blow-up bound we proved in \cite{MZimrn05} for $w_x$.\\
- In Section \ref{secfor}, we introduce what we call {\it ``the soliton-loosing mechanism''}, and formally present a scenario which shows that even though $w_x(s_0)$ has 4 solitons, for $s_0$ large and fixed and $x$ small enough, it will loose two of its solitons at some later time, if $x$ in on the bisectrices, and even a third soliton if not. This will give the behavior of $w_x$, based on a formal argument.\\
- In Section \ref{secrig}, we give the precise statements and proofs to justify the behavior of $w_x$ we have just found.

%%%%%%%%%%%%%%%%%%%%%%%%%%%%%%%%%%%%%
%%%%%%%%%%%%%%%%%%%%%%%%%%%%%%%%%%%%%
\subsection{Rough upper and lower bounds on the blow-up graph}\label{subprop1}
%%%%%%%%%%%%%%%%%%%%%%%%%%%%%%%%%%%%%
%%%%%%%%%%%%%%%%%%%%%%%%%%%%%%%%%%%%%
In the following, we give rough upper and lower bounds on the blow-up 
surface $\Gamma$:
%%%%%%%%%%%%%%%%%%%%%%%%%%%%%%%%%%%%%
%%%%%%%%%%%%%%%%%%%%%%%%%%%%%%%%%%%%%
\begin{prop}[Upper and lower bounds on the blow-up 
surface
]\label{prop1}
For all $\epsilon>0$, there exists $\delta_1(\epsilon)>0$ such that for all $|x|<\delta_1$, we have
\begin{equation}\label{pyramid}
-|x|\le T(x) \le -(1-\epsilon)\max(|x_1|,|x_2|).
\end{equation}
\end{prop}
%%%%%%%%%%%%%%%%%%%%%%%%%%%%%%%%%%%%%
%%%%%%%%%%%%%%%%%%%%%%%%%%%%%%%%%%%%%
\begin{proof}
  Note first that since $T(0)=0$, the lower bound follows from the fact that $x\mapsto T(x)$ is 1-Lipschitz, and the upper bound is obvious for $x=0$. It remains then to prove the upper bound for $x\neq 0$.
$\Gamma$.
Let us first explain how we derive it, before giving 
the precise proof.\\
 Consider $x\neq 0$. Using the information on $w_0$ given in \eqref{cprofile0}, we can use the similarity variables' transformation \eqref{defw} back and forth to derive an estimate on $w_x$. That estimate depends of course on the value of $T(x)$. In particular, if $T(x)$ is too large, then, the energy norm (in fact the $L^{p+1}_\rho$ norm) will be too large, contradicting the upper bound we proved in \cite{MZajm03} and \cite{MZimrn05}. Thus, $T(x)$ cannot be too large.
Let us now give the precise proof.\\
Consider $\epsilon\in (0,1)$.
From the symmetries of the solution, it is enough to consider $x=(x_1,x_2)$ such that
\begin{equation}\label{portion}
0\le x_2\le x_1 \le \delta_1\mbox{ with }x_1\neq 0,
\end{equation}
where $\delta_1>0$ will be chosen small enough.\\
Arguing by contradiction, we assume that 
\begin{equation}\label{diction}
T(x) >T_\epsilon(x)\equiv-(1-\epsilon) x_1.
\end{equation}
Introducing the quantity
\begin{equation}\label{defI}
I(x) =\int_{s_1}^{s_1+1}\int_{|y|<1} |w_x(y,s)|^{p+1}\rho(y) dyds \mbox{ where }
s_1(x) = -\log(T(x)-t_1(x))
\end{equation}
where $t_1(x)<T(x)$ will be suitably chosen, the idea is to bound $I(x)$ from above thanks to the blow-up rate we determined in our paper \cite{MZimrn05}, then from below thanks to \eqref{cprofile0}, which will imply a constraint on $T(x)$, contradicting \eqref{diction}, hence, yielding \eqref{pyramid}.

\medskip

Of course, the key point is to carefully choose $t_1(x)$. 

\medskip

Let us proceed with the proof, and we will see that the value of $t_1(x)$ will naturally appear (see \eqref{deft1} below). We have three steps, one devoted to the upper bound on $I(x)$ \eqref{defI}, the second to the lower bound, and the third to the conclusion.

\bigskip

{\bf Step 1: The upper bound on $I(x)$}: Let us first recall the following upper bound on the blow-up rate from \cite{MZajm03} and \cite{MZimrn05}:
%%%%%%%%%%%%%%%%%%%%%%%%%%%%%%%%%%%%%
%%%%%%%%%%%%%%%%%%%%%%%%%%%%%%%%%%%%%
\begin{prop}[Upper bound on the blow-up rate; Proposition 2.2 page 1134 in \cite{MZimrn05}]\label{propupper}
Consider $w$ a solution of equation \eqref{equ} defined for all $(y,s) \in B(0,1) \times [s^*, \infty)$ for some $s^*\in \m R$. Then, for all $s\ge s^*+1$, we have
\[
\int_s^{s+1} |w(y,s')|^{p+1}\rho(y) dyds' \le C(E(w(s^*), \partial_s w(s^*))+1),
\]
where $E$ is the Lyapunov functional defined in \eqref{defenergy}.
\end{prop}
%%%%%%%%%%%%%%%%%%%%%%%%%%%%%%%%%%%%%
%%%%%%%%%%%%%%%%%%%%%%%%%%%%%%%%%%%%%
\begin{proof}
This statement is given in Proposition 2.2 page 1134 in \cite{MZimrn05}, and the proof follows from Section 2 in \cite{MZajm03}.
\end{proof}
Imposing the condition
\begin{equation}\label{hyp1}
t_1(x) \to 0\mbox{ as }x\to 0,
\end{equation}
and recalling that $T(x) \to T(0)=0$ from Proposition \ref{propyr}, we see that $s_1$ defined in \eqref{defI} goes to infinity as $x\to 0$.
 Therefore, using the upper bound on the blow-up rate proved in \cite{MZimrn05} and 
given above
in Proposition \ref{propupper}, we see that 
\begin{equation*}%\label{up0}
I(x)\le C+CE((w_x, \partial_s w_x)(-\log T(x)+1))
\end{equation*}
where the functional $E$ is a Lyapunov functional for equation \eqref{eqw}, as proved by Antonini and Merle in \cite{AMimrn01}, defined by
\begin{align}
&E(w(s),\partial_s w(s))\label{defenergy}\\
&= \iint \left(\frac 12 \left(\partial_s w\right)^2 + \frac 12 |\nabla w|^2 -\frac 12 (y\cdot\nabla w)^2+\frac{(p+1)}{(p-1)^2}w^2 - \frac 1{p+1} |w|^{p+1}\right)\rho dy.\nonumber
\end{align}
Since the values of $(w_x, \partial_s w_x)$ at time $s=-\log(T(x)+1)$ correspond to the values of initial data $(u_0, u_1)$ which are in 
$H^1\times L^2(\m R^2)$, 
it follows from Lebesgue's theorem that $E((w_x, \partial_s w_x)(-\log T(x)+1))\to E((w_0, \partial_s w_0)(0))$ as $x\to 0$. Thus, taking the parameter $\delta_1$ definining the region \eqref{portion} small enough, we see that
\begin{equation}\label{up0}
I(x)\le C.
\end{equation}

{\bf Step 2: The lower bound on $I(x)$}: Going back to the original variables $u(x,t)$ through the similarity variables' transformation \eqref{defw}, we write:
\begin{align*}
I(x)&=\int_{s_1}^{s_1+1}e^{-\frac{2(p+1)s}{p-1}}\int_{|y|<1} 
|u(x+ye^{-s}, T(x)-e^{-s})|^{p+1}\rho(y) dyds\\
&=\int_{t_1}^{t_1'}(T(x)-t)^{2\alpha}
\int_{|\xi-x|< T(x)-t}|u(\xi,t)|^{p+1}\rho\left(\frac{\xi-x}{T(x)-t}\right)d\xi dt,
\end{align*}
where $\alpha>0$ is introduced in \eqref{defro} and
\begin{equation}\label{deft1'}
 t_1'=T(x) - e^{-s_1-1}\mbox{ satisfies }T(x) - t_1'= (T(x)-t_1)/e.
\end{equation}
 In particular, we have
\begin{equation}\label{poids}
 e^{-1}\le \frac{T(x)-t}{T(x)-t_1}\le 1\mbox{ whenever }t_1\le t\le t_1'.
\end{equation} 
Imposing the additional condition
\begin{equation}\label{hyp2}
t_1'(x)<T_\epsilon(x),
\end{equation}
and using \eqref{poids} together with the fact that $T(x) \ge T_\epsilon(x)$ (see \eqref{diction}), we can bound $I(x)$ as follows:
\begin{equation}\label{ii00}
I(x) \ge (T(x)-t_1)^{2 \alpha}\int_{t_1}^{t_1'}
\int_{|\xi-x|< T_\epsilon(x)-t}|u(\xi,t)|^{p+1}\rho\left(\frac{\xi-x}{T_\epsilon(x)-t}\right)d\xi dt.
\end{equation}
By restricting the integration domain to a smaller ball $\q B_1(t) =B(x, (1-\delta)(T_\epsilon(x)-t))$ where $\delta>0$ will be fixed small enough (independently from $x$, see \eqref{defdelta} below), we get rid of the weight $\rho$ in \eqref{ii00} and write
\[
I(x) \ge \frac{\delta^\alpha}C (T(x)-t_1)^{2 \alpha}\int_{t_1}^{t_1'}\int_{\q B_1(t)}|u(\xi,t)|^{p+1}d\xi dt.
\]
From \eqref{hyp2} and \eqref{diction}, we see that we have information on $u(\xi,t)$ in the ball $\q B_0=B(0, -t)$ as shown by \eqref{cprofile0}, which holds true thanks to Proposition \ref{propyr}. Therefore, we can further restrict the integration domain and write
\[
I(x) \ge \frac{\delta^\alpha}C (T(x)-t_1)^{2 \alpha}\int_{t_1}^{t_1'}\int_{\q B_1\cap \q B_0(t)}|u(\xi,t)|^{p+1}d\xi dt.
\]
Going into similarity variables $z=\frac \xi{-t}$ and $\sigma=-\log(-t)$, we write
\[
I(x) \ge \frac{\delta^\alpha}C \left(T(x)-t_1\right)^{2\alpha}\int_{s_+}^{s_+'}
\int_{\q {\bar B}}(-t)^{-2\alpha}|w_0(z,\sigma)|^{p+1}dz d\sigma
\]
where
\begin{equation}\label{defs+}
s_+(x)=- \log(-t_1(x))\mbox{ and } s_+'(x) = -\log(-t_1'(x)),
\end{equation}
and 
\begin{equation}\label{defB}
\q {\bar B}(\sigma)=  B(0,1)\cap \B_1(\sigma)\mbox{ with }\B_1(\sigma)= B\left(-\frac x{t},(1-\delta)(1-\frac{T_\epsilon}{t})\right)\mbox{ and }-\frac 1t= e^\sigma.
\end{equation}
Since 
$p<5$, hence $\alpha>0$, assuming that 
\begin{equation}\label{hyp3}
t_1'<0
\end{equation}
we see that $(-t)^{-2 \alpha}\ge (-t_1)^{-2 \alpha}$ whenever $t_1\le t\le t_1'$, hence
\begin{equation*}%\label{defs+}
I(x) \ge \frac{\delta^\alpha}C \left(\frac{T(x)-t_1}{-t_1}\right)^{2\alpha}
\int_{s_+}^{s_+'}\int_{\q {\bar B}(\sigma)}|w_0(z,\sigma)|^{p+1}dzd\sigma.
\end{equation*}
 Since we want to use the expansion \eqref{cprofile0} for $w_0(z,\sigma)$, which holds in $\q H$, hence in $L^{p+1}_\rho(|z|<1)$ by the Hardy-Sobolev inequality given in Lemma \ref{lemhs}, 
we will restrict once more the integration domain to some $\q {\bar E}(\sigma) \subset \q {\bar B}(\sigma)$ 
not far from the unit circle. 
Considering such a set, we write
\begin{equation}\label{lowI}
I(x) \ge \frac{\delta^\alpha}{CD^\alpha} \left(\frac{T(x)-t_1}{-t_1}\right)^{2 \alpha}
\int_{s_+}^{s_+'}\N(w_0(\sigma))^{p+1}d\sigma
\end{equation}
where 
\begin{equation}\label{defNv}
\N(v)^{p+1}=\int_{\E(\sigma)}|v(z)|^{p+1}\rho(z)dz\mbox{ and }
D=\max_{s_+\le \sigma\le s_+',\;z\in \E(\sigma)}1-|z|>0.
\end{equation}
 Since we choose $x$ in the portion defined by \eqref{portion}, it is reasonable to choose the integration domain $\q {\bar E}(\sigma)$ localized near the point $(0,1)$, which is the region where the right soliton $\kappa(\bar d(\sigma)e_1, y)$ is dominant with respect to the 3 others appearing in the expansion \eqref{cprofile0}. Hence, one may expect that 
\begin{equation}\label{expect}
\N(w_0(\sigma))\sim \N(\kappa(\bar d(\sigma)e_1))
\end{equation}
Then, since 
\begin{equation}\label{upn}
\N(\kappa(\bar d(\sigma)e_1))\le \int_{|z|<1}\kappa(\bar d(\sigma)e_1,z)^{p+1} \rho(z) dz\le C_0
\end{equation}
 from Lemma \ref{lemhs} and item (i) in Lemma \ref{lemkd}, we will choose the integration domain $\q {\bar E}(\sigma)$ large enough to have 
\[
\N(\kappa(\bar d(\sigma)e_1, y)=c_0
\]
for some constant $c_0>0$, which is the maximal possible rate, from \eqref{upn}. Given that 
\begin{equation}\label{lorint}
\forall d\in (-1,1),\;\;\int_{\EE(d)}\kappa(de_1,Y)^{p+1} \rho(Y)dY=\kappa_0^{p+1}\int_{|y|<\frac 12} \rho(y)dy
\end{equation}
where the integration domain $\EE(d)\subset B(0,1)$ is the ellipse centered at $(c(d),0)=(-\frac{3d}{4-d^2},0)$ whose horizontal axis is $2a(d)=\frac{4(1-d^2)}{4-d^2}$ and its vertical axis is $2b(d)=2\sqrt{\frac{1-d^2}{4-d^2}}$ (see Lemma \ref{lemlorint} for a justification of this fact), we will simply choose $\E(\sigma) = \EE(\bar d(\sigma))$, and this is possible only if 
\begin{equation}\label{const}
\EE(\bar d(\sigma))\subset \B_1(\sigma)
\end{equation}
 defined in \eqref{defB}. This last condition will induce a constraint on $s_+$, fixing the value of $t_1$, obeying the 3 above-mentioned conditions \eqref{hyp1}, \eqref{hyp2} and \eqref{hyp3}. 
More precisely, given that $x$ satisfies \eqref{portion} on the one hand, and $s_+$ is large, $\sigma \ge s_+$ and $\bar d(\sigma)$ is close to $-1$ from \eqref{hyp1}, \eqref{hyp2}, \eqref{defs+} and \eqref{solpart} on the other hand, the constraint \eqref{const} will be satisfied if the four corners of the ellipse bounding rectangle are in $\B_1(\sigma)$, i.e.
\begin{equation}\label{conrec}
\forall s_+\le \sigma \le s_+',\;\;(c(\bar d(\sigma))+\theta a(\bar d(\sigma)),\theta' b(\bar d(\sigma)))\in \B_1(\sigma)
\mbox{ where }\theta, \theta'\in \{-1,1\}. 
\end{equation}
Given that $x_1\ge x_2 \ge 0$ by \eqref{portion}, it is enough to satisfy \eqref{conrec} only for the lower corners, defined by $\theta'=-1$ and $\theta=\pm 1$. In other words, it is enough to satisfy for all 
$t\in [t_1, t_1']$
%$\sigma \in [s_+, s_+']$ 
and $\theta=\pm 1$, $f_\theta(t)\ge 0$, where 
\[
f_\theta(t) = (1-\delta)^2 \left(1-\frac{T_\epsilon}{t}\right)^2 - \left(c(\bar d(\sigma))+\theta a(\bar d(\sigma))+\frac{x_1}{t}\right)^2-\left(-b(\bar d(\sigma))+\frac{x_2}{t}\right)^2
%\mbox{ where } -\frac 1t = e^\sigma.
\]
and $-\frac 1t = e^\sigma$.
Using the definition \eqref{diction}, \eqref{defs+} and \eqref{solpart} of $T_\epsilon(x)$, $s_+(x)$ and $\bar d(s_+)$, together with the parameters of the ellipse given above, 
we see that when $t_1\to 0$, we have $s_+\to \infty$, 
\begin{equation}\label{abc0}
\bar d(\sigma)\to -1,\;\;c(\bar d(\sigma))\to 1,\;\;a(\bar d(\sigma))\to 0\mbox{ and }b(\bar d(\sigma))\to 0,
\end{equation}
uniformly for $\sigma \in [s_+, s_+']$, 
 hence 
\begin{equation}\label{t2f}
t^2f_\theta(t) = t^2((1-\delta)^2-1+o(1))+2tx_1[(1-\delta)^2(1-\epsilon)-1]
+x_1^2[(1-\delta)^2(1-\epsilon)^2 -1+o(1)]-x_2^2
\end{equation}
as $x_1 \to 0$ (and $x$ satisfies \eqref{portion}). A quick study of this second order equation with unknown $t$ shows that for 
\begin{equation}\label{defdelta}
\delta = \frac{\epsilon^2}8,
\end{equation}
 the discriminant
\[
 \Delta_\theta =4x_1^2[\epsilon^2+O(\epsilon^3)+o(1)]-[\epsilon^2+O(\epsilon^3)+o(1)]x_2^2\ge x_1^2[3\epsilon^2+O(\epsilon^3)+o(1)]
\]
as $x\to 0$, 
since $0\le x_2\le x_1$ by \eqref{portion} (note that $O(\epsilon^3)$ stands for a term bounded by $C\epsilon^3$ and $o(1)$ stands for a function which goes to $0$ as $x\to 0$). In particular, $\Delta_\theta\ge 0$ if $\epsilon$ is small enough, then $x$ is less than some function of $\epsilon$. Since the coefficient of $t^2$ in \eqref{t2f} is negative under the same conditions, we see that $f_\theta(t)$ may have positive values,  also for small $\epsilon$ and $x$ less than some function of $\epsilon$. As a matter of fact, solving for the roots, we see that taking
\begin{equation}\label{deft1}
t_1=-\frac{4x_1}\epsilon,
\end{equation}
we have from \eqref{t2f} and \eqref{defdelta},
\[
t_1^2f_\theta(t_1) = x_1^2[4+O(\epsilon)+\frac{o(1)}{\epsilon^2}]-x_2^2\ge x_1^2[3+O(\epsilon)+\frac{o(1)}{\epsilon^2}]\ge 0,
\]
for $\epsilon$ small enough, and $|x|$ less than some function of $\epsilon$,
again from \eqref{portion}. Similarly, since $|T(x)|\le |x|\le x_1\sqrt 2$ from \eqref{portion}, using the expression \eqref{deft1'} of $t_1'$, we see that 
\begin{equation}\label{equivt1'}
t_1'=T(x)\left(\frac{e-1}e\right)+\frac{t_1}e
= x_1 \left[-\frac 4{\epsilon e}+O(1)\right]
\end{equation}
for $\epsilon$ small enough,
hence, 
\[
(t_1')^2f_\theta(t_1') = x_1^2[\frac 8e - \frac 4{e^2}+O(\epsilon)+\frac{o(1)}{\epsilon^2}]-x_2^2\ge x_1^2[\frac 8e - \frac 4{e^2}-1+O(\epsilon)+\frac{o(1)}{\epsilon^2}]\ge 0,
\]
for $x$ small enough, and $|x|$ less than some function of $\epsilon$.
Therefore, from sign considerations for the second-order polynomial in \eqref{t2f},
it follows that $f_\theta(t)\ge 0$ 
for any $t\in [t_1, t_1']$, and \eqref{const} holds for this choice of $t_1$, provided that $\epsilon$ is small enough, and $|x|$ is less than some function of $\epsilon$. Therefore, as announced right before \eqref{const}, we can choose
\begin{equation}\label{esigma}
\E(\sigma) = \EE(\bar d(\sigma)),
\end{equation}
 and see from \eqref{lorint} and the definition \eqref{defNv} of the norm $\N$ that
\begin{equation}\label{nsk}
\forall\sigma \in [s_+, s_+'],\;\;
\N(\kappa(\bar d(\sigma))= \kappa_0\left[\int_{|y|<\frac 12} \rho(y) dy\right]^{\frac 1{p+1}}>0, 
\end{equation}
a constant indpendent of $\sigma$.
Let us remark that the choice \eqref{deft1} for $t_1$ is consistent with the 3 conditions \eqref{hyp1}, \eqref{hyp2} and \eqref{hyp3} given above, as one may see from the  definition \eqref{diction} of $T_\epsilon(x)$. In particular, all the above arguments hold, and the following claim allows us to derive a nice lower bound on the quantity $I(x)$ defined in \eqref{defI}:
%%%%%%%%%%%%%%%%%%%%%%%%%%%%%%%%%%%
%%%%%%%%%%%%%%%%%%%%%%%%%%%%%%%%%%%
\begin{cl}\label{cli0i1}We have:\\
(i) $\sup_{s_+\le \sigma \le s_+'}\left|\N(w_0(\sigma))- \N(\kappa(\bar d(\sigma)))\right|\to 0$ as $x\to 0$, where the norm $\N$ is introduced in \eqref{defNv}.\\
(ii) The quantity $D$ defined in \eqref{defNv} satisfies $D\le C  |\log x_1|^{-\frac{p-1}2}$.
\end{cl}
%%%%%%%%%%%%%%%%%%%%%%%%%%%%%%%%%%%
%%%%%%%%%%%%%%%%%%%%%%%%%%%%%%%%%%%
Indeed, using 
\eqref{nsk}
and item (i) of this claim, we see that the expected estimate \eqref{expect} holds, provided that $|x|$ is less than some function of $\epsilon$ and $\epsilon$ is taken small enough. Furthermore, from \eqref{deft1} and \eqref{equivt1'}, we see by definition \eqref{defs+} of $s_+(x)$ and $s_+'(x)$ that $s_+'(x)-s_+(x)=1+O(\epsilon)$.
Therefore, using \eqref{lowI} together with \eqref{defdelta} and item (ii) of Claim \ref{cli0i1}, we see that
\begin{equation}\label{lowi}
I(x) \ge C\epsilon^{2\alpha}|\log x_1|^{\frac{2\alpha}{p-1}}\left(\frac{T(x)-t_1}{-t_1}\right)^{2 \alpha},
\end{equation}
provided that $\epsilon$ is small enough, and $|x|$ is taken also small enough, less than some function of $\epsilon$.\\
Now, it remains to prove Claim \ref{cli0i1} in order to conclude the proof of \eqref{lowi}.
\begin{proof}[Proof of Claim \ref{cli0i1}] $ $\\
(i) Take $\sigma \in [s_+, s_+']$.  Since $\N$ is a norm, using the triangular inequality, we write
\begin{equation*}
|\N(w_0(\sigma))- \N(\kappa(\bar d(\sigma)))|\le \N[w_0(\sigma)- \kappa(\bar d(\sigma))]
\le N_0+\sum_{(i,\theta)\neq (1,1)} N_{i,\theta}
\end{equation*}
where 
\[
N_0=\N\left[w_0(\sigma)- \sum_{(j,\eta)}\eta\kappa(\eta \bar d(\sigma)e_j)\right]
\mbox{ and }N_{i,\theta}= \N(\kappa(\theta\bar d(\sigma)e_i)).
\]
Since $\E(\sigma)=\bar {\q E}_0(\bar d(\sigma))\subset B(0,1)$ from \eqref{esigma} and Lemma \ref{lemlorint}, using the Hardy-Sobolev estimate given in Lemma \ref{lemhs}, we write
\begin{align*}
\sup_{s_+\le \sigma \le s_+'}N_0 \le & \sup_{s_+\le \sigma \le s_+'}\left\|w_0(\sigma)- \sum_{(j,\eta)}\eta\kappa(\eta \bar d(\sigma)e_j)\right\|_{L^{p+1}_\rho(|z|<1)}\\
\le& C\sup_{s_+\le \sigma \le s_+'}\left\|w_0(\sigma)- \sum_{(j,\eta)}\eta\kappa(\eta \bar d(\sigma)e_j)\right\|_{\q H}\to 0
\end{align*}
as $s_+\to \infty$ from estimate \eqref{cprofile0}, which holds true by Proposition \ref{propyr}, hence also as $x\to 0$, by definitions \eqref{defs+} and \eqref{deft1} of $s_+$ and $t_1$ (remember that we had to take $\epsilon$ small enough then take $|x|$ less than some function of $\epsilon$, so that various estimates hold, in particular \eqref{const}, which is crucial to justify the choice we made in \eqref{esigma}).\\
As for $N_{i,\theta}$ for $(i,\theta) \neq (1,1)$, simply note that since the ellipse $\E(\sigma)$ is localized near the point $z=(1,0)$ from \eqref{abc0}, it follows again from \eqref{abc0} that $\kappa(\theta \bar d(\sigma),y)\le C(1-\bar d(\sigma))^{\frac 1{p-1}}$ when $z\in \E(\sigma)$, hence $N_{i,\theta}\le C(1-\bar d(\sigma)^2)^{\frac 1{p-1}}|\E(\sigma)|^{\frac 1{p+1}}\to 0$, since $\sigma \ge s_+\to \infty$ as $x\to 0$ from \eqref{defs+} and \eqref{deft1}. Thus, item (i) follows.\\
(ii) Take $\sigma \in [s_+, s_+']$. Refining estimate \eqref{abc0}, we see from \eqref{devdbar} together with definitions \eqref{defs+} and \eqref{deft1} of $s_+$ and $t_1$ that
\begin{align}
&\frac{|\log x_1|^{-\frac{p-1}2}}C\le \bar d(\sigma)+1\le C|\log x_1|^{-\frac{p-1}2},\label{equid}\\
&1-c(\bar d(\sigma))\sim \frac 53(1+\bar d(\sigma)),\;\;
a(\bar d(\sigma))\sim \frac 43(1+\bar d(\sigma)),\;\;
b(\bar d(\sigma))\sim \sqrt{\frac{2(1+\bar d(\sigma))}3}
\end{align}
as $x\to 0$, where $c$, $a$ and $b$ are defined in Lemma \ref{lemlorint} (take $d=\bar d(\sigma)$ by \eqref{esigma}). Inserting the ellipse $\E(\sigma)$ in the bounding rectangle as before, we see that $D\le 1-|(c(\bar d(\sigma))-a(\bar d(\sigma)), b(\bar d(\sigma)))| \le C(1+\bar d(\sigma))$ and the conclusion follows from \eqref{equid}. This concludes the proof of Claim \ref{cli0i1}, and also the proof of \eqref{lowi}.
\end{proof}

\bigskip

{\bf Step 3: Conclusion of the proof of Proposition \ref{prop1}}

Assuming that $\epsilon$ is small enough, and that $|x|$ is also small enough, less than some function of $\epsilon$, we see that the upper bound \eqref{up0} and the lower bound \eqref{lowi} on $I(x)$ \eqref{defI} both hold. Therefore, it follows that
\[
C_0 \ge I(x)^{\frac 1{2\alpha}} \ge C\epsilon|\log x_1|^{\frac 1{p-1}}\left(\frac{T(x)-t_1}{-t_1}\right),
\]
therefore, by definition \eqref{deft1} of $t_1(x)$, we have
\[
T(x)\le t_1(x)(1-C\epsilon^{-1}|\log x_1|^{-\frac 1{p-1}})
\le -\frac {3x_1}\epsilon 
\]
for $x$ small enough. Taking $\epsilon$ small enough, we see that this is in contradiction with \eqref{diction}. Thus, Proposition \ref{prop1} holds.
\end{proof}

%%%%%%%%%%%%%%%%%%%%%%%%%%%%%%%%%%%%
%%%%%%%%%%%%%%%%%%%%%%%%%%%%%%%%%%%%
\subsection{The expected behavior of the solution outside the origin}\label{secfor}
%%%%%%%%%%%%%%%%%%%%%%%%%%%%%%%%%%%%
%%%%%%%%%%%%%%%%%%%%%%%%%%%%%%%%%%%%

The construction procedure provided us with a sharp estimate on the behavior of $w_0$: see estimate \eqref{cprofile0} justified in Proposition \ref{propyr}. That estimate shows 4 solitons which have the same size (they are in fact symmetric).\\
As a next step, we will consider $x\neq 0$ and look for the behavior of $w_x(s)$, for $s$ large. 
In this section, we will adopt a formal approach in order to derive the behavior of $w_x(s)$, which follows from what we call the {\it ``soliton-loosing mechanism''}. In a later step (see Section \ref{secrig} below), we will adopt a rigorous approach in order to prove that behavior for $w_x(s)$. 

%%%%%%%%%%%%%%%%%%%%%%%%%%%%%%%%%%%%
%%%%%%%%%%%%%%%%%%%%%%%%%%%%%%%%%%%%
\subsubsection{Deformation of the solitons outside the origin}\label{secbeh}
%%%%%%%%%%%%%%%%%%%%%%%%%%%%%%%%%%%%
%%%%%%%%%%%%%%%%%%%%%%%%%%%%%%%%%%%%
Consider an arbitrary $\epsilon\in (0, \frac 1{100})$ and 
$x=(x_1,x_2)\neq 0$. From the symmetries of the solution and Proposition \ref{prop1}, we may assume that
\begin{equation}\label{portion2}
0\le x_2\le x_1 \le \delta_2
\mbox{ with }x\neq 0
\end{equation}
and
\begin{equation}\label{pyramid2}
-x_1\sqrt 2\le -|x|\le T(x) \le -(1-\epsilon)x_1\le -\frac{99}{100}x_1,
\end{equation}
for some $\delta_2=\delta_2(\epsilon)>0$ which will be taken small enough.\\
 In this section, we aim at deriving the behavior of $w_x(s)$ for large $s$. Our approach is formal and simple:\\
- formal, since we assume that $w_0(y,s)$ is exactly the sum of the 4 solitons appearing in \eqref{cprofile0}, with no error term;\\
- simple, since we just use the similarity variables' transformation \eqref{defw} in order to link $w_x$ and $w_0$, seeing that $w_x$ also has 4 solitons, though not of the same size, since the symmetry is broken when $x\neq 0$.

\medskip

In order to make our argument clearer, we use the notation $W(Y,S)$ instead of $w_0(y,s)$ defined in the selfsimilar transformation in \eqref{defw}. As said just above, we assume (formally) that estimate \eqref{cprofile0} holds with no error terms, which means that 
\begin{equation}\label{defg**}
\vc{W(\cy,\cs)}{\partial_s W(Y,S)}=\sum_{dir \in \q D} \tdir\vc{\kappa(\bar d(\cs)\edir,\cy)}{0}
\end{equation}
where $\bar d(S)$ is given in \eqref{solpart} and the direction ``$dir$'' belongs to the set
\begin{equation}\label{defD}
\q D = \{right,\; left,\; up,\; down\},
\end{equation}
with
\begin{equation}\label{deftedir}
\er=e_1,\;\el=-e_1,\;\eup = e_2,\;\ed=-e_2,\;
\tr =
\tl = 
-\tup =
-\td = 1.
\end{equation}
 Note that the directional terminology we use here follows the natural localization of the center of the solitons $\kappa(\bar d(\cs)\edir,\cy)$, after the transformation
\begin{equation}\label{defvb}
\bar v(\xi)=(1-z^2)^{\frac 1{p-1}}v(z)\mbox{ with }z=\tanh \xi\in \m R,
\end{equation}
applied with $z=y_i$ and $i=1,2$.\\
%
% \bigskip
%
Now, let us introduce the following algebraic transformation 
\begin{equation}\label{defw**}
\tT_x(v)(\ty, \ts)=(1-T(\tx) e^{\ts})^{-\frac 2{p-1}} v(\cy,\cs),\;\;
\cy =\frac{\ty +\tx e^{\ts}}{1-T(\tx) e^{\ts}},\;\;
\cs= \ts -\log(1-T(\tx) e^{\ts}).
\end{equation}
In fact, the transformation $\q T_x$ is simply the composition of the inverse similarity transformation \eqref{defw} centered at $(0,T(0))=(0,0)$ with the direct similarity transformation (also given in \eqref{defw}) centered at $(x,T(x))$. As a matter of fact, given that $W=w_0$, we see from \eqref{defw} that
\begin{equation}\label{rach1}
\tT_x(W)=w_{\tx}.
\end{equation}
Using the definitions \eqref{defkd} and \eqref{defk*} of $\kappa(\bd,y)$ and $\kappa_1^*(\bd,\nu,\ty)$, 
we see that for 
$dir \in \q D$,
$d\in (-1,1)$
and $\ts \in \m R$, if $1+[dx.\edir -T(x)] e^{\ts}>|d|$, then
\begin{equation}\label{rach2}
\forall \ty \in (-1,1),\;\;
\tT_x(\kappa(\bar d(\cs)\edir,\cdot))(\ty, \ts)=\kdir(x,y,s)
\end{equation}
where 
\begin{equation}\label{foursol}
\kdir(x,y,s)=\kappa^*(\hat d(x,s) \edir, \ndir(x,s),y) 
\end{equation}
with $\tdir$ and $\edir$ given in \eqref{deftedir}, 
\begin{equation}\label{defs2}
\ndir(x,\ts)=[\cd(x,s) x\cdot\edir-T(x)]e^{\ts},\;\;
\cd(x,\ts)=\bar d(\cs(x,s)),\;\;
-e^{-\cs(x,s)}=T(\tx) - e^{-\ts},
\end{equation}
 and $\bar d$ is introduced in \eqref{solpart}. Since $T(x)<0$ from \eqref{pyramid2}, it follows that for $\delta_2$ small enough, the time variable $s$ lies in $[\bar s_0,\infty)$ and $S\in [\bar s_0, -\log(-T(x)))$, for some fixed and large $\bar s_0$.\\
Applying the transformation $\q T_x$ to the first component of identity \eqref{defg**}, we see from \eqref{rach1} and \eqref{rach2} that the first component of the following identity holds
\begin{equation}\label{defh**}
\vc{w_{\tx}(\ty,\ts)}{\partial_s w_x(y,s)}
=\sum_{dir \in \q D}\tdir \kdir(x,y,s)
\end{equation}
(and we will also assume formally that the equality holds for the second component). 
Let us now make a remark concerning the ``localization'' and the ``shape'' of the generalized soliton $\kappa^*(\bd, \nu)$, since it appears in the decomposition of $(w_x, \partial_s w_x)(y,s)$ by \eqref{defh**} and \eqref{foursol}.\\
Using item (iii) and (i) in Lemma \ref{lemkd}, we see that 
\[
\kappa^*_1(\bd, \nu, y) = \lambda(|\bd|,\nu)\kappa\left(\bD,y\right)
\]
and
 $\left\|\kappa\left(\bD,y\right)\right\|_{\q H_0}\le C_0$ 
 uniformly, where
\begin{equation}\label{deflambda}
\bD = \frac{\bd}{1+\nu}\mbox{ and }\lambda(d, \nu)= \left[\frac{1-d^2}{(1+\nu)^2 - d^2}\right]^{\frac 1{p-1}}.
\end{equation}
Using the transformation \eqref{defvb} with $z$ parallel to $\bd$, we see that the soliton in ``localized'' (or ``centered'') around $\bD$. Furthermore, it appears that  $\lambda(|\bd|,\nu)$ controls the size  of $\kappa^*_1(\bd, \nu, y)$ in $\q H_0$. 
Given that we know from item (iii) in Lemma \ref{lemkd} that
\begin{equation}\label{compare}
\min\left(\mu,\mu^2\right)\le \lambda(|\bd|, \nu)^{p-1}\le \max\left(\mu,\mu^2\right),
\end{equation}
where
\begin{equation}\label{defmu}
\mu(d,\nu) =1/\left(1+\frac \nu{1-|d|}\right),
\end{equation}
we justify that $\mu$ may be called the ``size'' variable of the generalized soliton $\kappa^*(\bd,\nu)$. 
In particular, if $\mu$ is small, then $\lambda$ is small and so is $\kappa^*_1(\bd, \nu, y)$ (we may say that the soliton {\it ``vanishes''}, or that it is ``lost''), and if $\mu$ is large, then $\lambda$ is large and so is $\kappa^*_1(\bd, \nu, y)$ (we may say that the soliton {\it ``blows up''}).

\medskip

Introducing the ``center'' and the ``size'' of the soliton $\kdir(x,s)$:
\begin{equation}\label{deftDtm}
\Ddir(x,s)=\frac{\hat d(x,s)}{1+\ndir(x,s)}\mbox{ and }
\mdir(x,s) = 1/\left(1+ \frac{\ndir(x,s)}{1-|\hat d(x,s)|}\right),
\end{equation}
we see by definition \eqref{defs2} of $\ndir(x,s)$ that in the portion of the space where $x$ is located (see \eqref{portion2}), we have
\begin{equation}\label{order0}
\nr(x,s) \le \nup(x,s) \le \nd(x,s) \le \nl(x,s),
\end{equation}
hence
\begin{equation}\label{order}
\ml(x,s)\le \md(x,s) \le \mup(x,s) \le \mr(x,s),
\end{equation}
which means that the solitons are ``ordered'', and if ever they change size,  $\kl(x)$ would be the first to be ``lost'', if ever, and $\kr(x)$ would be the first to ``blow up'', if ever.

%%%%%%%%%%%%%%%%%%%%%%%%%%%%%%%%%%%%
%%%%%%%%%%%%%%%%%%%%%%%%%%%%%%%%%%%%
\subsubsection{The soliton-loosing mechanism}\label{secloose}
%%%%%%%%%%%%%%%%%%%%%%%%%%%%%%%%%%%%
%%%%%%%%%%%%%%%%%%%%%%%%%%%%%%%%%%%%
In this section, we remain in the formal approach.
Proceeding as we did in one dimension in \cite{MZdmj12}, we are ready to exhibit the ``soliton-loosing mechanism''. Roughly speaking, based on the order in \eqref{order}, this is the scenario that occurs for the behavior of $w_x(s)$:\\
- for small $s$, the four solitons appearing in \eqref{defh**} have comparable sizes;\\
- at some later time, the size of the left-soliton $\kl(x,s)$, becomes much smaller the others', we say that we ``loose'' it,  and we are left with only 3 solitons;\\
- shortly later, the same occurs with the down-soliton $\kd(x,s)$, and we are left with only 2 solitons;\\
- if $x_1=x_2$ (i.e. on the bisectrix $\{x_1=x_2\}$), from symmetry, we keep both the up and the right-soliton;\\
- if $x_2<x_1$ (i.e. outside that bisectrix), we loose the up-soliton $\kup(x,s)$: only the right-soliton $\kr(x,s)$ remains.

\medskip

In the following, for each of the four solitons appearing in the decomposition \eqref{defh**}, we will compute a vanishing time, depending on $x$. 

\medskip

We consider the soliton $\kdir(x,s)$ defined in \eqref{foursol}, where $dir =left$, $down$ or $up$.
Given $A>0$ large enough,
we define $\sdir(x)$ as the time such that the ``size'' 
\begin{equation}\label{condsize}
\mdir(x,\sdir(x))=\frac 1{A+1}\mbox{ i.e. }\frac{[\hat d(x,\sdir(x))x\cdot \edir -T(x)]e^{\sdir(x)}}{1-|\hat d(x,\sdir(x))|}= A>0
\end{equation}
Introducing $\Sdir(x)$ such that
\begin{equation}\label{Ss}
-e^{-\Sdir(x)}=T(x) - e^{-\sdir(x)}
\end{equation}
as in \eqref{defs2}, we see that 
\begin{equation}\label{defsdir2}
A=\frac{[\hat d(x,\sdir(x))x\cdot \edir -T(x)]e^{\sdir(x)}}{1-|\hat d(x,\sdir(x))|}= 
\frac{\bar d(\Sdir(x))x\cdot\edir-T(x)}{(1-|\bar d(\Sdir(x))|)(T(x)+e^{-\Sdir(x)})}.
\end{equation}
Since $T(x) \to T(0)=0$ as $x\to 0$, it follows from \eqref{devdbar} that
\[
\Sdir(x) \to \infty\mbox{ and }\sdir(x) \to \infty\mbox{ as }x\to 0
\]
(remember in the same time that 
\[
\bar s_0\le \Sdir(x) <-\log(-T(x)),
\]
from the remark right after \eqref{defs2}).
Using again \eqref{devdbar}, it follows that
\[
\frac{\Sdir(x)^{\frac{p-1}2}\left(\bar d(\Sdir(x))x\cdot\edir-T(x)\right)}{C_0(T(x)+e^{-\Sdir(x)})}\sim A\mbox{ as }x\to 0.
\]
Note from \eqref{pyramid2} that we have
\begin{equation}\label{simlog}
-\log|T(x)|=l+O(1)\mbox{ as }x\to 0,\mbox{ with } l=-\log x_1.
\end{equation}
Looking for a solution
\begin{equation}\label{ansatz}
\Sdir(x) = |\log|T(x)||-\sigdir(x)\mbox{ with }0<\sigdir(x)=o(l),
\end{equation}
we end up with 
\begin{equation}\label{end}
\frac{l^{\frac{p-1}2}}{C_0}e^{\sdir(x)}\left(\bar d(\Sdir(x))x\cdot\edir-T(x)\right)=
\frac{l^{\frac{p-1}2}\left(\bar d(\Sdir(x))x\cdot\edir-T(x)\right)}{C_0T(x)(1-e^{\sigdir(x)})}\sim A
\end{equation}
as $x\to 0$.

%%%%%%%%%%%%%%%%%%%%%%%%%%%%%%%%%
%%%%%%%%%%%%%%%%%%%%%%%%%%%%%%%%%
\subsubsection{The loosing-time for the left and the down-soliton} 
%%%%%%%%%%%%%%%%%%%%%%%%%%%%%%%%%
%%%%%%%%%%%%%%%%%%%%%%%%%%%%%%%%%
In these cases, $x\cdot \el = -x_1<0$ and $x\cdot \ed = -x_2\in [-x_1,0]$. Using \eqref{pyramid2} and \eqref{devdbar}, we see that in both cases, we have
\begin{equation}\label{equiv}
\bar d(\Sdir(x))x\cdot\edir-T(x)\sim -x\cdot\edir-T(x),
\end{equation}
hence, from \eqref{end}, 
\[
\frac{l^{\frac{p-1}2}\left(-x\cdot\edir-T(x)\right)}{C_0T(x)(1-e^{\sigdir(x)})}\sim A\mbox{ as }x\to 0.
\]
looking for a solution $\sigdir\to \infty$, we get
\[
\frac{l^{\frac{p-1}2}\left(x\cdot\edir+T(x)\right)}{T(x)e^{\sigdir(x)}}\sim C_0A,
\]
which yields 
\[
\sigdir(x) = \frac{p-1}2 \log l +\log\frac {x\cdot\edir+T(x)}{T(x)}- \log\left(C_0A\right)+o(1).
\]
Since 
\[
1\le \frac {x\cdot\edir+T(x)}{T(x)}\le \frac{x_1+x_1\sqrt 2}{\frac{99}{100}x_1 }\le 3
\]
from \eqref{portion2} and \eqref{pyramid2}, 
we may fix now a common value 
\begin{equation}\label{defsigl}
\sigl(A,x) = \sigd(A,x) =  \frac{p-1}2 \log l -\log (C_0A),
\end{equation}
and follow upward the previous calculations to derive the following:
%%%%%%%%%%%%%%%%%%%%%%%%%%%%%%%%%%%%
%%%%%%%%%%%%%%%%%%%%%%%%%%%%%%%%%%%%
\begin{lem}[The loosing time for the left and the down-soliton]\label{LemVanLD}
Consider $A>0$ large enough and $x$ small enough satisfying \eqref{portion2}. Then, it holds that $\|\kl(x,\sl)\|_{\q H}+\|\kd(x,\sl)\|_{\q H}\le C A^{-\frac 1{p-1}}$, where
\begin{equation}\label{defslsd2}
\sl(A,x)=\sd(A,x)\equiv -\log|T(x)|-\log\left(\frac{l^{\frac{p-1}2}}{C_0A}-1\right).
\end{equation}
\end{lem}
%%%%%%%%%%%%%%%%%%%%%%%%%%%%%%%%%%%%
%%%%%%%%%%%%%%%%%%%%%%%%%%%%%%%%%%%%
\begin{nb}
From \eqref{defslsd2} and \eqref{simlog}, it holds that
\begin{equation}\label{expsl}
\sl(A,x)=\sd(A,x)= l - \frac{p-1}2 \log l +\log A+O(1)\mbox{ as }x\to 0.
\end{equation}
\end{nb}
\begin{proof}
The proof is straightforward is one follows the above analysis upward until \eqref{condsize}. In particular, we do have
from the expression \eqref{defsigl} that 
\begin{equation}\label{ansatz2}
\sigdir(A,x) \sim \frac{p-1}2 \log l =o(l)\mbox{ as }x\to 0,
\end{equation}
 and this was required in \eqref{ansatz}.\\
As for the value of $\sl=\sd$ given in \eqref{defslsd2}, it directly follows from \eqref{defsigl} thanks to \eqref{ansatz} and \eqref{Ss}.
\end{proof}

%%%%%%%%%%%%%%%%%%%%%%%%%%%%%%%%%%%%
%%%%%%%%%%%%%%%%%%%%%%%%%%%%%%%%%%%%
\subsubsection{The loosing-time for the up-soliton} 
%%%%%%%%%%%%%%%%%%%%%%%%%%%%%%%%%%%%
%%%%%%%%%%%%%%%%%%%%%%%%%%%%%%%%%%%%
As stated in the beginning of Section \ref{secloose}, we expect the up-soliton to be lost, only outside the bisectrix $\{x_1=x_2\}$, i.e. when
\begin{equation}\label{outbis}
x_2<x_1.
\end{equation}
However, as we will see shortly, our approach in this subsection provides an expression for the loosing-time of the up-soliton (see \eqref{defSu}), only under a different condition, namely condition \eqref{condup} below. In order to overcome this difficulty, we will introduce in a later section a notion of ``relative vanishing'' of the up-soliton with respect to the right-soliton, which will turn to be valid exactly under condition \eqref{outbis}. If coupled to a notion of  ``blowing-up time'' for the right-soliton, the relative vanishing time will help us to extend the notion of a loosing-time for the up-soliton to all points outside the bisectrix, i.e. under \eqref{outbis}.

\medskip

Let us then concentrate on the introduction of the loosing-time for the up-soliton.
In comparison to the left and the down-solitons, the case of the up-soliton is indeed very delicate, since \eqref{equiv} may not hold here, as $x\cdot \eup +T(x)= x_2+T(x)$ may be zero (in the previous case, this was excluded by \eqref{portion2} and \eqref{pyramid2}). 
Since our assumption in \eqref{ansatz} implies by \eqref{simlog} that $\Su(x) \sim -\log|T(x)|\sim l$, it is legitimate to make the following further assumption:
\[
\bar d(\Su(x))x_2-T(x)\sim \bar d(-\log|T(x)|)x_2-T(x)\mbox{ as } x\to 0,
\]
in other words,
\begin{equation}\label{assump}
\sigu(x)l^{-\frac{p-1}2-1}x_2=o\left(\bar d(-\log|T(x)|)x_2-T(x)\right)
\mbox{ as } x\to 0,
\end{equation}
where $\sigu$ is defined in \eqref{ansatz}.
This way,
we may replace our requirement in \eqref{end} and \eqref{ansatz} by the following condition (where we replace $A$ by $A^2$ in order to separate the different loosing times for the solitons):
\begin{equation}\label{reduc}
\frac{l^{\frac{p-1}2}}{C_0}e^{\su(x)}\left(\bar d(-\log|T(x)|)x_2-T(x)\right)=A^2\mbox{ with } \sigu(x) = o(l)\mbox{ as }x\to 0,
\end{equation}
which readily shows that we need to have
\begin{equation}\label{condup}
\bar d(-\log|T(x)|)x_2-T(x)>0
\end{equation}
in order to get the following solution:
\begin{equation}\label{defSu}
\Su(A,x) = -\log(e^{-\su(A,x)}-T(x))
= -\log\left[\frac{[\bar d(-\log|T(x)|) x_2-T(x)]}{C_0A^2 l^{-\gamma}}-T(x)\right]
\end{equation}
where $\gamma= \frac{p-1}2$.
If
\[
\bar d(-\log|T(x)|)x_2-T(x)\le 0, \mbox{ we take }\su = +\infty.
\]
Of course, we need to check that \eqref{assump} together with the second condition in \eqref{reduc} hold (see the proof of Lemma \ref{lemlooseup} below). 
Thus, we have just proved the following statement:
%%%%%%%%%%%%%%%%%%%%%%%%%%%%%%%%%%%%
%%%%%%%%%%%%%%%%%%%%%%%%%%%%%%%%%%%%
\begin{lem}[The loosing-time for the up-soliton]\label{lemlooseup}
Consider $A>0$ large enough and $x$ small enough satisfying \eqref{portion2} and \eqref{condup}. Then, $\mu(x,\su(x))=(A^2+1+o(1))^{-1}$ and
  $\|\kup(x,\su)\|_{\q H}\le C A^{-\frac 2{p-1}}$, where $\su$ is defined in \eqref{defSu}.
\end{lem}
%%%%%%%%%%%%%%%%%%%%%%%%%%%%%%%%%%%%
%%%%%%%%%%%%%%%%%%%%%%%%%%%%%%%%%%%%
\begin{proof}
 The proof is straightforward from our previous analysis, if we check that the second condition in \eqref{reduc} together with the assumption \eqref{assump} hold.\\
For the second condition in \eqref{reduc}, simply note from the expression \eqref{defSu} together with \eqref{pyramid2}, \eqref{portion2} and \eqref{Ss} that $|\bar d(-\log T(x))x_2-T(x)|\le 3 x_1$, hence $e^{-\su} \le \frac{3x_1l^{\frac{p-1}2}}{A^2C_0}$ and
\[
(1-\epsilon) x_1 \le -T(x) \le e^{-\Su} \le - T(x) +e^{-\su} \le x_1 \sqrt 2 +3x_1l^{\frac{p-1}2}/(A^2C_0).
\]
Therefore, $\Su = l+O(\log l)$, and from \eqref{pyramid2} and the definition \eqref{ansatz} of $\sigu$, we have
\begin{equation}\label{sigu0}
\sigu= O(\log l),
\end{equation}
 and the second condition in \eqref{reduc} holds.\\
For the assumption \eqref{assump}, simply note from \eqref{Ss}, \eqref{ansatz} and \eqref{sigu0} that $\sigu \ge 0$, hence
$e^{-\su}=T(x)+e^{-\Su}= T(x)(1-e^{\sigu})\ge -T(x) \sigu$.
Using \eqref{reduc}, we see that
\begin{equation}\label{siguequiv}
\frac{l^{\frac{p-1}2}}{A^2C_0}\left(\bar d(-\log|T(x)|)x_2-T(x)\right)=e^{-\su}\ge -T(x) \sigu.
\end{equation}
Using \eqref{pyramid2} together with \eqref{portion2}, we readily see that \eqref{assump} holds.\\
\end{proof}

%%%%%%%%%%%%%%%%%%%%%%%%%%%%%%%%%%%%
%%%%%%%%%%%%%%%%%%%%%%%%%%%%%%%%%%%%
\subsubsection{Relative vanishing of the up-soliton with respect to
  the right-soliton} \label{secrelvan}
%%%%%%%%%%%%%%%%%%%%%%%%%%%%%%%%%%%%
%%%%%%%%%%%%%%%%%%%%%%%%%%%%%%%%%%%%

From \eqref{order}, we know that for any $s$ where $w_x(y,s)$ is defined, we have $\mr(x,s)\ge \mup(x,s)$. Assuming that we are outside the bisectrix $\{x_1=x_2\}$, i.e. under condition \eqref{outbis}, we will define in the following a relative vanishing time for the up-soliton with respect to the right-soliton. Given the expression \eqref{deftDtm} of $\mdir$,  we look for some time $\sur(x)$ such that
\begin{equation}\label{condrel}
\frac 1{\mup(x,\sur)}-\frac 1{\mr(x,\sur)}=2A^2,
\end{equation}
where $A>0$ is given. Again by \eqref{deftDtm}, this is equivalent to having
\[
\frac{\bar d(\Sur)(x_2-x_1)e^{\sur}}{1-|\bar d(\Sur)|}=2A^2,
\]
where
\begin{equation}\label{Ssur}
-e^{-\Sur} = T(x) - e^{-\sur}.
\end{equation}
Making the additional assumption that
\begin{equation}\label{surl}
\Sur \sim l,
%\sigur \equiv - \log|T(x)|-\Sur = o(l),
\end{equation}
we see from \eqref{devdbar} and \eqref{simlog} that 
\[
\frac{(x_1-x_2) e^{\sur}}{C_0l^{-\frac{p-1}2}}\sim 2A^2\mbox{ as }x\to 0.
\]
We may then replace the ``$\sim$'' symbol by an equality, fixing this way the value of $\sur$:
\begin{equation}\label{defsur}
\sur = -\log(x_1-x_2)-\frac{p-1}2 \log l +\log(2A^2C_0). 
\end{equation}
Since one easily checks from \eqref{Ssur} that \eqref{surl} holds, we may follow back the above calculations up to \eqref{condrel} and see that
\begin{equation}\label{condrel1}
\frac 1{\mup(x,\sur)}-\frac 1{\mr(x,\sur)}\sim 2A^2\mbox{ as }x\to 0.
\end{equation}

%%%%%%%%%%%%%%%%%%%%%%%%%%%%%%%%%%%%
%%%%%%%%%%%%%%%%%%%%%%%%%%%%%%%%%%%%
\subsubsection{The blowing-up time for the right-soliton} 
%%%%%%%%%%%%%%%%%%%%%%%%%%%%%%%%%%%%
%%%%%%%%%%%%%%%%%%%%%%%%%%%%%%%%%%%%
In order to take advantage of the notion of the relative vanishing time of the up-soliton with respect to the right-soliton, we need to introduce a notion of ``blowing-up time'' for the right-soliton.\\
The blowing-up time is found if one replaces $\frac 1{A+1}$ in \eqref{condsize} by $A^2$ (in other words, if one replaces $A$ by $-1+\frac 1{A^2}$). If we do that for the right-soliton, then we would see that  the discussion 
is in all points comparable to what we did for the loosing-time of the up-soliton.
More precisely, if we assume that
\begin{equation}\label{condright}
\bar d(-\log|T(x)|)x_1-T(x)<0,
\end{equation}
then we may define the following candidate for the blowing-up time for the right-soliton: 
\begin{equation}\label{defsrb}
\Srb(A,x) = -\log(e^{-\srb(A,x)}-T(x))= -\log[\frac{[T(x)-\bar d(-\log|T(x)|) x_1]}{C_0(1-\frac 1{A^2}) l^{-\gamma}}-T(x)]
\end{equation}
where $A>0$ is large enough.
If
\[
\bar d(-\log|T(x)|)x_1-T(x)\ge 0, \mbox{ we take }\srb=+\infty.
\]

 As a matter of fact, we have the following statement:
%%%%%%%%%%%%%%%%%%%%%%%%%%%%%%%%%%%%
%%%%%%%%%%%%%%%%%%%%%%%%%%%%%%%%%%%%
\begin{lem}[The blowing-up time for the up-soliton]\label{lemblowr}
Consider $A>0$ large enough and $x$ small enough satisfying \eqref{portion2} and \eqref{condright}. Then, we have $\mr(x,\srb(x))=A^2+o(1)$ and $\|\kr(x,\srb)\|_{\q H}\ge \frac{A^{\frac 2{p-1}}}C$, where $\srb$ is defined in \eqref{defsrb}. 
\end{lem}
%%%%%%%%%%%%%%%%%%%%%%%%%%%%%%%%%%%%
%%%%%%%%%%%%%%%%%%%%%%%%%%%%%%%%%%%%
\begin{proof}
From the discussion above, we may follow the calculations up to \eqref{condsize} and figure out that $\mr(x,\srb(x))=A^2+o(1)\ge \frac {A^2}2$.
Therefore, by definition \eqref{defnh} of the norm in $\q H$, using items (iii) and (i) in Lemma \ref{lemkd} together with \eqref{compare} and the definition \eqref{deftDtm}, we write
\[
\|\kr(x,\srb)\|_{\q H} \ge \|{\kr}_1(x,\srb)\|_{\q H} \ge \mr(x,\srb(x))^{\frac 1{p-1}}/C\ge \frac {A^{\frac 2{p-1}}}C,
\]
which is the desired conclusion.
\end{proof}

%%%%%%%%%%%%%%%%%%%%%%%%%%%%%%%%%%%%
%%%%%%%%%%%%%%%%%%%%%%%%%%%%%%%%%%%%
\subsubsection{Chronology for the soliton-loosing mechanism}
%%%%%%%%%%%%%%%%%%%%%%%%%%%%%%%%%%%%
%%%%%%%%%%%%%%%%%%%%%%%%%%%%%%%%%%%%
From the previous subsections, we keep in mind that two events may occur in the expansion of $w_x(y,s)$:\\
- the loss of the left and the down-soliton, at time $s=\sl=\sd$ defined in \eqref{defslsd2};\\
- if $x_2<x_1$, the loss of the up-soliton, the relative vanishing of the up-soliton with respect to the right-soliton, or the blowing-up of the right-soliton, whatever event occurs first. This means that we place ourselves at time 
\begin{equation}\label{defsmin0}
\smin
\equiv\min(\su, \sur, \srb),
\end{equation}
where the various times are defined in \eqref{defSu}, \eqref{defsur} and \eqref{defsrb}. We would like to prove in the following that for $A$ and $\bar s_0$ large enough, the loss of the left and the down-soliton occurs first:
%%%%%%%%%%%%%%%%%%%%%%%%%%%%%%%%
%%%%%%%%%%%%%%%%%%%%%%%%%%%%%%%%
\begin{lem}\label{lemchrono}Consider $A>0$ large enough. If $\bar s_0$ is large enough and $x$ small enough, then 
\begin{equation}\label{aim1}
\forall s\in [\bar s_0, \sl],\;\;-1+\frac 1A \le \frac{\ndir(x,s)}{1-|\bar d(S)|}\le CA,
\end{equation}
for any direction $dir$, where $\ndir$ is defined in \eqref{defs2}, provided that $x$ is small enough.
In particular,
\begin{equation}\label{aim2}
\sl(A,x) \le \smin(A,x).
\end{equation}
\end{lem}
%%%%%%%%%%%%%%%%%%%%%%%%%%%%%%%%
%%%%%%%%%%%%%%%%%%%%%%%%%%%%%%%%
\begin{proof}
We only prove \eqref{aim1}, since \eqref{aim2} follows from \eqref{aim1} thanks to estimate \eqref{condrel1} together with Lemmas \ref{lemlooseup} and \eqref{lemblowr}, provided that $A$ is large enough.\\
If $dir = left$ or $down$, arguing as for \eqref{equiv} and using \eqref{devdbar}, \eqref{expsl}, \eqref{portion2}, \eqref{simlog} and \eqref{ansatz} (which holds by \eqref{ansatz2}) together with monotonicity, we write for all $s\in [\bar s_0, \sl]$, 
\[
\frac{\bar S_0^{\frac{p-1}2}(-x\cdot \edir - T(x))e^{\bar s_0}}{2C_0}
\le \frac{\ndir(x,s)}{1-|\bar d(S)|}
\le \frac{2\Sl^{\frac{p-1}2}(-x\cdot \edir - T(x))e^\sl}{C_0}
\le CA,
\]
provided that $\bar s_0$ is large enough,
where $\bar S_0$ is associated to $\bar s_0$ as in \eqref{Ss}. Taking $x$ small enough, we see that the lower bound holds also. \\
When $dir = up$ or $right$, using \eqref{order0}, we see that we only need to prove the lower bound with $dir=right$.\\
Using \eqref{pyramid2} and \eqref{devdbar}, and monotonicity, we write for all $s\in [\bar s_0, \sl]$,
\[
 \frac{\nr(x,s)}{1-|\bar d(S)|}\ge \frac{-2\epsilon x_1e^{\sl}\Sl^{\frac{p-1}2}}{C_0}\ge - C\epsilon A \ge -1 +\frac 1{A^2},
\]
provided that we impose that $\epsilon \le \frac 1{2CA}$ (which needs $x$ to be small enough by Lemma \ref{prop1}) and $A$ is large enough. This concludes the proof of Lemma \ref{lemchrono}.
\end{proof}

%%%%%%%%%%%%%%%%%%%%%%%%%%%%%%%%%%%%%
%%%%%%%%%%%%%%%%%%%%%%%%%%%%%%%%%%%%%
\subsection{A rigorous derivation of the behavior outside the origin}\label{secrig}
%%%%%%%%%%%%%%%%%%%%%%%%%%%%%%%%%%%%%
%%%%%%%%%%%%%%%%%%%%%%%%%%%%%%%%%%%%%
In this section, we use the precise vanishing times computed in the previous section, and give a rigorous proof to show that two solitons remain when we are on the bisectrices, while only one remains outside the bisectrices. 

%%%%%%%%%%%%%%%%%%%%%%%%%%%%%%%%%%%%%
%%%%%%%%%%%%%%%%%%%%%%%%%%%%%%%%%%%%%
\subsubsection{The function $w_x$ from 4 to 2 solitons when $x\neq 0$}
%%%%%%%%%%%%%%%%%%%%%%%%%%%%%%%%%%%%%
%%%%%%%%%%%%%%%%%%%%%%%%%%%%%%%%%%%%%
In this section, we give a rigorous version of the formal estimates contained in section \ref{secbeh}. Some of our arguments are the same as in the one-dimensional case treated in \cite{MZdmj12}. They will be given in the Appendix, so that the expert reader may get easier to the argument, and the less expert may find details in the Appendix.\\
 The expansion we find will be valid from some fixed large time $s_0$ (independent from $x$) up to some time (depending on $x$) where only two solitons will remain in the expansion of $w_x$.\\
Consider an arbitrary $\epsilon>0$. As in section \ref{secbeh}, using Proposition \ref{prop1} and the symmetries of the solution, we only consider $x=(x_1,x_2)\in \m R^2$
such that 
\begin{equation}\label{portion3}
0\le x_2\le x_1 \le \delta_3
\mbox{ with }(x_1,x_2)\neq 0.
\end{equation}
and
\begin{equation}\label{pyramid3}
-x_1\sqrt 2\le -|x|\le T(x) \le -(1-\epsilon)x_1.
\end{equation}
Unlike for estimate \eqref{defh**} where we neglected error terms, here we introduce that error term:
\begin{equation}\label{defh*}
r(x,y,s) = \vc{w_{\tx}(\ty,\ts)}{\ps w_{\tx}(\ty,\ts)}
-\sum_{dir \in \q D}\tdir \kdir(x,y,s)
\end{equation}
where $\kdir$ and the directions' set $\cal D$ are introduced in \eqref{foursol} and \eqref{defD}. Recalling the expression of the vanishing time for the left and the down solitons from \eqref{defslsd2}: 
\begin{equation}\label{defslsd3}
\sl(A,x)=\sd(A,x)\equiv -\log|T(x)|-\log\left(\frac{l^{\frac{p-1}2}}{C_0A}-1\right).
\end{equation}
where $A>0$ is arbitrary large, we claim the following:
%%%%%%%%%%%%%%%%%%%%%%%%%%%%%%%%%%%%%
%%%%%%%%%%%%%%%%%%%%%%%%%%%%%%%%%%%%%
\begin{prop}[Behavior of $w_x$ for $x\neq 0$, from 4 to 2 solitons]\label{prop422}
For any $A>0$, we have
\[
\sup_{s_3\le s\le \sl(x,A)}\|r(x,s)\|_{\q H}\to 0\mbox{ as }s_3 \to \infty\mbox{ and }x\to 0,
\]
if \eqref{portion3} is satisfied.
\end{prop}
%%%%%%%%%%%%%%%%%%%%%%%%%%%%%%%%%%%%%
%%%%%%%%%%%%%%%%%%%%%%%%%%%%%%%%%%%%%
\begin{proof} Once the loosing-time for the left and the down soliton are computed (see \eqref{defslsd3}), the argument becomes the same as in the one-dimensional case treated in \cite{MZdmj12} (precisely, in the proof of Lemma 4.6 page 2864). In order to keep the paper into reasonable limits, we leave the proof to Appendix \ref{apptx}, where only the key steps of the proof are presented, and refer the reader to \cite{MZdmj12} for details.
\end{proof}
Now, we recall that the particular value we took for $\sl(A,x)$ in  Lemma \ref{LemVanLD} was designed on purpose to make the size of the left and the down solitons less than $CA^{-\frac 1{p-1}}$, hence small for $A$ large enough. Therefore, in view of Proposition \ref{prop422}, we see that only two solitons remain when $s=\sl(A,x)$, namely the up and the right solitons. More precisely, this our statement:
%%%%%%%%%%%%%%%%%%%%%%%%%%%%%%%%%%%%%
%%%%%%%%%%%%%%%%%%%%%%%%%%%%%%%%%%%%%
\begin{cor}[Only two solitons remain when $s=\sl(x,A)$]\label{cor2}
For all $A>0$, if $x$ is small enough and satisfies \eqref{portion3}, then we have
\[
\left\|(w_x,\partial_s w_x)(y,\sl)-\kr(x,y,\sl)+\kup(x,y,\sl)\right\|_{\q H}
\le C A^{-\frac 1{p-1}},
\]
where $\sl(A,x)$ is introduced in \eqref{defslsd3}.
\end{cor}
%%%%%%%%%%%%%%%%%%%%%%%%%%%%%%%%%%%%%
%%%%%%%%%%%%%%%%%%%%%%%%%%%%%%%%%%%%%

%%%%%%%%%%%%%%%%%%%%%%%%%%%%%%%%%%%%%
%%%%%%%%%%%%%%%%%%%%%%%%%%%%%%%%%%%%%
\subsubsection{
%Estimates on the blow-up surface and 
The asymptotic behavior on the bisectrices in the non-characteristic case}
%%%%%%%%%%%%%%%%%%%%%%%%%%%%%%%%%%%%%
%%%%%%%%%%%%%%%%%%%%%%%%%%%%%%%%%%%%%

In this section, we assume that $x\in \RR$ and $x$ is on the bisectrix $\{x_1=x_2\}$, namely that
\[
x=(x_1,x_2)\in \RR.
\]
We can therefore apply our argument in \cite{MZjfa07} (specifically Theorem 2 part (A) page 47), to see that the distance of $(w_x, \partial_s w_x)$ to the set of stationary solutions of equation \eqref{eqw} will go to $0$ as $s\to \infty$. Since at $s=\sl(A,x)$,  $(w_x, \partial_s w_x)(\sl)$ is near the sum of two antisymmetric solitons as we can see in Corollary \ref{cor2}, this means that $(w_x, \partial_s w_x)(s_n)$ converges, for some subsequence $s_n \to \infty$ as $n\to \infty$ to some stationary solution, itself near the sum of two antisymmetric solitons. More precisely, this our statement:
%%%%%%%%%%%%%%%%%%%%%%%%%%%%%%%%%%%%%
%%%%%%%%%%%%%%%%%%%%%%%%%%%%%%%%%%%%%
\begin{lem}[Behavior of $w_x$ when $x$ is on the bisectrix] \label{lembehbis}
If $x=(x_1,x_1)$, then $(w_x, \partial_s w_x)(s_n)$ converges, for some subsequence $s_n \to \infty$ as $n\to \infty$ to some stationary solution $w^*_x$, close to the sum of two antisymmetric solitons one can see in Corollary \ref{cor2}. 
\end{lem}
%%%%%%%%%%%%%%%%%%%%%%%%%%%%%%%%%%%%%
%%%%%%%%%%%%%%%%%%%%%%%%%%%%%%%%%%%%%
\begin{proof}
As we have just said before the statement, this is a direct application of our argument in \cite{MZjfa07} (specifically Theorem 2 part (A) page 47).
\end{proof}

%%%%%%%%%%%%%%%%%%%%%%%%%%%%%%%%%%%%%
%%%%%%%%%%%%%%%%%%%%%%%%%%%%%%%%%%%%%
\subsubsection{Loosing the third soliton outside the bisectrix}
%%%%%%%%%%%%%%%%%%%%%%%%%%%%%%%%%%%%%
%%%%%%%%%%%%%%%%%%%%%%%%%%%%%%%%%%%%%
Here, we assume that $x_2<x_1$ (still under \eqref{portion3}, with $\delta_3>0$ small). We will show that the up-soliton vanishes here, while the right-soliton is of order $1$. We will control the equation from $\sl$, the loosing-time for the left and the down-solitons, up to the time $\smin$ defined in \eqref{defsmin0}, which is the minimum of three times:\\
- $\su$, the loosing time of the up-soliton, defined in \eqref{defSu};\\
- $\sur$, the relative-vanishing time of the up-soliton with respect to the right-soliton, defined in \eqref{defsur};\\
- $\srb$, the blowing-up time of the right-soliton, defined in \eqref{defsrb}.

\medskip
Introducing $\smin(A,x)$ such that
\begin{equation}\label{defsmin}
-e^{-\Smin}=T(x) - e^{-\smin}
\end{equation}
and referring to the ``size'' variable defined in the formal approach (see \eqref{deftDtm}), it is easy from its montonicity in time that up to $\Smin(A,x)$:\\
- the up-soliton has a size $\mup(x,s)\ge \frac 1{A+1}$;\\
- the right-soliton has a size $\mr(x,s) \le A$;\\
Since $\mup(x,s)\le \mr(x,s)$ by \eqref{order}, it follows that for all $s\in [\Sl(A,x), \Smin]$, we have
\begin{equation}\label{modere}
\frac 1{A+1}\le \mup(x,s)\le \mr(x,s) \le A.
\end{equation}
Introducing
\begin{equation}\label{defh***}
\hr(x,y,s) = \vc{w_{\tx}(\ty,\ts)}{\ps w_{\tx}(\ty,\ts)}
-\left[\kr(x,y,s)-\kup(x,y,s)\right]
\end{equation}
then arguing as for Proposition \ref{prop422}, we use the techniques of \cite{MZdmj12} to show thanks to \eqref{modere} that we can follow the two solitons present in Corollary \ref{cor2} up to time $s=\smin(x,A)$. More precisely, we claim the following:
%%%%%%%%%%%%%%%%%%%%%%%%%%%%%%%%%%%%%
%%%%%%%%%%%%%%%%%%%%%%%%%%%%%%%%%%%%%
\begin{prop}[Behavior of $w_x$ for $x\neq 0$, from 2 to 1 solitons]\label{prop221}
For all $A>0$, there exists $\delta_3(A)>0$ such that 
\[
\lim_{A \to \infty}\left(\sup_{\sl(A,x)\le s\le \smin(A,x)}\|\hr(x,s)\|_{\q H}\right)\to 0\mbox{ as }x\to 0
\]
with \eqref{portion3} satisfied and $r(x,y,s)$ recalled in \eqref{defh***}.
\end{prop}
%%%%%%%%%%%%%%%%%%%%%%%%%%%%%%%%%%%%%
%%%%%%%%%%%%%%%%%%%%%%%%%%%%%%%%%%%%%
\begin{proof} Again, as we did for Proposition \ref{prop422}, the proof follows from our techniques in \cite{MZdmj12}. Since we put some indications on the proof of Proposition \ref{prop422} in Appendix \ref{apptx}, we don't give details here, just to keep to paper into reasonable limits.
\end{proof}
Now, we claim that the up-soliton indeed vanishes at $s=\smin(A,x)$. More precisely, we have the following:
%%%%%%%%%%%%%%%%%%%%%%%%%%%%%%%%%%%%%
%%%%%%%%%%%%%%%%%%%%%%%%%%%%%%%%%%%
\begin{lem}[Only one soliton remains at $s=\smin(A,x)$]\label{lemonly1}For all $A>0$, the following holds:\\
(i) we have $\smin(A,x)= \su(A,x)$ and 
\begin{equation}\label{estonly1}
\left\|(w_x,\partial_s w_x)(y,\smin)-\kr(x,y,\smin)\right\|_{\q H} \le CA^{-\frac 1{p-1}};
\end{equation}
(ii) we have $\Su(A,x) \sim l$ as $x\to 0$, where $\su$ and $\Su$ are defined in
\eqref{defSu}.
\end{lem}
%%%%%%%%%%%%%%%%%%%%%%%%%%%%%%%%%%%%%
%%%%%%%%%%%%%%%%%%%%%%%%%%%%%%%%%%%%%
\begin{proof} Consider $A>0$.\\
(i) From the definition of the vanishing time for the up-soliton (see the formal approach), it is enough to prove that $\smin(A,x)= \su(A,x)$. Proceeding by contradiction, two cases may occur by defintion \eqref{defsmin} of $\smin(A,x)$. Let us find a contradiction in each case.\\
- If $\smin(A,x)= \srb(A,x)$, then the right-soliton is very large by the formal approach, and its energy can be shown to be negative. If we manage to prove that the up and the right soliton are decoupled, we may see that the energy of $w_x$ is also negative, which is a contradiction, by the blow-up criterion of Antonini and Merle \cite{AMimrn01}.\\
- If $\smin(A,x)=\sur(A,x)$, then we have  by definition of $\sur(A,x)$ (see the formal approach)
\[
\frac 1{\mup(x,\sur)} -\frac 1{\mr(x,\sur)}=2A.
\]
Given that
\begin{equation}\label{bornesA}
\frac 1A\le \frac 1{\mr(x,\sur)}\le \frac 1{\mup(x,\sur)}\le A+1,
\end{equation}
a contradiction follows for $A>1$.\\
(ii)  By definition 
\eqref{defSu}
of $\Su$, we see that $\Su(A,x) \le -\log|T(x)|$. Since we have from \eqref{pyramid2} $(1-\epsilon) x_1 \le - T(x) \le x_1 \sqrt 2$, hence $-\log|T(x)|\sim l$, we see that $\Su(A,x)\le l+o(l)$ as $x\to 0$, on the one hand.\\
On the other hand, using \eqref{devdbar}, we see that $\bar d(-\log|T(x)|) \le 0$, hence
\[
\frac{[\bar d(-\log|T(x)|) x_2-T(x)]}{Ac_0 l^{-\gamma}}-T(x) \le -\frac{2 T(x)l^\gamma}{A c_0} \le \frac{2 \sqrt 2 x_1l^\gamma}{A c_0}.
\]
By definition 
\eqref{defSu}
of $\Su$, this gives $\Su(A,x) \ge l +o(l)$ as $x\to 0$.
This concludes the proof of Lemma \ref{lemonly1}.
\end{proof}
%%%%%%%%%%%%%%%%%%%%%%%%%%%%%%%%%%%%
%%%%%%%%%%%%%%%%%%%%%%%%%%%%%%%%%%%%
\subsubsection{The blow-up limit outside the bisectrices and differentiability of the blow-up surface at non-characteristic points}
%%%%%%%%%%%%%%%%%%%%%%%%%%%%%%%%%%%%
%%%%%%%%%%%%%%%%%%%%%%%%%%%%%%%%%%%%

With Lemma \ref{lemonly1}, we are ready to use a modified version of the trapping result of \cite{MZtams14} to show that we have a blow-up limit for $w_x$, when $x$ is outside the bisectrices and small. In particular, assuming in addition that $x$ is non-characteristic, we will see that estimate \eqref{estonly1} persists after time $\smin(A,x)$. 

\medskip

Let us first give our modified trapping result:
%%%%%%%%%%%%%%%%%%%%%%%%%%%%%%%%%%%%
%%%%%%%%%%%%%%%%%%%%%%%%%%%%%%%%%%%%
\begin{lem}[Behavior of solutions of equation \eqref{eqw} near $\pm \kappa^*(d,\nu,y)$ and trapping in the non-characteristic case]\label{proptrap} 
There exist $\eb>0$, $K_0>0$ and $\mu_0>0$ such that for any $x^*\in \m R^2$, $s^*\ge -\log T(x^*)$, 
$\omega^*=\pm 1$, $|d^*|<1$ and $\nu^* \in \m R$, if
\begin{equation*}%\label{trap}
\ee\equiv\left \|\vc{w_{x^*}(s^*)}{\partial_s w_{x^*}(s^*)}-\bar\omega\kappa^*(d^*, \nu^*)\right\|_{\H}\le \eb,
\end{equation*}
then:\\
- either
\begin{equation}\label{c0}
 (w_{x^*}(s), \partial_s w_{x^*}(s))\to 0\mbox{ in }\q H\mbox{ as }s\to \infty,
\end{equation}
- or
\begin{equation}\label{conv0}
\forall s\ge s^*,\;\;\|(w_{x^*}(s), \partial_s w_{x^*}(s))- (\kappa(d_{\infty}),0)\|_{\q H}\le K_0\ee e^{-\mu_0 (s-s^*)},
\end{equation}
for some $d_\infty(x^*)\in B(0,1)$ satisfying
\begin{equation}\label{proxi0}
\frac{|l_{d^*}(d^*-d_\infty(x^*))|}{1-|d^*|}+
|\arg\tanh |d^*|-\arg\tanh |d_\infty(x^*)||
+\frac{|d^*-d_\infty(x^*)|}{\sqrt{1-|d^*|}}\le K_0\ee,
\end{equation}
where $l_{d^*}$ is the orthogonal projector on $\frac{d^*}{|d^*|}$ if $d^*\neq 0$, and on $e_1$, if $d^*=0$.\\
Moreover, we also have 
\begin{equation}\label{iii}
\left|\frac{\nu^*}{1-|d^*|}\right|\le K_0 \ee.
\end{equation}
If in addition $x^*\in \q R$, then only \eqref{conv0} holds, and $T$ is differentiable at $x=x^*$ with $\nabla T(x^*)=d_\infty(x^*)$.
\\In particular, if $\frac{\nu^*}{1-|d^*|}\ge 1$, then only \eqref{c0} holds, and $x^*\in \q S$.
\end{lem}
%%%%%%%%%%%%%%%%%%%%%%%%%%%%%%%%%%%%
%%%%%%%%%%%%%%%%%%%%%%%%%%%%%%%%%%%%
\begin{proof} The case $\nu^*=0$ was treated in \cite{MZtams14} and \cite{MZcmp14}. Although the first term in \eqref{proxi0} was not present in the statements given in those papers, it clearly follows from the proof, as we will briefly explain in Appendix \ref{apptx}.\\ When $\nu^*\neq 0$, we may use back and forth the similarity variables' transformation \eqref{defw}, changing the position of the rescaling time, in order to reduce to the case $\nu^* =0$. For details, see Appendix \ref{apptx}.
\end{proof}
From Lemma \ref{lemonly1}, we can apply Lemma \ref{proptrap} and obtain the following statement:
%%%%%%%%%%%%%%%%%%%%%%%%%%%%%%
%%%%%%%%%%%%%%%%%%%%%%%%%%%%%%
\begin{lem}[The blow-up limit in similarity variables]\label{lembl}
For any $x$ outside the bisectrices in a small neighborhood, the function $w_x(s)$ has a limit as $s\to \infty$, which is either $0$ or $\kappa(\hat d(x))$ \eqref{defkd}.
\end{lem}
%%%%%%%%%%%%%%%%%%%%%%%%%%%%%%
%%%%%%%%%%%%%%%%%%%%%%%%%%%%%%
\begin{proof}Take $x$ is outside the bisectrices such that \eqref{portion3} holds and consider $A>0$. From Lemma \ref{lemonly1} and the definition \eqref{foursol} of $\kup$, we see that the hypothesis of Lemma \ref{proptrap} is satisfied with $x^*=x$, $s^* = \su(A,x)$, $d^* = \bar d(\Su(A,x))e_1$, $\nu^* = \nr(x,\su(A,x))$ and $\ee \le e^{-A}$. Applying Lemma \ref{proptrap} gives the result.
\end{proof}

Now, if $x\in \q R$, then we get a more precise statement, together with some geometrical information on the blow-up set:
%%%%%%%%%%%%%%%%%%%%%%%%%%%%%%%%%%%%
%%%%%%%%%%%%%%%%%%%%%%%%%%%%%%%%%%%%
\begin{lem}[Estimate on $\nabla T(x)$ and $T(x)$ when $x$ is non-characteristic outside the bisectrices]\label{lemgrad}
If $x$ is outside the bisectrices and non-characteristic,  and \eqref{portion3} holds with $\delta_3$ small enough, then:\\
(i) $T$ is differentiable at $x$ and 
\[
\nabla T(x) =(-1+c_0l^{-\gamma}+o(l^{-\gamma}))e_1+o(l^{-\frac \gamma 2}) e_2 \mbox{ as }x\to 0,
\]
where $\gamma= \frac{p-1}2$ and $c_0=c_0(p)>0$;\\
(ii) it holds that
\begin{equation*}%\label{conv0}
\forall s\ge \su(A,x),\;\;\|(w_x(s), \partial_s w_x(s))- (\kappa(\nabla T(x)),0)\|_{\q H}\le K_0A^{-\frac 1{p-1}} e^{-\mu_0 (s-\su(A,x))};
\end{equation*}
(iii) we also have $T(x)\sim -x_1$ as $x\to 0$.
\end{lem}
%%%%%%%%%%%%%%%%%%%%%%%%%%%%%%%%%%%%
%%%%%%%%%%%%%%%%%%%%%%%%%%%%%%%%%%%%
\begin{proof}This time, we take a non-characteristic $x$ is outside the bisectrices such that \eqref{portion3} holds and consider $A$. As in the proof of Lemma \ref{lembl}, we know that Lemma \ref{proptrap} applies. Recalling that $x\in \q R$, we see that only \eqref{conv0} holds and that $T$ is differentiable at $x$. In particular, estimate \eqref{proxi0} holds with $d_\infty(x) = \nabla T(x)$. Projecting that estimate on the basis vectors $e_1$ and $e_2$, we see that
\begin{equation}\label{prox}
\frac{|\bar d(\Su(A,x))-\partial_{x_1}T(x)|}{1-|\bar d(\Su(A,x))|}
+\frac{|\partial_{x_2} T(x)|}{\sqrt{1-|\bar d(\Su(A,x))|}}\le CA^{-\frac 1{p-1}}.
\end{equation}
 Since
\begin{equation}\label{expd}
\bar d(\Su(A,x)) = -1+c_0 l^{-\gamma}+o(l^{-\gamma})
\end{equation}
from \eqref{devdbar} and item (ii) of Lemma \ref{lemonly1}, 
taking $A$ large enough, we obtain the conclusion of Lemma \ref{lemgrad}.\\
(ii) This is also a consequence of Lemma \ref{proptrap} (see \eqref{conv0}).\\
(iii) Applying estimate \eqref{iii} in Lemma \ref{proptrap}, we see by definition of $\nr(x,\su(A,x))$ given in \eqref{defs2} that
\[
\frac{|\bar d(\Su(x,A))x_1-T(x)|e^{\su(x,A)}}{1-|\bar d(x,A)|}\le CA^{-\frac 1{p-1}}.
\]
Since by definition of $\su(x,A)$, we have
\[
e^{-\su(x,A)}=\frac{|\bar d(-\log |T(x)|)x_2-T(x)|}{Ac_0l^{-\gamma}}\le \frac CAx_1l^\gamma,
\]
because $T$ is 1-Lipschitz, using \eqref{expd}, we see that
\[
|T(x)+x_1|\le CA^{-1-\frac 1{p-1}}x_1.
\le C\frac{e^{-A}}Ax_1.
\] 
Taking $A$ large enough yields the result.
\end{proof}

%%%%%%%%%%%%%%%%%%%%%%%%%%%%%%%%%%%%%
%%%%%%%%%%%%%%%%%%%%%%%%%%%%%%%%%%%%%
\section{A geometric approach for the stylized pyramid shape} \label{secreg}
%%%%%%%%%%%%%%%%%%%%%%%%%%%%%%%%%%%%%
%%%%%%%%%%%%%%%%%%%%%%%%%%%%%%%%%%%%%
Now, we are going to collect the dynamical information we got on $w_x(s)$ in the previous section for $x$ small enough, in order to obtain deep geometrical information on $T(x)$. 
For each step of the proof, the difficulty is to select the right $x$ where we will use the dynamical information we have on $w_x(s)$. That choice will result from a new technique, where we use a monotonic family of non-characteristic cones, that approach the blow-up graph from below.\\
We proceed in four subsections, one dedicated to points outside the bisectrices, the second to points on the bisectrices and the third to the origin. Finally, in the fourth subsection, we collect all the previous results to conclude the proof of Theorem \ref{mainth}.

%%%%%%%%%%%%%%%%%%%%%%%%%%%%%%%%%%%%%
%%%%%%%%%%%%%%%%%%%%%%%%%%%%%%%%%%%%%
\subsection{Points outside the bisectrices are non-characteristic}\label{subprop3}
%%%%%%%%%%%%%%%%%%%%%%%%%%%%%%%%%%%%%
%%%%%%%%%%%%%%%%%%%%%%%%%%%%%%%%%%%%%
This is the aim of the section:
%%%%%%%%%%%%%%%%%%%%%%%%%%%%%%%%%%%%%
%%%%%%%%%%%%%%%%%%%%%%%%%%%%%%%%%%%%%
\begin{prop}\label{lem3.1}
All points $|x|\le \delta$ outside the bisectrices for some $\delta>0$
are non-characteristic. 
\end{prop}
%%%%%%%%%%%%%%%%%%%%%%%%%%%%%%%%%%%%%
%%%%%%%%%%%%%%%%%%%%%%%%%%%%%%%%%%%%%
Before giving the proof, let us derive from this result and Lemma \ref{lemgrad}  the following sharp estimates for $T(x)$ and $\nabla T(x)$ outside the bisectrices:
%%%%%%%%%%%%%%%%%%%%%%%%%%%%%%%%%%%%
%%%%%%%%%%%%%%%%%%%%%%%%%%%%%%%%%%%%
\begin{cor}[Estimate on $\nabla T(x)$ and $T(x)$ when $x$ is outside the bisectrices]\label{lemgrad2}
If $x$ is outside the bisectrices and small enough, with $0\le x_2<x_1$, then:\\
(i) $T$ is differentiable at $x$ and 
\[
\nabla T(x) =(-1+c_0l^{-\gamma}+o(l^{-\gamma}))e_1+o(l^{-\frac \gamma 2}) e_2 \mbox{ as }x\to 0,
\]
where $\gamma= \frac{p-1}2$ and $c_0=c_0(p)$;\\
(ii) it holds that
\begin{equation*}%\label{conv0}
\forall s\ge \su(A,x),\;\;\|(w_x(s), \partial_s w_x(s))- (\kappa(\nabla T(x)),0)\|_{\q H}\le K_0A^{-\frac 1{p-1}} e^{-\mu_0 (s-\su(A,x))}.
\end{equation*}
(iii) we also have $T(x) = -x_1(1-c_0l^{-\gamma}+o(l^{-\gamma}))+o(x_2 l^{-\frac \gamma 2})$ as $x\to 0$.
\end{cor}
%%%%%%%%%%%%%%%%%%%%%%%%%%%%%%%%%%%%
%%%%%%%%%%%%%%%%%%%%%%%%%%%%%%%%%%%%
\begin{proof}[Proof assuming Proposition \ref{lem3.1} holds] 
Take a small $x$ outside the bisectrices with $0\le x_2 <x_1$ and $x_1$ small.\\
%Take a small $x$ outside the bisectrices. From the symmetries of the solution (see Proposition \ref{propyr}), we may assume that $0\le x_2 <x_1$ with $x_1$ small.\\
(i)-(ii) From Proposition \ref{lem3.1}, we know that $x\in \RR$. Therefore, Lemma \ref{lemgrad} applies and items (i) and (ii) hold.\\
(iii) If $x_2=0$, recalling that $T(0)=0$ from Proposition \ref{propyr},  we write $$T(x) = \int_0^{x_1} \partial_{x_1}T(\xi,0)d\xi.$$ Using item (i), we get the conclusion.\\
Now, if $x_2>0$, we write $T(x) = T(x_1,0)+\int_0^{x_2} \partial_{x_2}T(x_1,\xi)d\xi$, and the conclusion follows again from item (i), together with the case $x_2=0$.
\end{proof}

The proof of Proposition \ref{lem3.1} follows from some introductory results which hold for any blow-up solution of equation \eqref{equ}. Let us give them first (Step 1), then give the proof of Proposition \ref{lem3.1} (Step 2).

\bigskip

{\bf Step 1:
Preliminary
% General
 blow-up results for equation \eqref{equ}}

In this step, we first prove the following statement:
%%%%%%%%%%%%%%%%%%%%%%%%%%%%%%%%%%%%
%%%%%%%%%%%%%%%%%%%%%%%%%%%%%%%%%%%%
\begin{lem}Consider $x_0$ and $x_1$ such that the segment 
$[(x_0,T(x_0)), (x_1,T(x_1))]\subset \Gamma$ 
and has slope $-1$. Then, 
for any $\tau\in [0,1)$, 
$x_\tau\equiv (1-\tau) x_0+\tau x_1 \in \q S$.
\end{lem}
%%%%%%%%%%%%%%%%%%%%%%%%%%%%%%%%%%%%
%%%%%%%%%%%%%%%%%%%%%%%%%%%%%%%%%%%%
\begin{proof} Up to changing $x_1$, we may assume the segment
  ``maximal'', in the sense that for any $\tau>1$, the segment
  $[(x_0,T(x_0)), (x_\tau,T(x_\tau))]$ either has a slope different
  from $-1$ (hence larger than $-1$ from the finite speed of
  propagation), or is not included in
$\Gamma$.\\
%  $\q S$.\\
If $\tau \in [0,1)$, then the segment $[(x_\tau, T(x_\tau), (x_1, T(x_1)]$ has an empty interior and slope $-1$, which means that it is inside any cone of the form $\q C_{x_\tau, T(x_\tau), 1-\beta}$ with arbitrary $\beta\in (0,1)$. Thus, by definition, $x_\tau\in \q S$.
\end{proof}

Now, we prove the following lemma concerning stationary solutions of equation \eqref{eqw} in the energy space $\q H$.
%%%%%%%%%%%%%%%%%%%%%%%%%%%%%%%%
%%%%%%%%%%%%%%%%%%%%%%%%%%%%%%%%
\begin{lem}[A positive lower bound on the norm of non-zero stationary solutions of equation \eqref{eqw} in the energy space $\q H$]\label{lemlb}
There exist $\epsilon_0>0$ such that for any $w^* \in \q H$ stationary solution of equation \eqref{eqw}, we have
\[
\|w^*\|_{\q H} \ge \epsilon_0\|w^*\|_{L^{p+1}_\rho} \ge \epsilon_0^2.
\]
\end{lem}
%%%%%%%%%%%%%%%%%%%%%%%%%%%%%%%%
%%%%%%%%%%%%%%%%%%%%%%%%%%%%%%%%
\begin{proof}
Consider $w^* \in \q H$ a stationary solution of equation \eqref{eqw}. Multiplying the stationary version by $w^*\rho$ and integrating in space, we see that
\begin{align*}
&\int_{|y|<1} \left[ |\nabla w^*|^2 - (y\cdot \nabla w^*)^2\right]\rho(y) dy + \frac{2(p+1)}{(p-1)^2} \int_{|y|<1}|w^*(y)|^2 \rho(y)dy\\
=&\frac 1{p+1} \int_{|y|<1}|w^*(y)|^{p+1} \rho(y)dy,
\end{align*}
which yields by definition of the norm in $\q H$ \eqref{defnh}
\[
\|w^*\|_{\q H}^2 \le C \|w^*\|_{L^{p+1}_\rho}^{p+1},
\]
on the one hand.\\
On the other hand, using the Hardy-Sobolev inequality given in Lemma \ref{lemhs}, we see that
 \[
\|w^*\|_{L^{p+1}_\rho}\le C\|w^*\|_{\q H}.
\]
Combining the two inequalities and choosing a non-zero $w^*$ yields the result.  
\end{proof}

\bigskip

{\bf Step 2: Proof of Proposition \ref{lem3.1}}

Using the statements of Step 1, we are ready to give the proof of Proposition \ref{lem3.1}.

\medskip

\begin{proof}[Proof of Proposition \ref{lem3.1}]
Let us first recall the following result we proved in Lemma
\ref{lemgrad} (please remember the symmetries of the solution): for some $\delta_3>0$, 
$T$ is differentiable at $x^*$, whenever $x^*\in \RR$, $|x^*_i|> |x^*_j|$ for some $i,j\in\{1,2\}$ and $|x^*|\le \delta_3$, with
\begin{equation}\label{gradtx*}
|\partial_{x_i}T(x^*) +\sgn(x^*_i)(1-c_0(l^*)^{-\gamma })e_i|\le \frac{c_0}2(l^*)^{-\gamma}\mbox{ and }|\partial_{x_j} T(x^*)|\le C(l^*)^{-\frac \gamma 2} ,
\end{equation}
with $l^*=-\log|x^*_i|$. 

\medskip

Now, in order to prove Proposition \ref{lem3.1}, we proceed by contradiction, and assume that we have a sequence $x_m\to 0$ as $m\to \infty$ with $|x_{m,1}|\neq |x_{m,2}|$ such that
\begin{equation}\label{xms}
x_m\in \q S.
\end{equation}
Using the symmetries of the solution as before, 
we may assume that 
\begin{equation}\label{portion4}
0\le x_{m,2}<x_{m,1}\le \delta_4,
\end{equation}
for some $\delta_4\in(0,\delta_3)$.
Our aim is to find some contradiction.

\medskip

Consider $m\in \m N$. For each $\beta \in (0,x_{m,1}^2]$, we consider a family of cones $\q C_{x_m,t,1-\beta}$ defined in \eqref{defcone} and indexed by $t\le T(x_m)$. 
Since $T(x)\ge -1$ from Proposition \ref{prop1}, we can define
\begin{equation}\label{deft*}
t^*_m=t^*_m(\beta) = \sup\{0\le t<T(x_m)\;|\;\q C_{x_m,t,1-\beta}\cap 
\Gamma
= \emptyset\}.
\end{equation}
By continuity and maximality, 
$\Gamma$
 is above the cone $\q C_{x_m,t^*_m,1-\beta}$, and they touch each other at some point $(x^*_m, T(x^*_m))$ for some $x^*_m=x^*_m(\beta)$ with
\begin{equation}\label{btm}
t^*_m=t^*_m(\beta)\in [-1,T(x_m)].
\end{equation}
 Since the cone $\q C_{x^*_m,T(x^*_m),1-\beta}$ is inside the cone $\q C_{x_m,t^*_m,1-\beta}$, it is also below 
$\Gamma$.
Since $1-\beta<1$, this means that 
$x_m^*(\beta)\in \RR$.
Since $x_m\in \q S$ by \eqref{xms}, it follows that 
\begin{equation}\label{diff}
x_m^*(\beta)\neq x_m.
\end{equation}
 Several cases then arise:\\
{\bf Case 1}:
There is a subsequence (still denoted by $x_m$) such that for each $m\in \m N$, there exists $\beta_m\in (0,x_{m,1}^2]$ such that $x_m^*(\beta_m)$ is outside the bisectrices and $|x^*_m(\beta_m)|\le \delta_3$.\\
If we consider the intersection of the cone $\q C_{x_m,t^*_m,1-\beta_m}$ and 
$\Gamma$
with the plane orthogonal to the plane $\{t=-1\}$ and passing through the points $(x_m,T(x_m))$ and $(x^*_m, T(x^*_m))$ (remember that $x_m\neq x^*_m$ from \eqref{diff}), we see that the trace of the cone stays under the trace of 
$\Gamma$,
with contact point $(x^*_m,T(x^*_m))$. Therefore, their slopes have to be equal, namely
\[
1-\beta_m=-\nabla T(x^*_m) \cdot \frac{x^*_m-x_m}{|x^*_m-x_m|}\le |\partial_{x_i} T(x^*_m)|\le 1-\frac{c_0}2(l^*_m)^{-\gamma},
\]
where 
$l^*_m=-\log |x^*_{m,i}|$,
 $i\in \{1,2\}$ is such that $|x^*_{m,i}|>|x^*_{m,j}|$, and
where we have used \eqref{gradtx*} for the last inequality.
This implies in particular that
\begin{equation}\label{1h}
x_{m,1}^2\ge \beta_m\ge \frac{c_0}2(l^*_m)^{-\gamma},\mbox{ hence, }
x^*_{m,i} =o(x_{m,1})\mbox{ as }m\to \infty,
\end{equation}
on the one hand. On the other hand, using again the fact that the cone $\q C_{x_m,t^*_m,1-\beta_m}$ is below 
$\Gamma$,
we write from the bounds on the blow-up 
surface
given in Proposition \ref{prop1} (with $\epsilon =\frac 12$) for $m$ large enough:
\begin{equation}\label{2h}
-\frac{x_{m,1}}2\ge T(x_m) \ge t^*(x_m) \ge t^*(x_m)-|x^*_m-x_m|(1-\beta_m)=T(x^*_m) \ge -\sqrt 2 |x^*_{m,i}|.
\end{equation}
Since $x_m \to 0$ as $m\to \infty$, a contradiction follows from \eqref{1h} and \eqref{2h}.\\
{\bf Case 2}: For all $m\ge m_0$ for some $m_0\in \m N$, 
for all $\beta\in (0,x_{m,1}^2]$, 
either $x^*_m(\beta)$ is on the bisectrices or $|x^*_m(\beta)|> \delta_3$. 
Two subcases arise again:\\
{\bf Case 2.a}: We can extract a subsequence (still denoted by $x_m$) such that for all $m\in \m N$, there exists $\beta_m\in (0, x_{m,1}^2]$ such that $|x^*_m(\beta_m)|>\delta_3$. In order to simplify the notation, we omit the dependence on $\beta_m$ in the remaining part of Case 2.a, and write (for example) $x^*_m$ instead of $x^*_m(\beta_m)$. We claim that
\begin{equation}\label{b2}
|x^*_m|\le 2
\end{equation}
for $m$ large enough. Indeed, since $(x^*_m, T(x^*_m))\in \q C_{x_m,
  t^*_m, 1-\beta_m}$, it follows that
$|x^*_m-x_m|=\frac{t^*_m-T(x^*_m)}{1-\beta_m}\le \frac{T(x_m)-T(x^*_m)}{1-\beta_m}$. Since $T(x_m) \to T(0)=0$ and $\beta_m\to 0$ as $m\to \infty$, recalling that $T(x) \ge -1$ by construction in Proposition \ref{prop1}, we see that \eqref{b2} follows. Similarly, we get from \eqref{btm} that $-1\le t^*_m\le 1$ for $m$ large. Therefore, up to extracting a subsequence, we may assume that $x^*_m \to \bar x$ and $t^*_m\to \bar t$ as $m\to \infty$, for some $\bar t\le T(0)$ and
\begin{equation}\label{bigxb}
|\bar x|\ge \delta_3.
\end{equation}
 Since $(x_m^*, T(x_m^*))\in \q C_{x_m,t_m^*, 1-\beta_m}$ and $x_m\to 0$ as $m\to \infty$, it follows that $(\bar x, T(\bar x))\in \q C_{0,\bar t,1}$, therefore, $\bar t=T(0)$ since $T$ is $1$-Lipschitz (this follows from the finite speed of propagation). Thus,
\begin{equation}\label{001}
(\bar x, T(\bar x))\in \q C_{0,0,1}.
\end{equation}
Again, two cases arise:\\
{\bf Case 2.a.i}: $\bar x$ is not on the bisectrices. Using \eqref{001}, we see that $T(\bar x) = -|\bar x|$. Up to making a rotation of coordinates, we may assume that $T(\bar x) = -|\bar x|=-\bar x_1$. Using the soliton-loosing mechanism we presented in Section \ref{secloose}, we see that for time 
\begin{equation}\label{emta}
s=\hat s = -\log(\bar x_1-\bar x_2)-\frac{p-1}2 \log|\log \bar x_1|+B
\end{equation}
with $B>0$ large, the up-soliton vanishes, and so do the left and down
solitons, thanks to the hierarchy we noted in \eqref{order0} and
\eqref{order}. More precisely, we have
\begin{equation}\label{hierch}
\|\kappa^*(-\bar d(\bar s)e_1, \bar \nu) \|_{\q H}
       +\|\kappa^*(-\bar d(\bar s)e_2, \bar \nu) \|_{\q H}
       \le 2 \|\kappa^*(\bar d(\bar s)e_2, \bar \nu)\|_{\q H} \le 2 \epsilon_0
  \end{equation}
for some small $\epsilon_0\to 0$ as $B\to \infty$, where 
\[
\bar \nu = \frac{(\bar d(\bar s)+1)\bar x_1 e^B}{(\bar x_1-\bar
  x_2)|\log \bar x_1|^{\frac{p-1}2}}\mbox{ and } -e^{-\bar s}=T(x) -
e^{-\hat s}
\]
from \eqref{defs2}. Therefore, only the right soliton defined in \eqref{foursol} remains, and we see that  
\begin{equation}\label{wbx}
\left\|\vc{w_{\bar x}(\hat s)}{\partial_s w_{\bar x}(\hat s)} -\kappa^*(\bar d(\bar s)e_1, \bar \nu)\right\|_{\q H}\le 4\epsilon_0.
\end{equation}
The justification of \eqref{wbx} is
straightforward. Indeed, we first write
\begin{align*}
  &\left\|\vc{w_{\bar x}(\hat s) }{\partial_s w_{\bar x}(\hat s)}
  -\kappa^*(\bar d(\bar s)e_1, \bar \nu)\right\|_{\q H}\\
  \le& \left\|\vc{w_{\bar x}(\hat s) }{\partial_s w_{\bar x}(\hat s)}
       -\left[\kappa^*(\bar d(\bar s)e_1, \bar \nu)
       -\kappa^*(\bar d(\bar s)e_2, \bar \nu)
       +\kappa^*(-\bar d(\bar s)e_1, \bar \nu)
       -\kappa^*(-\bar d(\bar s)e_2, \bar \nu)
       \right]\right\|_{\q H}\\
  &+\|\kappa^*(\bar d(\bar s)e_2, \bar \nu)\|_{\q H}
       +\|\kappa^*(-\bar d(\bar s)e_1, \bar \nu) \|_{\q H}
       +\|\kappa^*(-\bar d(\bar s)e_2, \bar \nu) \|_{\q H}.  
\end{align*}
Note that this last line is bounded by $3\epsilon_0$ thanks to
\eqref{hierch}. As for the intermediate line, we will transform it
using the algebraic transformation \eqref{defw**} which we recall here
with our current notations
\[
  \tT_{\bar x}(v)(\ty, \hat s)=(1-T(\bar x) e^{\hat s})^{-\frac
    2{p-1}} v(\cy,\hat s),\;\;
\cy =\frac{\ty +\bar x e^{\hat s}}{1-T(\tx) e^{\ts}},\;\;
\bar s= \hat s -\log(1-T(\bar x) e^{\hat s}).
\]
Note in particular that we have
\[
\tT_{\bar x}(w_0)= w_{\bar x}\mbox{ and }T_{\bar x}(\kappa(\pm \bar
d(\bar s) e_i,Y),0)(y,\hat s) = \kappa^*(\pm\bar d(\bar s) e_i, y).
\]
Therefore, using this transformation and straightforward computations,
we write
\begin{align}
&\left\|\vc{w_{\bar x}(\hat s) }{\partial_s w_{\bar x}(\hat s)}
       -\left[\kappa^*(\bar d(\bar s)e_1, \bar \nu)
       -\kappa^*(\bar d(\bar s)e_2, \bar \nu)
       +\kappa^*(-\bar d(\bar s)e_1, \bar \nu)
       -\kappa^*(-\bar d(\bar s)e_2, \bar \nu)
                 \right]\right\|_{\q H}\nonumber\\
  \le &
      C  \left\|\vc{w(\bar s)}{\ps w(\bar s)} - \left[\vc{\kappa(\bar d(\bar
        s)e_1)}{0}-\vc{\kappa(\bar d(\bar s)e_2)}{0}+\vc{\kappa(-\bar d(\bar
        s)e_1)}{0}-\vc{\kappa(-\bar d(\bar
        s)e_2)}{0}\right]\right\|_{\H(|y|<1)}\nonumber\\
  \le & C \left\|\vc{w(\bar s)}{\ps w(\bar s)} - \left[\vc{\kappa(\bar d(\bar
        s)e_1)}{0}-\vc{\kappa(\bar d(\bar s)e_2)}{0}+\vc{\kappa(-\bar d(\bar
        s)e_1)}{0}-\vc{\kappa(-\bar d(\bar
        s)e_2)}{0}\right]\right\|_{\H},\label{akhir}
  \end{align}
where the norm in $\q H(|y|<1)$ is the same as the norm in $\q H$
defined in \eqref{defnh} with integrals considered only on the  ball
$\{|y|<1\}$ instead of the unit ball $\{|Y|<1\}$. Note that 
\begin{equation}\label{inclus}
\{|y|<1\}\subset\{|Y|<1\}
\end{equation}
by \eqref{001}. Since $(w,\partial_s w)$ is close to the sum of the 4
solitons by construction (see estimate \eqref{cprofile} in Proposition
\ref{propw}), we may make \eqref{akhir} small enough, say less than
$\epsilon_0$. Gathering all these estimates, we obtain estimate
\eqref{wbx}.

\medskip

We would like to insist on the fact that estimate \eqref{001} is crucial here, since it implies estimate \eqref{inclus}, which is not available in the general case and constrained us in Section \ref{secfor} to adopt a formal approach in order to recover formally the intformation on the set $\{|y|<1;|Y|\ge 1\}$. Here, this latter set is empty, and the algebraic transformation \eqref{defw**} is enough to yield the result. 

\medskip

 Since $\bar t<T(\bar x)$, it follows that $\bar s\le -\log (-T(\bar x))$. Since $T(\bar x) = -|\bar x|\le -\delta_3$ by \eqref{001} and \eqref{bigxb}, it follows that, 
\begin{equation}\label{sblocked}
\bar s<-\log\delta_3.
\end{equation} 
Since $x_m^* \to \bar x$ as $m\to \infty$, by continuity, we have
\[
\left\|\vc{w_{x^*_m}(-\log(T(x^*_m)-\bar t))}{w_{x^*_m}(-\log(T(x^*_m)-\bar t))}-\kappa^*(\bar d(\bar s)e_1, \bar \nu)\right\|_{\q H}\le 5\epsilon_0,
\]
for $m$ large enough. Taking $s_0$ large enough, we can make $\epsilon_0$ as small as we want. Since $x^*_m\in \q R$, the trapping result applies (see Lemma \ref{proptrap} above) and we see that 
\[
w_{x^*_m}(s) \to \kappa(\nabla T(x^*_m))\mbox{ as } s \to \infty
\]
with
\begin{equation}\label{proche}
|\nabla T(x^*_m)- \bar d(\bar s)e_1|\le C\epsilon_0.
\end{equation}
Since $\bar s$ is blocked far from $\infty$ by \eqref{sblocked}, the same holds for $\bar d(\bar s)$ by definition \eqref{solpart}:
\[
\bar d(\bar s) \ge -1+\frac{c_0}2 \bar s^{-\frac{p-1}2}\ge -1 +\frac{c_0}2 |\log \delta_3|^{-\frac{p-1}2}.
\]
Not also that we have $\bar s \ge s_0$, hence
\[
\bar d(\bar s) \le \bar d(s_0)\le -1+2c_0s_0^{-\frac{p-1}2},
\]
still from \eqref{solpart}.\\
Making $s_0$ large enough, we can make $\epsilon_0$ small enough, and get from \eqref{proche}
\[
|\partial_{x_1}T(x^*_m)|\ge 1-\delta_0'\mbox{ where }\delta_0=\min(2c_0 s_0^{-\frac{p-1}2}, \frac{c_0}4 |\log \delta_3|^{-\frac{p-1}2})\mbox{ and }|\partial_{x_2}T(x^*_m)|\le \epsilon_0
\]
on the one hand.\\
On the other hand, since the point $(x^*_m,T(x^*_m))$ is the contact point between the cone $\q C(x_m,t^*_m,1-\beta_m)$ and 
$\Gamma$
and that the cone is below 
$\Gamma$
elsewhere, it follows that their slopes have to be equal:
\[
1-\beta_m = -\nabla T(x^*m)\cdot \frac{x^*_m-x_m}{|x^*_m-x_m|} \le 
|\partial_{x^*_{m,1}}T(x^*_m)|\le 1- \delta_0'.
\] 
Making $m\to \infty$, we remember that $\beta_m \to 0$, and we get a contradiction.\\
{\bf Case 2.a.ii}: $\bar x$ is on the bisectrices. From symmetry, we can assume that $\bar x_1=\bar x_2$. Then, for any $\bar t<T(\bar x)$, in the section of the light cone with vertex $(\bar x, T(\bar x))$, we always see two solitons, with opposite signs and parameters $\bar d(\bar s) e_1$ and $\bar d(\bar s) e_2$. Again, $\bar s$ is blocked as in \eqref{sblocked}. Since $T(\bar x)=-|\bar x|\le -\delta_3$ by \eqref{bigxb}, we are far from the expected value, namely $-\bar x_1$ (since $\bar x_1=\bar x_2$, we have $|\bar x|= \bar x_1 \sqrt 2\ge \delta_3$). A simple calculation shows that
\[
\|(w_{\bar x}, \partial_s w_{\bar x})(-\log(T(\bar x)-\bar t))\|_{\q H}\mbox{ is small.}
\]
By continuity, for $m$ large enough, we get $\|w_{x^*_m}(-\log (T(x^*_m)-\bar t))\|_{\q H}$ is small too. Since $x^*_m$ is non characteristic, we get a contradiction with the non-degeneracy of blow-up limits at non-characteristic points (see Proposition \ref{propnondeg}).\\
{\it Conclusion}. If Case 2.a holds, then a contradiction all the time. \\
\noindent {\bf Case 2.b}: For all $m\ge m_1$ for some $m_1\in \m N$, for all $\beta \in (0, x_{m,1}^2]$, $x^*_m(\beta)$ is on the bisectrices and $|x^*_m(\beta)|\le \delta_3$. Here, $m$ will remain fixed, and we will make $\beta\to 0$ in order to find a contradiction. In order to simplify the notation, we drop-down the subscript $m$ and write $x$ and $x^*(\beta)$ instead of $x_m$ and $x^*_m(\beta)$.\\
Up to extracting a subsequence $\beta_n\to 0$ as $n\to \infty$, we assume that $x^*(\beta_n)$ converges to some $\bar x$ on the bisectrices with
\begin{equation}\label{taillexbar}
|\bar x|\le \delta_3.
\end{equation}
 Then, we have the following:
%%%%%%%%%%%%%%%%%%%%%%%%%%%%%%
%%%%%%%%%%%%%%%%%%%%%%%%%%%%%%
\begin{cl}\label{clcone}$ $\\
(i) It holds that $\bar x \in \q C_{x, T(x),1}$.\\
(ii) The segment $[(x,T(x)), (\bar x, T(\bar x))]$ 
belongs to $\Gamma$.
\end{cl}
%%%%%%%%%%%%%%%%%%%%%%%%%%%%%%
%%%%%%%%%%%%%%%%%%%%%%%%%%%%%%
\begin{proof}$ $\\
(i) Since $0\le t^*(\beta_n)<T(x)$ by definition \eqref{deft*}, we may assume that $t^*(\beta_n) \to \bar t\in [0, T(x)]$ as $n\to \infty$, up to extracting a subsequence. Since $x^*(\beta_n)$ is on the cone $\q C_{x,t^*(\beta_n), 1-\beta_n}$, making $n\to \infty$, we see that $\bar x \in \q C_{x, \bar t,1}$.\\
Now, assume by contradiction by $\bar t<T(x)$, it follows that the slope between $(\bar x, T(\bar x))$ and $(x,T(x))$ is strictly larger than the slope between $(\bar x, T(\bar x))$ and $(x,\bar t)$, which is $1$, since $\bar x \in \q C_{x, \bar t,1}$. This is a contradiction, since $1$ is the Lipschitz constant of $T$. Thus, $\bar t= T(x)$, and item (i) follows.\\
(ii) Assume by contradiction that for some $\tau \in (0,1)$, we have $T(x_\tau) \neq \tau T(\bar x) +(1-\tau) T(x)$, where $x_\tau = \tau \bar x +(1-\tau) x$.\\
If $T(x_\tau) > \tau T(\bar x) +(1-\tau) T(x)$, then the Lipschitz constant of $T$ on the segment $(\bar x, x_\tau)$ has to be larger than $1$, and this is a contradiction. If $T(x_\tau)$ is less than that, then we have a similar contradiction on the segment $(x_\tau, x)$. Thus, $T(x_\tau) = \tau T(\bar x) +(1-\tau) T(x)$ and item (ii) follows.\\
\end{proof}
Following this claim, 
for each $\gamma \in (0,\bar x_1^2]$, we consider a family of cones $\q C_{\bar x, t, 1-\gamma}$ indexed by $t\in [0, T(\bar x))$. As we did earlier with the cones $\q C_{x,t, 1-\beta}$, we can select $\tilde t=\tilde t(\gamma)$ the highest such that the cone $\q C_{\bar x,\tilde t, 1-\gamma}$ is under 
$\Gamma$
and touches it at some point $(\tilde x, T(\tilde x))$ where $\tilde x=\tilde x(\gamma)$. By construction, $\tilde x \in \q R$.\\
Then, we have two cases:\\
- {\bf Case 2.b.i}: There exists $\gamma\in (0,x_1^2]$ such that $\tilde x(\gamma) = \bar x$. In this case $\bar x \in \q R$.
 Using our techniques in \cite{MZjfa07} (in particular the existence of a Lyapunov functional; see Theorem 2 page 47 in that paper), we have
\begin{equation*}%\label{conv}
w_{\bar x}(\sigma_n) \to \bar w\mbox{ in }\q H,\mbox{ as }n\to \infty
\end{equation*}
 for some $\bar w\in \q H$, stationary solution of equation \eqref{eqw} and some sequence $\sigma_n\to \infty$. We also have by the same techniques
\begin{equation}\label{conv}
\sup_{\sigma\in [0,1]}\|w_{\bar x}(\sigma+\sigma_n)-\bar w\|_{\q H} \to 0\mbox{ as }n\to \infty.
\end{equation}
Thanks to the following algebraic identity linking $w_x$, $u$ and $w_{\bar x}$:
\begin{equation}\label{defwx}
 e^{\frac{2s_n}{p-1}} w_x(y,s) = u(\xi,t)= e^{\frac{2\sigma_n}{p-1}} w_{\bar x}(z,\sigma), 
\end{equation}
 where 
\begin{equation}\label{deftn}
T(x)-e^{-s}=t=T(\bar x) -e^{-\sigma}\mbox{ and }
\frac{y-x}{T(x)-t}=\xi = \frac{z-\bar x}{T(\bar x)-t},
\end{equation}
estimate \eqref{conv} translates into some estimate for $w_x$, which will turn to be in contradiction with the following bound we proved in \cite{MZma05} and \cite{MZjfa07}:\footnote{Note that $t=-1$ is our initial time by construction in Proposition \ref{propyr}, which explains why the following property shows the logarithm of $T(x)+1$} For all $s\ge - \log (T(x)+1)+1$, 
\begin{equation}\label{est}
\int_s^{s+1}\int_{|z|<1}|w_x(z,s')|^{p+1}\rho(z) dz ds' \le C(\bar x),
\end{equation}
for some $C(x)>0$, and this will imply that the case 2.b.i doesn't occur. More precisely, the contradiction will follow if we prove that 
\begin{equation}\label{hadaf}
I_n=\int_{s_n}^{s_n'}\int_{|y|<1} |w_x(y,s)|^{p+1}\rho(y) dyds \to \infty\mbox{ as }n\to \infty,
\end{equation}
where 
\[
T(x)-e^{-s_n}=t_n=T(\bar x) -e^{-\sigma_n}\mbox{ and }
T(x)-e^{-s_n'}=t_n'=T(\bar x) -e^{-\sigma_n-1}
\]
(note in particular that $s_n \to -\log(T(x) - T(\bar x))$ and $s_n'-s_n\to 0$ ; 
given that $x=x_m \to 0$ as $m\to \infty$ and $|\bar x|\le \delta_3$ by \eqref{portion4} and \eqref{taillexbar},
and $T$ is 1-Lipschitz, we clearly see that \eqref{hadaf} is in contradiction with \eqref{est}), if $\delta_4$ is small enough).

\medskip

Let us now give the details for \eqref{hadaf}. In fact, our calculation is based on a similar calculation performed in the proof of Proposition \ref{prop1} given in Section \ref{subprop1}.\\
Since we have
\[
e^{-1} \le \frac{T(x)-t}{T(x)-t_n}\le 1\mbox{ whenever } t_n \le t\le t_n', 
\] 
using \eqref{defwx} and getting rid of the weight $\rho(y)$, we write
\begin{align}
I_n &= \int_{s_n}^{s_n'} e^{-\frac{2(p+1)s}{p-1}}\int_{|y|<1}|u(x+ye^{-s}, T(x)-e^{-s})|^{p+1} \rho(y) ds\nonumber\\
&= \int_{t_n}^{t_n'}(T(x)-t)^{2\alpha}\int_{|\xi-x|<T(x)-t}|u(\xi,t)|^{p+1} \rho\left(\frac{\xi-x}{T(x)-t}\right) d\xi dt\nonumber\\
&\ge e^{-2\alpha}\delta_n^\alpha (T(x)-t_n)^{2\alpha} \int_{t_n}^{t_n'}\int_{|\xi-x|<(1-\delta_n)(T(x)-t)}|u(\xi,t)|^{p+1} d\xi dt\label{ahmed}
\end{align}
where $\alpha$ is introduced in \eqref{defro}, and the parameter $\delta_n\in (0,1)$ is arbitrary and will be chosen later.  Since $\bar x\in \q C_{x,T(x),1}$ by item (i) of Claim \ref{clcone}, it follows that the cone $\q C_{\bar x, T(\bar x),1}$ is under the cone $\q C_{x,T(x),1}$. In particular, at any time $t\in [t_n,t_n']$, the section of the inner cone is included in the section of the outer, in other words, $B(\bar x, T(\bar x) -t) \subset B(x, T(x)- t)$. Shrinking a bit these sections, we may ask to have for all $t\in [t_n,t_n']$, 
\[
B(\bar x,(1-\bar \delta_n)( T(\bar x) -t)) \subset B(x,(1-\delta_n) (T(x)- t)), 
\]
i.e.
\[
|x-\bar x|+(1-\bar \delta_n)(T(\bar x) - t)\le (1-\delta_n) (T(\bar x) -t),
\]
and this is possible if the inequality is true at $t=t_n'$, i.e. when 
\begin{equation}\label{defbd}
\bar \delta_n = \delta_n\frac{(T(\bar x)-t_n')+|x-\bar x|}{T(\bar x)-t_n'}
\end{equation}
(remember that $T(x)-T(\bar x) = |x-\bar x|$ from item (i) of Claim \ref{clcone}).
Therefore, with this choice of $\bar \delta_n$,  we write from \eqref{ahmed} and \eqref{defwx}
\begin{align}
I_n &\ge e^{-2\alpha}\delta_n^\alpha (T(x) - t_n)^{2\alpha} \int_{t_n}^{t_n'}\int_{|\xi-\bar x|<(1-\bar \delta_n)(T(\bar x) -t)}|u(\xi, t)|^{p+1}d\xi dt\nonumber\\
&\ge e^{-2\alpha}\delta_n^\alpha  (T(x) - t_n)^{2\alpha}\int_{\sigma_n}^{\sigma_n+1}(T(\bar x) -t)^{-2\alpha}\int_{|z|\le (1-\bar \delta_n)}
|w_{\bar x}(z,\sigma)|^{p+1}\rho(z) dzd\sigma\\
&\ge e^{-2\alpha}\delta_n^\alpha \left(\frac{T(x) - t_n}{T(\bar x) - t_n}\right)^{2\alpha}\int_{\sigma_n}^{\sigma_n+1}\int_{|z|\le (1-\bar \delta_n)}
|w_{\bar x}(z,\sigma)|^{p+1}\rho(z) dzd\sigma.\label{lb}
\end{align}
Choosing 
\begin{equation}\label{defdn}
\delta_n = (T(\bar x) -t_n)^\eta=e^{-\eta \sigma_n}
\end{equation}
 for some fixed $\eta\in (0,1)$ (for example $\eta = \frac 32$), we will see that we are able to conclude the proof of \eqref{hadaf}. Indeed:\\ 
- Note first that $\delta_n \to 0$ as $n\to \infty$, and from \eqref{defbd}, we have $\bar \delta_n \sim \delta_n|x-\bar x|(T(\bar x) -t_n)=(T(\bar x)-t_n)^{\eta -1}|x-\bar x|\to 0$ as $n\to \infty$.\\
- Then, using the convergence in \eqref{conv} and the Hardy-Sobolev estimate in Lemma \ref{lemhs}, we see that 
\begin{align}
&\sup_{\sigma_n \le \sigma\le \sigma_n+1}|\|w_{\bar x}(\sigma)\|_{L^{p+1}_\rho(|z|<1-\bar \delta_n)}-\|\bar w\|_{L^{p+1}_\rho(|z|<1-\bar \delta_n)}|\le\\
&\sup_{\sigma_n \le \sigma\le \sigma_n+1}\|w_{\bar x}(\sigma)-\bar w\|_{L^{p+1}_\rho(|z|<1-\bar \delta_n)}\le
\sup_{\sigma_n \le \sigma\le \sigma_n+1}\|w_{\bar x}(\sigma)-\bar w\|_{L^{p+1}_\rho(|z|<1)}\nonumber\\
\le&
\sup_{\sigma_n \le \sigma\le \sigma_n+1}\|w_{\bar x}(\sigma)-\bar w\|_{\q H}\to 0\mbox{ as }n\to \infty. \label{convw}
\end{align}
Since $\bar w \in \q H\subset L^{p+1}_\rho$, using Lebesgue's theorem, we see that $\|\bar w\|_{L^{p+1}_\rho(|z|<1-\bar \delta_n)}\to \|\bar w\|_{L^{p+1}_\rho(|z|<1)}$ as $n\to \infty$.
 Therefore, using \eqref{convw}, we see that 
\[
\sup_{\sigma_n \le \sigma\le \sigma_n+1}| \int_{|z|<1-\bar \delta_n}|w_{\bar x}(z,\sigma)|^{p+1}\rho(z) dz
- \int_{|z|<1}|\bar w(z)|^{p+1} \rho(z) dz|\to 0\mbox{ as }n\to \infty.
\]
Using the lower bound in Lemma \ref{lemlb}, we see that for $n$ large enough and for all $\sigma \in [\sigma_n, \sigma_n+1]$, we have
\[
\int_{|z|<1-\bar \delta_n}|w_{\bar x}(z,\sigma)|^{p+1}\rho(z) dz\ge \frac{\epsilon_0}2.
\]
Since $T(x) - t_n \to T(x) - T(\bar x) = |x-\bar x|$ as $n\to\infty$, from \eqref{deftn} and item (i) of Claim \ref{clcone}, recalling the choice we made for $\delta_n$ in \eqref{defdn}, then using \eqref{lb}, we see that
\[
I_n \ge e^{-2\alpha}\frac{\epsilon_0}4 (T(\bar x) -t_n)^{(\eta-2)\alpha}|x-\bar x|^{2\alpha}\to \infty\mbox{ as }n\to \infty,
\]
and \eqref{hadaf} follows. As explained in the line right after \eqref{hadaf}, this is in contradiction with \eqref{est}. Thus, case 2.b.i never occurs.\\
\noindent - {\bf Case 2.b.ii}: For any $\gamma \in (0,\bar x_1^2]$, $\tilde x(\gamma) \neq \bar x$.\\
 If for some $\gamma\in (0,\bar x_1^2]$, $\tilde x(\gamma)$ is outside the bisectrices, with $|\tilde x(\gamma)|\le \delta_4$ defined in \eqref{portion4}, then we obtain a contradiction as in Case 1 above.\\
Now, we assume that for all $\gamma \in (0, \bar x_1^2]$, $\tilde x(\gamma)$ is either on the bisectrices, or satisfies $|\tilde x(\gamma)|> \delta_4$. Two subcases then arise:\\
- either we have $|\tilde x(\gamma_n)|>\delta_4$, for some sequence $\gamma_n \to 0$. A contradiction then follows as in Case 2.a.i above.\\
- or, for all $\gamma\in (0, \gamma_0]$ and for some $\gamma_0\in (0,\bar x_1^2]$, we have $|\tilde x(\gamma)|\le \delta_4$ and $\tilde x(\gamma)$ is on the bisectrices.\\
If we can extract a subsequence $\gamma_n\to 0$ as $n\to \infty$ such that $\tilde x(\gamma_n)$ is on the same bisectrix as $\bar x$, then, then we have a contradiction, as we explain below. Indeed, note first that $\bar x\in \q C_{x,T(x), 1}$ by item (i) in Claim \ref{clcone}. Then, since $x$ is not on the bisectrix containing $\bar x$, for any point $z$ in this bisectrix such that $(z,T(z))\in\q C_{\bar x, T(\bar x), 1}$, we clearly see that $(z,T(z))$ is strictly below the cone $\q C_{x,T(x), 1}$, which is forbidden by the finite speed of propagation. Still because $x$ is not on that bisectrix, this property extends by continuiy to the situation where $z$ is in that bisectrix with $(z,T(z))\in\q C_{\bar x, T(\bar x), 1-\gamma}$ with $\gamma>0$ small enough. In fact, this happens to occur with $z=\tilde x(\gamma_n)$ for $n$ large enough, since by definition of $\tilde x(\gamma_n)$,  $(\tilde x, T(\tilde x))\in \q C_{\bar x, T(\bar x), 1-\gamma_n}$ and $\tilde x(\gamma_n)$ is precisely on the bisectrix, as we have just assumed a few lines before. Therefore, $\tilde x(\gamma_n)$ is strictly below the cone  $\q C_{x,T(x), 1}$, and this is forbidden by the finite speed of propagation. Contradiction.\\
Now, we assume that for all $\gamma \in (0, \gamma_1]$ for some $\gamma_1\le \gamma_0 \le \bar x_1^2$, $\tilde x(\gamma)$ is on the other bisectrix. If the half line starting from $x$ and containing $\bar x$ never encounters the other bisectrix, then by the same argument as the previous paragraph, we obtain a contradiction. If that half-line encounters the other bisectrix, then, a contradiction follows in a similar way.

This concludes the proof of Proposition \ref{lem3.1}.

\end{proof}
%%%%%%%%%%%%%%%%%%%%%%%%%%%%%%%%%%%%%
%%%%%%%%%%%%%%%%%%%%%%%%%%%%%%%%%%%%%
\subsection{Points on the bisectrices outside the origin are non-characteristic too}
%%%%%%%%%%%%%%%%%%%%%%%%%%%%%%%%%%%%%
%%%%%%%%%%%%%%%%%%%%%%%%%%%%%%%%%%%%%

We assume here that $x_2=x_1$. From the symmetries of the solution, the function $w_x(y,s)$ is odd in the $y$ variable. Using the non-degeneracy of blow-up limits at non-characteristic points from our earlier work \cite{MZtams14} (see Proposition \ref{propnondeg}), we derive the following lower bound on the blow-up 
surface
in the non-characteristic case:
%%%%%%%%%%%%%%%%%%%%%%%%%%%%%%%%%%%%%
%%%%%%%%%%%%%%%%%%%%%%%%%%%%%%%%%%%%%
\begin{lem}[A lower bound on the blow-up curve at non-characteristic points on the bisectrices]\label{lemlowbis} Assume that $x\in \RR$ with $|x_1|=|x_2|$. Then, we have $T(x) \ge -|x_1|+o(x_1)$ as $x\to 0$. 
\end{lem}
%%%%%%%%%%%%%%%%%%%%%%%%%%%%%%%%%%%%%
%%%%%%%%%%%%%%%%%%%%%%%%%%%%%%%%%%%%%
\begin{proof}
From the symmetries of the solution, we may assume that $0<x_1=x_2$ and $x_1$ is small.
Consider $A>0$. From Proposition \ref{propnondeg}, we know that 
\[
\|(w_x(\sl(x,A)), \partial_s w_x(\sl(x,A)))\|_{\q H}\ge \bar \epsilon_1,
\]
 where $\sl(x,A)$ is given in \eqref{defslsd3}, on the one hand. On the other hand, from Corollary \ref{cor2}, we know that 
\[
\left\|(w_x,\partial_s w_x)(y,\sl)-\kr(x,y,\sl)+\kup(x,y,\sl)\right\|_{\q H}\le C e^{-A}. 
\]
Taking $A$ large enough so that $Ce^{-A} \le \frac{\bar \epsilon_1}2$, we see that 
\[
\left\|\kr(x,y,\sl)+\kup(x,y,\sl)\right\|_{\q H}\ge \bar \epsilon_1- C e^{-A}\ge \frac{ \bar \epsilon_1}2. 
\]
But it happens from symmetry that $\left\|\kr(x,y,\sl)\right\|_{\q H}=\left\|\kup(x,y,\sl)\right\|_{\q H}$, in particular, 
\[
\left\|\kr(x,y,\sl)+\kup(x,y,\sl)\right\|_{\q H}\le 2  \left\|\kr(x,y,\sl)\right\|_{\q H}.
\]
Thus, we see that 
\begin{equation}\label{acquis}
\left\|\kr(x,y,\sl)\right\|_{\q H}\ge \frac{\bar \epsilon_1}4.
\end{equation}
Since we know from item (iii) of Lemma \ref{lemkd} that
\begin{equation}\label{maj}
\left\|\kr(x,y,\sl)\right\|_{\q H}\le C^* \lambda(|\bar d(\Sl)|, \nr) + \frac{C^*|\nr|}{\sqrt{1-|\bar d(\Sl)|^2}}\lambda(|\bar d(\Sl)|, \nr) ^{\frac{p+1}2},
\end{equation}
where $\lambda$ is defined in \eqref{deflambda}, we claim that
\begin{equation}\label{goall}
\lambda(|\bar d(\Sl)|, \nr) \ge \frac{\bar \epsilon_1}{16C^*}.
\end{equation}
Indeed, using the definitions \eqref{defs2} and \eqref{defslsd3} of $\nr$ and $\sl(x,A)$, together with the expansion \eqref{devdbar} of $\bar d$, we see that
\[
\frac{|\nr|}{\sqrt{1-|\bar d(\Sl)|^2}}\le Ce^A l^{-\frac \gamma 2}.
\]
Therefore, assuming by contradiction that \eqref{goall} doesn't hold, we see from \eqref{maj} that for $x$ small enough, we have
\[
\left\|\kr(x,y,\sl)\right\|_{\q H}\le 2 C^* \lambda(|\bar d(\Sl)|, \nr) \le \frac{\bar \epsilon_1}{8 C^*},
\]
and this is a contradiction by \eqref{acquis}. Thus, \eqref{goall} holds.\\
Using \eqref{compare}, we see that 
\[
\frac{|\nr|}{1-|\bar d(\Sl)|}\le C(\epsilon_1).
\]
Using again \eqref{defs2}, \eqref{defslsd3} and \eqref{devdbar}, we see that
\[
\frac{|\nr|}{1-|\bar d(\Sl)|}\sim \frac{[(-1+c_0l^{-\gamma}+o(l^{-\gamma}))x_1-T(x)]e^A}{x_1 c_0}.
\]
Thus, it follows that
\[
T(x) \ge x_1[-1+c_0l^{-\gamma}+o(l^{-\gamma})-c_0e^{-A}].
\]
Taking $A>0$ small enough, we conclude the proof of Lemma \ref{lemlowbis}. 
\end{proof}

Using the behavior of $T$ and $\nabla T$ outside the bisectrices given in Corollary \ref{lemgrad2}, we derive their behavior on the bisectrices:
%%%%%%%%%%%%%%%%%%%%%%%%%%%%%%%%%%%%%
%%%%%%%%%%%%%%%%%%%%%%%%%%%%%%%%%%%%%
\begin{lem}[Behavior of $T$ and $\nabla T$ on the bisectrices]\label{lembi}
 If $0<x_2=x_1$ is small, then:\\
(i)  $T(x) = -x_1+o(x_1 l^{-\frac \gamma 2})$ as $x\to 0$;\\
(ii) $T$ has upper and lower left derivatives along any direction non parallel to the bisectrix $\{x_1=x_2\}$, and the same holds from the right. In particular, 
 if $|\omega|=1$ and $\omega_2-\omega_1>0$, then:\\
$\partial_{\omega,r,\pm} T(x) = (-1+c_0l^{-\gamma}+o(l^{-\gamma}))\omega_1+o(l^{-\frac \gamma 2}) \omega_2$;\\
$\partial_{\omega,l,\pm} T(x) = (-1+c_0l^{-\gamma}+o(l^{-\gamma}))\omega_2+o(l^{-\frac \gamma 2}) \omega_1$
 as $x\to 0$, where the subscript $r$ and $l$ stands for ``right'' and ``left'', whereas the subscript $\pm$ stands for ``upper'' or ``lower''.
\end{lem}
%%%%%%%%%%%%%%%%%%%%%%%%%%%%%%%%%%%%%
%%%%%%%%%%%%%%%%%%%%%%%%%%%%%%%%%%%%%
\begin{proof}$ $\\
(i) This follows from item (i) in Corollary \ref{lemgrad2} by continuity of $T$.\\
(ii) Consider $\omega\in \m R^2$ such that $|\omega|=1$ and $\omega_2 - \omega_1>0$. From symmetry, It is enough to concentrate on the right directional derivative of $T$ along $\omega$.\\
Since $T$ is $C^1$ outside the bisectrice, for any $0<t'<t$ small enough, we write
\begin{equation}\label{taylorT}
T(x+t\omega) = T(x+t'\omega) + \int_{t'}^t \nabla T(x+\tau \omega)\cdot\omega d\tau.  
\end{equation}
From the expansion in item (iii) of Corollary \ref{lemgrad2}, we see that $\tau \mapsto \nabla T(x+\tau \omega)\cdot\omega$ is integrable on $(0,t)$. From continuity of $T$, we see that \eqref{taylorT} holds also with $t'=0$. Using again item (iii) of Corollary \ref{lemgrad2}, we find the expressions for the upper and lower left derivatives along $\omega$.
This concludes the proof of Lemma \ref{lembi}.
\end{proof}
Now, we prove that all points on the bisectrices (outside the origin) in a small neighborhood of the origin, are non-characteristic:
%%%%%%%%%%%%%%%%%%%%%%%%%%%%%%%%%%%%%
%%%%%%%%%%%%%%%%%%%%%%%%%%%%%%%%%%%%%
\begin{lem} \label{lem3.24} All points on the bisectrices in a small neighborhood of the origin (and outside the origin) are non-characteristic.
\end{lem}
%%%%%%%%%%%%%%%%%%%%%%%%%%%%%%%%%%%%%
%%%%%%%%%%%%%%%%%%%%%%%%%%%%%%%%%%%%%
\begin{proof} Consider some small $x\neq 0$ such that $|x_1|=|x_2|$. From the symmetries of the solution, we may assume that $0<x_2=x_1$ with $x_1$ small. It is enough to show that the cone $\q C_{x,T(x), 1-\frac{c_0}2 l^{-\gamma}}$ is under 
$\Gamma$,
in any direction $\omega$ with $|\omega|=1$. From symmetry, we only consider the case where $\omega_2 - \omega_1\ge 0$.\\
If  $\omega_2=\omega_1$, from item (i) in Lemma \ref{lembi}, we see that $T(y)= - y_1+o(y_1 |\log y_1|^{-\frac\gamma 2})$ when $y_1=y_2$, therefore the restriction of $T$ to the bisectrix has a slope close to $-\frac 1{\sqrt 2}>-1$. Since the slope of the cone is $-1+\frac{c_0}2 l^{-\gamma}$, this means that 
$\Gamma$
is above the cone.\\
If $\omega_2>\omega_1$, using item (ii) of Lemma \ref{lembi}, we see that the slope of $T$ is bounded from below by $-1+\frac{3c_0}4 l^{-\gamma}$, which is larger than  $-1+\frac{c_0}2 l^{-\gamma}$, the slope of the cone. Therefore, 
$\Gamma$
is above the cone.\\
This concludes the proof of Lemma \ref{lem3.24}.
\end{proof}
%%%%%%%%%%%%%%%%%%%%%%%%%%%%%%%%%%%%%
%%%%%%%%%%%%%%%%%%%%%%%%%%%%%%%%%%%%%
\subsection{The origin is a characteristic point}
%%%%%%%%%%%%%%%%%%%%%%%%%%%%%%%%%%%%%
%%%%%%%%%%%%%%%%%%%%%%%%%%%%%%%%%%%%%
This is our statement in this section:
%%%%%%%%%%%%%%%%%%%%%%%%%%%%%%%%%%%%%
%%%%%%%%%%%%%%%%%%%%%%%%%%%%%%%%%%%%%
\begin{lem}\label{lem3.25}$ $\\ 
(i) At the origin, $T$  has a right derivative with respect to $x_1$ whose value is  $\partial_{x_1,r}T(0)=-1$, with similar statements from the left and in the direction $x_2$.\\
(ii) The origin is a characteristic point.
\end{lem}
%%%%%%%%%%%%%%%%%%%%%%%%%%%%%%%%%%%%%
%%%%%%%%%%%%%%%%%%%%%%%%%%%%%%%%%%%%%
\begin{proof} Clearly, (ii) is a consequence of (i), and (i) is a consequence of item (ii) in Corollary \ref{lemgrad2}.
\end{proof}

%%%%%%%%%%%%%%%%%%%%%%%%%%%%%%%%%%%%%
%%%%%%%%%%%%%%%%%%%%%%%%%%%%%%%%%%%%%
\subsection{Conclusion of the proof of Theorem \ref{mainth}}
\label{subconcl}
%%%%%%%%%%%%%%%%%%%%%%%%%%%%%%%%%%%%%
%%%%%%%%%%%%%%%%%%%%%%%%%%%%%%%%%%%%%

In this section, we collect all the previous estimates to finish the proof of Theorem \ref{mainth}. 

\medskip

We first start with the proof of Theorem \ref{mainth}.

\begin{proof}[Proof of Theorem \ref{mainth}]
(A) Using Proposition \ref{propyr} and the invariance by time translation of equation \eqref{equ}, we have a solution $u(x,t)$ to the Cauchy problem, symmetric with respect to the axes and anti-symmetric with respect to the bisectrices, which blows up on a 
surface $\Gamma =\{(x,T(x))\}$
satisfying $T(0)>0$.\\
From Proposition \ref{lem3.1}, Lemmas \ref{lem3.24} and \ref{lem3.25}, we see that the origin is isolated characteristic point. This finishes the proof of item (A) in Theorem \ref{mainth}.\\
(B) and the fifth remark following the statement of Theorem \ref{mainth}:\\ 
Outside the bisectrix $\{ x_1=x_2\}$, this comes from items (i) and (iii) in Corollary \ref{lemgrad2} ; using the continuity of $T$, this extends to the bisectrix.\\
As for the fifth remark following the statement of Thoerem \ref{mainth}, see item (ii) in Lemma \ref{lembi} and  item (i) in Lemma \ref{lem3.25}.
This concludes the proof of item (B) in Theorem \ref{mainth}.\\
(C) In the following, we give indications on how to derive the proofs:\\
(i) This is a consequence of Proposition \ref{propyr}.\\
(ii) Given Part (B), this is a consequence of items (ii) and (i) in Lemma \ref{lemgrad}.\\
(iii) Given Part (B), this is a consequence of Lemma \ref{lembehbis}.
This concludes the proof of Theorem \ref{mainth}.
\end{proof}

\appendix

%%%%%%%%%%%%%%%%%%%%%%%%%%%%%%%%%%%%
%%%%%%%%%%%%%%%%%%%%%%%%%%%%%%%%%%%%
\section{Details for the dynamics of the equation near 4 solitons}
\label{appdyn}
%%%%%%%%%%%%%%%%%%%%%%%%%%%%%%%%%%%%
%%%%%%%%%%%%%%%%%%%%%%%%%%%%%%%%%%%%

This section is devoted to the proof of Lemma \ref{propdyn}.
Since the proof is just a two-dimensional version of our work in one dimension with multi-solitons in \cite{MZajm12}, \cite{MZdmj12} and \cite{CZcpam13}, we only focus of the truly two-dimensional
 estimates, and refer the reader to our earlier work for more classical terms. 

\begin{proof}[Proof of Lemma \ref{propdyn}] Using the definition \eqref{defq} of $q$, we transform equation \eqref{eqw} satisfied by $w$ into the following system satisfied by $q$, for all $s\in [s_0, \bar s)$:
\begin{align}
\partial_s q 
& = {\hat L} (q)
+ \vc{0}{f(q_1)}+\vc{0}{R}\label{eqq*}\\
&-(\nu'(s)-\nu(s))\sum_\ind (-1)^{i+1}\pnu\kappa^*(\theta d(s)e_i,\nu(s),y)\nonumber\\
&-d'(s) \sum_\ind (-1)^{i+1}\theta\partial_{{\bs d}_i}\kappa^*(\theta d(s)e_i,\nu(s),y) \nonumber 
\end{align}
where
\begin{align}
{\hat L}(q)
 & = \vc{q_2}{\q L q_1+\psi q_1-\frac{p+3}{p-1}q_2-2y\py q_2},
\psi(y,s) = p|K^*_1(y,s)|^{p-1} -\frac{2(p+1)}{(p-1)^2},\nonumber \\%\qquad
K^*_1(y,s)& = \sum_\ind  (-1)^{i+1}\kappa^*_1(\theta d(s)e_i,\nu(s),y),\label{defk*1}\\
f(q_1) & = |K^*_1+q_1|^{p-1}(K^*_1+q_1)- |K^*_1|^{p-1}K^*_1- p|K^*_1|^{p-1} q_1,\nonumber \\
R & = |K^*_1|^{p-1}K^*_1- \sum_\ind (-1)^{i+1}\kappa^*_1(\theta d(s)e_i,\nu(s),y)^p\nonumber
\end{align}
and the sum $\sum_\ind$ runs for $i=1,2$ and $\theta=\pm 1$, here and in the following.
\begin{nb} In order to use our estimates established in \cite{MZtams14} in the case of one soliton in higher dimensions, we decompose the linear term ${\hat L} (q)$ as follows:
\[
{\hat L} (q(s)) = \bar L (q(s)) + \vc{0}{\bar V(y,s)q_1} 
\]
where 
\begin{align}
\bar L(q)& = \vc{q_2}{\q L q_1+\bar \psi q_1-\frac{p+3}{p-1}q_2-2y\py q_2},\nonumber\\
\bar \psi(y,s)& = p\kappa^*_1(d(s) e_1, \nu(s),y)^{p-1} -\frac{2(p+1)}{(p-1)^2}, \nonumber\\
\bar V(y,s) &= p|K^*_1(y,s)|^{p-1}- p\kappa^*_1(d(s) e_1, \nu(s),y)^{p-1}.\label{defbv}
\end{align}
\end{nb}

\medskip

We proceed in two parts:

- In Part 1, we project equation \eqref{eqq*} with the projector $\pp_l(d^*(s) e_1)$ defined in \eqref{defpdi} with $l=0,1$ and $d^*(s) = \frac{d(s)}{1+\nu(s)}$, in order to prove \eqref{est:nu} and \eqref{est:zeta}.

- In Part 2, we will find a Lyapunov functional for equation \eqref{eqq*}, which is equivalent to the norm squared, deriving estimate \eqref{est:q}.

\bigskip

{\bf Part 1: Projection of equation \eqref{eqq*} with the projector $\pp_l(d^*(s) e_1)$}

Let us assume that $s_0\ge 1$ and take $s\in [s_0, \bar s)$.\\
In this part, 
we will project each term of equation \eqref{eqq*} with the projector $\pp_l(d^*(s) e_1)$ defined in \eqref{defpdi} with $l=0,1$ and $d^*(s) = \frac{d(s)}{1+\nu(s)}$, ending by deriving equations \eqref{est:nu} and \eqref{est:zeta}.\\
Using the modulation effect \eqref{mod} and the bounds \eqref{conmod}, arguing as in \cite{MZdmj12}, \cite{CZcpam13} and \cite{MZtams14} and using the relation
\begin{equation}\label{defnhz1}
\|V\|_{\H}^2 =C(N) \int_{-1}^1\left((\bar V_1(z_1))^2 +(\partial_{z_1}\bar V(z_1))^2(1-z_1^2)+(\bar V_1(z_1))^2\right)(1-z_1^2)^{\frac 2{p-1}} dz_1
\end{equation}
for some $C(N)>0$, between the multi-dimensional and the one-dimensional expressions for the norm, we get the following estimates for the projections of what we called ``traditional'' terms of \eqref{eqq*}:
\begin{align}
&|\pp_l(d^*(s) e_1, \partial_s q(s))|\le C \frac{|{d^*}'(s)|}{1-|d^*(s)|^2}\|q(s)\|_{\q H}\nonumber\\
&\le C\frac{\|q(s)\|_{\q H}}{1-|d(s)|^2}(|d'(s)|+|\nu'(s) - \nu(s)|+|\nu(s)|),\nonumber\\
&|\pp_l(d^*(s) e_1, \bar L(q(s)))|\le C \frac{|\nu(s)|}{1-|d(s)|^2}\|q(s)\|_{\q H}.\nonumber\\
&\pp_0(d^*(s) e_1, \partial_\nu \kappa^*(d(s) e_1,\nu(s)))=0,\nonumber\\
&-C\le (1-|d(s)|^2)\pp_1(d^*(s) e_1, \partial_\nu \kappa^*(d(s) e_1,\nu(s)))\le - \frac 1{C},\nonumber\\
&|\pp_1(d^*(s) e_1, \partial_d \kappa^*(d(s) e_1,\nu(s)))|\le \frac {C}{1-|d(s)|^2}\label{p1dnk},\\
&\left|\pp_0(d^*(s) e_1, \partial_d \kappa^*(d(s) e_1,\nu(s)))+\frac{c_3}{1-|d(s)|^2}\right|\le \frac{c|\nu(s)|}{(1-|d(s)|^2)^2}\mbox{ for some }c_3(p)>0,\label{p0ddk}\\
&|\pp_l(d^*(s) e_1, \partial_\nu \kappa^*(- d(s) e_1,\nu(s)))|
+|\pp_l(d^*(s) e_1, \partial_d \kappa^*(- d(s) e_1,\nu(s)))|\le \frac{Ce^{-\frac 2{p-1}\zeta(s)}}{1-|d(s)|^2},\nonumber\\
&|\pp_l(d^*(s) e_1, \partial_\nu \kappa^*(\pm d(s) e_2,\nu(s),y))|
+|\pp_l(d^*(s) e_1, \partial_d \kappa^*(\pm d(s) e_2,\nu(s),y))|\nonumber\\
& \le \frac C{1-|d(s)|^2}\iint\kappa(d^*(s)e_1,y)\kappa(d^*(s)e_2,y)\frac{\rho(y)}{1-|y|^2}dy.\label{petit}
\end{align}
Now, we focus on the ``new'' terms of equation \eqref{eqq*}, which were never encountered in our previous papers, and need henceforth some delicate treatment. Note first from \eqref{conmod} that taking $s_0$ large enough, we get
\begin{align}
|d(s)+1|+|d^*(s)+1|&\le Ce^{-2 \zeta(s)},\; 
\frac{e^{-2 \zeta(s)}}C\le 1-|d^*(s)|^2\le C(1-|d(s)|^2)\le C e^{-2 \zeta(s)},\nonumber\\
\left|e^{2(\zeta(s)-\zeta^*(s))}-1\right|&\le Ce^{-2\zeta(s)}+C\frac{|\nu(s)|}{1-|d(s)|},\label{conseq}
\end{align}
where $\zeta^*(s) = -\arg \tanh d^*(s)$.

\bigskip

{\it - Estimate of $\pp_l(d^*(s) e_1, (0,\bar V(y,s)q_1))$:}

Introducing for $i=1,2$ and $\theta = \pm 1$, 
\begin{equation}\label{defqit}
Q_{i,\theta} = \{|y|<1\;|\; \theta y_i \ge |y_{3-i}|\},
\end{equation}
we see that we have a partition of the unit ball into four quarters with boundaries given by the bisectrices. Since $d(s)<0$ from \eqref{conseq}, it follows that for all $y \in Q_{i,\theta}$ and $(\eta, j) \neq (\theta,i)$,
\begin{equation}\label{domin}
C\kappa(\theta d^*(s) e_i,y)\ge C(1-|d^*(s)|^2)^{\frac 1{p-1}}\ge C \kappa(\eta d^*(s) e_j, \nu,y).
\end{equation}
Therefore, using item (iii) of Lemma \ref{lemkd} and the definition \eqref{defbv} of $\bar V(y,s)$, one can check from elementary expansions that
\begin{align}
|\bar V(y,s)|
&\le C \sum_{(i,\theta) \neq (1,1)} \kappa^*_1(d e_1, \nu,y)^{p-2}
\kappa^*_1(\theta d e_i,\nu,y) 1_{Q_{1,1}}
+ \kappa^*_1(\theta d e_i,\nu,y)^{p-1}1_{Q_{i,\theta}}\nonumber\\
&\le C \sum_{(i,\theta) \neq (1,1)} \kappa(d^* e_1,y)^{p-2}
\kappa(\theta d^* e_i,y) 1_{Q_{1,1}}
+ \kappa(\theta d^* e_i,y)^{p-1}1_{Q_{i,\theta}}\label{boundbv}
\end{align}
Using item (ii) of Lemma \ref{lemkd}, 
we write from the definition \eqref{defpdi} of $\pp_l(d^* e_1)$ and the Cauchy-Schwarz inequality 
\begin{align}
|\pp_l(d^*(s) e_1, (0,\bar V(y,s)q_1))|&\le C\iint \kappa(d^*(s) e_1,y) |\bar V(y,s)||q_1(y,s)|\rho(y) dy \nonumber\\
&\le C\|q_1\|_{L^2_{\frac \rho{1-|y|^2}}}\|\kappa(d^*(s) e_1,y) \bar V(y,s)\|_{L^2_{\rho(1-|y|^2)}}.\label{5a}
\end{align} 
Using \eqref{boundbv}, the symmetries of the solution, \eqref{domin} and the computation table given in Lemma \ref{cltech0}, we see that
\begin{align*}%\label{4/12}
&\|\kappa(d^*(s) e_1,y) \bar V(y,s)\|_{L^2_{\rho(1-|y|^2)}}^2\\
&\le C\int_{Q_{1,1}}\kappa(d^* e_1,y)^{2(p-1)}(\kappa(-d^* e_1,y)^{2}+\kappa(d^* e_2,y)^{2})\rho(1-|y|^2)dy\\
&\le C (1-|d^*|^2)^{\frac 2{p-1}}\iint \kappa(d^* e_1,y)^{2(p-1)} \rho (1-|y|^2) dy
\le
C (1-|d^*|^2)^{\frac {2\bar p}{p-1}}.
\end{align*}
Using \eqref{conseq}, it follows from \eqref{5a} and Lemma \ref{lemhs} that 
\[
|\pp_l(d^*(s) e_1, (0,\bar V(y,s)q_1))|\le C\|q\|_{\q H}e^{-\frac {2\bar p\zeta}{p-1}}\le C\|q\|_{\q H}^2+Ce^{-\frac {4\bar p\zeta}{p-1}}.
\]

\medskip

{\it - Estimate of $\pp_l(d^*(s) e_1, (0,f(q_1)))$:}

Note first that we have
\begin{equation}\label{boundfq1}
|f(q_1)|\le C \delta_{\{p\ge 2\}}|q_1(y,s)|^p
+
C |K^*_1(y,s)|^{p-2}|q_1(y,s)|^2.
\end{equation}
Using item (ii) of Lemma \ref{lemkd},
 we see that
\begin{align}
&|\pp_l(d^*(s) e_1, (0,f(q_1)))|\le 
C\iint \kappa(d^*(s) e_1,y)|f(q_1)| \rho dy\label{principal}\\
\le & C\delta_{\{p\ge 2\}} \iint \kappa(d^*(s) e_1,y)|q_1|^p \rho dy
+
C \iint \kappa(d^*(s) e_1,y) |K^*_1(y,s)|^{p-2}|q_1|^2\rho dy.\nonumber
 \end{align}
If $p\ge 2$, using item (i) of Lemma \ref{lemkd},
the H\"older inequality, the Hardy-Sobolev inequality of Lemma \ref{lemhs} and the last bound in \eqref{conmod}, we write
\begin{align}
\iint \kappa(d^*(s) e_1,y)|q_1|^p \rho dy
&\le \left(\iint  \kappa(d^*(s) e_1,y)^{p+1}\rho dy\right)^{\frac 1{p+1}}
 \left(\iint |q_1|^{p+1}\rho dy\right)^{\frac p{p+1}}\nonumber\\
&\le C\|\kappa(d^* e_1)\|_{\q H_0}\|q\|_{\q H}^p\le C\|q\|_{\q H}^p\le C\|q\|_{\q H}^2.\label{a912}
\end{align}
Regarding the second integral, if $p\ge 2$, then we write from the definition \eqref{defk*1} of $K^*_1$ and item (ii) of Lemma \ref{lemkd}
 $\kappa(d^*(s) e_1,y) |K^*_1(y,s)|^{p-2}\le C(1-|y|^2)^{-1}$. Using Lemma \ref{lemhs}, we obtain
\begin{equation}
\iint  \kappa(d^*(s) e_1,y) |K^*_1(y,s)|^{p-2}|q_1|^2\rho dy 
\le C \iint |q_1|^2\frac{\rho}{1-|y|^2} dy \le C \|q\|_{\q H}^2. \label{b912}
\end{equation}
Now, if $p<2$, we write
\begin{equation}\label{i1i2}
\iint  \kappa(d^*(s) e_1,y) |K^*_1(y,s)|^{p-2}|q_1|^2\rho dy = I_1+I_2
\end{equation}
where $I_i=\ds \int_{D_i}  \kappa(d^*(s) e_1,y) |K^*_1(y,s)|^{p-2}|q_1|^2\rho dy$ with 
\[
D_1 = \{\epsilon_0\kappa(d^*(s) e_1,y)\le |K^*_1(y,s)|\} \mbox{ and }
D_2 = \{\epsilon_0\kappa(d^*(s) e_1,y)> |K^*_1(y,s)|\},
\]
where $\epsilon_0>0$ will be fixed small enough.\\
As for $I_1$, using item (ii) of Lemma \ref{lemkd} and the Hardy-Sobolev estimate of Lemma \ref{lemhs}, we write
\begin{equation}\label{esti1}
I_1\le \iint \kappa(d^*(s) e_1,y)^{p-1}|q_1|^2\rho dy \le \frac C{\epsilon_0^{2-p}}\iint |q_1|^2\frac{\rho}{1-|y|^2} dy \le \frac C{\epsilon_0^{2-p}}\|q\|_{\q H}^2.
\end{equation}
As for $I_2$, 
note first from item (iii) in Lemma \ref{lemkd} and \eqref{conseq} that 
we may take $s_0$ large enough and $\epsilon_0$ small enough so that
\begin{equation}\label{defd3}
D_2 \subset \bar D(\eta)\equiv\{|y_1^2-y_2^2|\le \eta\}
\end{equation}
(a quick way to justify this, is to see that the set
$D_2$
approaches the zero set of $K^*(y,s)$, namely the bisectrices $\{y_1=\pm y_2\}$). Therefore, 
taking some $\gamma>\frac 2{p-1}>2$, using the H\"older inequality and the Hardy-Sobolev estimate of Lemma \ref{lemhs}, we write
\begin{align}
I_2 & \le \int_{D_2} \kappa(d^*(s) e_1,y) |K^*_1(y,s)|^{-(2-p)}(1-|y|^2)^{\frac 1\gamma - \frac 12}|q_1|^2(1-|y|^2)^{\frac 2{p-1} - \frac 1\gamma}dy\nonumber\\ 
&\le \bar A \left(\iint |q_1|^\gamma(1-|y|^2)^{\frac \gamma{p-1}-\frac 12} dy \right)^{\frac 2\gamma}\le C\bar A \|q(s)\|_{\q H}^2,\label{boundi2}
\end{align}
where 
\begin{align}
\bar A^{\frac \gamma{\gamma-2}} &=\ds \int_{D_2} \kappa(d^*(s)e_1,y)^{\frac \gamma{\gamma-2}} |K^*_1(y,s)|^{-\frac{\gamma(2-p)}{\gamma-2}}(1-|y|^2)^{-\frac 12}dy
\nonumber\\
&=\lambda(d(s),\nu(s)) ^{-\frac{\gamma(2-p)}{\gamma-2}} (1-|d^*(s)|^2)^{\frac \gamma{\gamma-2}}\bar B(\gamma, \eta, d^*(s)) \label{bounda3}
\end{align}
from item (iii) of Lemma \ref{lemkd}, with
\begin{equation}\label{defb3}
\bar B (\gamma, \eta, \delta) =
 \kappa_0^{\frac{\gamma(p-1)}{\gamma-2}}\ds \int_{\bar D(\eta)}\frac{|\bar K(\delta y)|^{-\frac{\gamma(2-p)}{\gamma-2}}}{(1-\delta y_1)^{\frac 2{p-1}}}(1-|y|^2)^{-\frac 12}dy
\mbox{ and }
\bar K(y) =\sum_\ind\frac{(-1)^{i+1}}{(1+\theta y_i)^{\frac 2{p-1}}}.
\end{equation}
Since item (iii) of Lemma \ref{cltech0} implies that $\bar B\le C$ for $s_0$ large enough and $\eta>0$ small enough, 
 using \eqref{boundi2}, \eqref{bounda3}, item (iii) of Lemma \ref{lemkd} and \eqref{conseq}, 
we see that
\[
I_2 \le C e^{-\frac{2\gamma}{\gamma-2}\zeta(s)}\|q(s)\|_{\q H}^2.
\]
Using \eqref{esti1}, \eqref{a912} and \eqref{principal}, we see that 
\begin{align}
&|\pp_l(d^*(s) e_1, (0,f(q_1)))|\le 
C\iint \kappa(d^*(s) e_1,y)|f(q_1)| \rho dy\nonumber\\
\le & C\delta_{\{p\ge 2\}} \iint \kappa(d^*(s) e_1,y)|q_1|^p \rho dy
+
C \iint \kappa(d^*(s) e_1,y) |K^*_1(y,s)|^{p-2}|q_1|^2\rho dy\nonumber\\
\le& C\|q(s)\|_{\q H}^2.\label{estnl}
 \end{align}

\medskip

{\it - Estimate of $\pp_l(d^*(s) e_1, (0,R))$:}
Writing the unit disc as the union of the four quarters defined in \eqref{defqit}, we obtain the following expansion for $R(y,s)$:
\begin{equation}\label{estR}
|R(y,s)-p\kappa^*_1(d(s)e_1,\nu(s),y)^{p-1}\sum_{(i,\theta)\neq (1,1)}(-1)^{i+1} \kappa^*_1(\theta d(s)e_i, \nu(s), y)|\le C G(y,s)
\end{equation}
where we write from item (iii) of Lemma \ref{lemkd}
\begin{align*}
G(y,s) &=
 \sum_\ind \kappa(\theta d^*(s)e_i,y)^p(1-1_{Q_{i,\theta}})\\
&+\kappa(d^*(s)e_1,y)^{p-1}(1-1_{Q_{1,1}})\sum_{(i,\theta) \neq (1,1)}\kappa(\theta d^*(s)e_i,  y)\\
&+\sum_{(i,\theta) \neq (1,1)}1_{Q_{i,\theta}}\kappa(\theta d^*(s) e_i,y)^{p-1}\sum_{(j,\eta)\neq (i,\theta)}\kappa(\eta d^*(s) e_j,y)\\
& + \kappa(d^*(s)e_1,y)^{p-2}1_{Q_{1,1}}\sum_{(i,\theta) \neq (1,1)}\kappa(\theta d^*(s) e_i,y)^2.
\end{align*}
Therefore, using the definition \eqref{defpdi} of $\pp_l(d^*(s)e_1)$ and item (ii) of Lemma \ref{lemkd}, we see that
\begin{equation}\label{mer1}
\left|\pp_l(d^*(s)e_1, R(s))- \sum_{(i,\theta) \neq (1,1)} (-1)^{i+1} R_{l,i,\theta}(s)\right|
\le C \iint \kappa(d^*(s) e_1, y) G(y,s) \rho dy,
\end{equation}
where 
\begin{equation}\label{defRlit}
 R_{l,i,\theta}(s) = p\iint W_{l,2}(d^*(s) e_1,y) \kappa^*_1(d(s)e_1,\nu,y)^{p-1}\kappa^*_1(\theta d(s)e_i, \nu(s), y) \rho dy.
\end{equation}
Using \eqref{domin} and the integral table of Lemma \ref{cltech0}, it is  straightforward to see that
\begin{equation}\label{mer2}
\iint \kappa(d^*(s) e_1, y) G(y,s) \rho dy \le Ce^{-\left(\frac 4{p-1}+\epsilon_0\right)\zeta(s)},
\end{equation}
for some $\epsilon_0>0$.\\
Similarly, when $l=1$, decomposing the unit disc into the 4 quarters $Q_{i,\theta}$, we derive that
\begin{equation}\label{mer3}
|R_{1,i,\theta}(s)|\le Ce^{-\frac 4{p-1} \zeta(s)}\mbox{ for all }(i,\theta) \neq (1,1),
\end{equation}
therefore,
\[
\left|\pp_1(d^*(s)e_1, R(s))\right|\le Ce^{-\frac 4{p-1} \zeta(s)}.
\]
When $l=0$ and $(i,\theta)\neq (1,1)$, using the definition \eqref{defWl2} of $W_{l,2}(d^*(s)e_1)$ and the relation in item (iii) of Lemma \ref{lemkd}, we see that
\begin{equation}\label{deftrit}
R_{0,i,\theta}(s)= \frac{pc_0}{\kappa_0}\lambda(d^*(s),\nu(s))^p\bar R_{i,\theta}(d^*(s))
\mbox{ where }
\bar R_{i,\theta}(\delta) = \iint\kappa(\delta e_1,y)^p\kappa(\theta \delta e_i, y) \rho dy.
\end{equation}
Using item (ii) of Lemma \ref{cltech0}, item (iii) of Lemma \ref{lemkd} and \eqref{conseq}, we see from \eqref{deftrit}, \eqref{mer1} and \eqref{mer2} that
\begin{equation}\label{p0r}
\left|e^{\frac 4{p-1} \zeta(s)}\pp_0(d^*e_1, R(s))-c_2(p)\right|\le  C\left(e^{-\epsilon_2\zeta(s)}+\frac {|\nu(s)|}{1-|d(s)|^2}\right),
\end{equation}
for some $\epsilon_2>0$, where $c_2(p)= \frac{pc_0\barcc}{\kappa_0}(2-2^{-\frac 2{p-1}})>0$, and the constants $c_0$ and $\barcc$ are defined in \eqref{defWl2} and item (ii) of Lemma \ref{cltech0}.

\bigskip

{\it - Estimate of $\pp_l(d^*(s) e_1,\partial_\nu \kappa^*(\pm  d(s)e_2,\nu(s)))$ and $\pp_l(d^*(s) e_1,\partial_d \kappa^*(\pm  d(s)e_2,\nu(s)))$  }:

Estimating the right-hand side of estimate \eqref{petit}, by considering the four quarters \eqref{defqit}, then evaluating the integral using the integral computation table given in item (i) of Lemma \ref{cltech0}, then using \eqref{conseq}, we see that 
\begin{align*}
|\pp_l(d^*(s) e_1, \partial_\nu \kappa^*(\pm d(s) e_2,\nu(s),y))|
+|\pp_l(d^*(s) e_1, \partial_d \kappa^*(\pm d(s) e_2,&\nu(s),y))|\\
&\le C \frac{\zeta(s) e^{-\frac 4{p-1} \zeta(s)}}{1-|d(s)|^2}.
\end{align*}

\medskip

{\it - Conclusion: Differential inequalities for $\nu(s)$ and $\zeta(s)$}:

In the previous pages, we projected all the terms of equation \eqref{eqq*} with the projector $\pp_l(d^*(s) e_1)$ where $l=0,1$. 
Putting on the left-hand side the main terms (namely \eqref{p0ddk} and \eqref{p0r} when $l=0$, then \eqref{p1dnk} when $l=1$), and bearing in mind that
\[
\zeta(s) = -\arg \tanh d(s),\mbox{ hence }
\zeta'(s) =-\frac{d'(s)}{1-|d(s)|^2},
\]
we obtain the following inequalities:
\begin{align*}
\left|c_3 \zeta'(s) -c_2 e^{-\frac 4{p-1}\zeta(s)}\right|&\le C\left(|\zeta'|+\frac{|\nu'-\nu|}{1-|d|^2}\right)\left(\frac{|\nu(s)|}{1-|d(s)|}+\|q(s)\|_{\q H}+e^{-\frac 2{p-1}\zeta(s)}\right)\\
&+C e^{-\frac 4{p-1}\zeta(s)}\left(e^{-\epsilon_3\zeta(s)}+\frac{|\nu(s)|}{1-|d(s)|^2}\right)
\\
&+C\frac{|\nu(s)|}{1-|d(s)|^2} \|q(s)\|_{\q H}
+C\|q(s)\|_{\q H}^2\\
\frac{|\nu'(s)-\nu(s)|}{1-|d(s)|^2}&\le
 C\frac{|\nu(s)|}{1-|d(s)|^2}\|q(s)\|_{\q H}
+C\|q(s)\|_{\q H}^2+Ce^{-\frac 4{p-1}\zeta(s)}\nonumber\\
&+\frac{|\nu'(s)-\nu(s)|}{1-|d(s)|^2}\left(\|q(s)\|_{\q H} +e^{-\frac 2{p-1}\zeta(s)}\right)
+C|\zeta'(s)|,\nonumber
\end{align*}
where the constants $c_2(p)$, $c_3(p)$ and $\epsilon_3$ are positive. 
Using the bounds in \eqref{conmod}, we derive estimates \eqref{est:nu} and \eqref{est:zeta}.

\bigskip

{\bf Part 2: A Lyapunov functional for equation \eqref{eqq*}}

We prove estimate \eqref{est:q} here.
Like for Claim 4.8 page 2867 in \cite{MZdmj12} and Proposition 3.6 in \cite{MZtams14}, the idea is simple: we construct a Lyapunov functional for equation \eqref{eqq*} which is equivalent to the norm squared. Naturally, this functional is obtained by multiplying the equation on $q_1$ derived from \eqref{eqq*} by $\partial_s q_1\rho$ then integrating on the unit ball. Without the modulation terms or the interaction term ($R(y,s$ defined in \eqref{defk*1}), that functional would be $\frac 12 \bar \varphi(q,q)-\iint \q F(q_1)\rho dy$, where
\begin{align}
\bar \varphi\left(q, r\right)&= \int_{|y|<1} \left(-\psi(d,y)q_1r_1+\nabla q_1\cdot \nabla r_1-(y\cdot \nabla r_1)(y\cdot \nabla q_1)+q_2r_2\right)\rho dy,\label{defphib}\\
\q F(q_1)& = \int_0^{q_1}f(\xi) d\xi = \frac{|K^*_1+q_1|^{p+1}}{p+1}-\frac{{K^*_1}^{p+1}}{p+1}-{K^*_1}^p q_1 - \frac p2 {K^*_1}^{p-1}q_1^2, \label{defF}
\end{align}
and $\psi(d,y)$, $f(q_1)$ and $K^*_1(y,s)$ are defined in \eqref{defk*1}. Because we ``killed'' the nonnegative directions in \eqref{mod}, we will see that our Lyapunov functional controls the square of the norm of the solution, because \eqref{conmod} holds.
However, because of the modulation terms and the interaction term $R(y,s)$, we need to slightly change the functional we intend to study, by defining:
\begin{align}
h_1(s)& = \frac 12 \bar \varphi(q,q)-\iint \q F(q_1)\rho dy+\eta \iint q_1 q_2 \rho dy,\\
h_2(s) & = h_1(s) - 2 \eta^{-2}\ds e^{-\frac{4\bar p \zeta(s)}{p-1}},\label{defh}
\end{align}
where $\bar p$ is introduced in \eqref{defpb} and $\eta>0$ will be fixed later as a small enough universal constant. 
Then, we clearly see that the following identity allows to conclude:

\medskip

 {\it There exist $\delta>0$ such that 
\begin{equation}\label{lemlyap}
\forall s\in [s_0, \bar s),\;\;\delta h_1(s) \le \|q(s)\|_{\q H}^2 \le \delta^{-1}h_1(s)\mbox{ and }h_2'(s)\le -\delta h_2(s).
\end{equation}
}
It remains then to prove \eqref{lemlyap} in order to conclude. We proceed in 2 steps to do that, each dedicated to a part of \eqref{lemlyap}.

\bigskip

{\bf Step 1: $h_1(s)$ is comparable to the square of the norm}

We prove the first part in \eqref{lemlyap} here. Clearly, it follows from the following, for $\eta$ small enough and $s_0$ large enough. %%%%%%%%%%%%%%%%%%%%%%%%%%%%%%%%%%%%%%%%%%%%%
%%%%%%%%%%%%%%%%%%%%%%%%%%%%%%%%%%%%%%%%%%%%%
\begin{lem}\label{lemquad} There exists $s_5\ge 0$ such that if $s_0\ge s_5$, 
then for all $s\in [s_0, \bar s)$, we have:
\begin{align}
\frac {\bar\varphi(q,q)}{C} &\le \|q\|_{\H}^2\le C \bar\varphi(q,q),
\label{eq2}\\
\left|\iint \q F(q_1) \rho dy\right|&\le C \|q\|_{\H}^{1+\min(p,2)}\le C s_0^{(1-\min(p,2))/2}\|q\|_{\H}^2,\label{eq2a}\\
\left|\iint q_1 q_2 \rho dy \right|&\le \|q\|_{\q H}^2.\label{eq2c}
\end{align}
\end{lem}
%%%%%%%%%%%%%%%%%%%%%%%%%%%%%%%%%%%%%%%%%%%%%
%%%%%%%%%%%%%%%%%%%%%%%%%%%%%%%%%%%%%%%%%%%%%
Let us then prove Lemma \ref{lemquad}. 

\medskip

\begin{proof}[Proof of Lemma \ref{lemquad}] Note first that \eqref{eq2c} is trivial.
 As for \eqref{eq2a}, it follows from a simplification of the argument we used for \eqref{estnl}, together with \eqref{conmod}, if one starts from the following bound:
\begin{equation}\label{boundFq1}
|\q F(q_1)|\le C |q_1(y,s)|^{p+1}
+
C \delta_{\{p\ge 2\}} |K^*_1(y,s)|^{p-2}|q_1(y,s)|^3,
\end{equation}
which follows in the same way as \eqref{boundfq1} (note that unlike in \eqref{boundfq1}, we put the $\delta_{\{p\ge 2\}}$ symbol in front of the second term here).
Therefore,
we only prove \eqref{eq2} here.\\
The lower bound in \eqref{eq2} is straightforward from the Hardy-Sobolev estimate of Lemma \ref{lemhs} and the boundedness of $\|\kappa(\theta d^*(s) e_i)\|_{\q H}$ given in Lemma \ref{lemkd}.\\
As for the upper bound, we will apply the one-soliton multi-dimensional version of \cite{MZtams14} locally near each of the 4 solitons, then glue the estimates together, as for the multi-solitons case in one space dimension, treated in Appendix B of \cite{MZajm12}. However, it happens that the method of \cite{MZajm12} is too specific to one space dimension, since it uses the one-dimensional change of variables
\[
\bar w(\xi) = (1-y^2)^{\frac 1{p-1}}w(y)\mbox{ with }\xi = \arg\tanh y\in \m R.
\]
Nevertheless, the one-dimensional case will inspire us, in the sense that near each of the 4 solitons, we will use truncations as functions of the variables
\begin{equation}\label{defxi}
\xi_1= \arg\tanh y_1\mbox{ and }\xi_2= \arg\tanh y_2.
\end{equation}
Let us give in the following some guidelines for the two-dimensional strategy, avoiding purely technical considerations.

\medskip

Introducing the following version of \eqref{defphib} localized near the soliton $\kappa^*(\theta d(s) e_i,\nu(s))$, where $i=1,2$ and $\theta =\pm 1$:
\[
\varphi_{i,\theta}\left(q, r\right)= \int_{|y|<1} \left(-\psi_{i,\theta}(y,s)q_1r_1+\nabla q_1\cdot \nabla r_1-(y\cdot \nabla r_1)(y\cdot \nabla q_1)+q_2r_2\right)\rho dy,
\]
and
\[
\psi_{i,\theta}(y,s) = p\kappa^*_1(d(s),\nu(s), y)^{p-1} - \frac{2(p+1)}{(p-1)^2},
\]
we know from \eqref{conmod} that Proposition 2.6 in \cite{MZtams14} applies and yields for any $r\in \q H$:
\begin{equation}\label{coer1c}
\varphi_{i,\theta}(r,r)\ge \frac{\|r\|_{\q H}^2}{C_0} - \sum_{l=0}^2 |\pp_l(\theta d^*,r)|^2,
\end{equation}
for some universal constant $C_0$, where $\pp_l$ is introduced in \eqref{defpdi}.\\
Then, we consider some $A>0$ to be fixed large enough later, and 
 introduce the following cut-off function localized near the soliton $\kappa^*(d(s), \nu(s),y)$ and defined by
\begin{equation}\label{defcit}
\chi_{i,\theta}(y,s) = \chi\left(\frac{\theta \xi_i-\zeta^*(s)}A\right),
\end{equation}
where $\xi$ in introduced in \eqref{defxi}, $\zeta^*(s) = -\arg\tanh d^*(s)$, $d^*(s) = \frac{d(s)}{1+\nu(s)}$, and  
the increasing function $\chi\in C^\infty(\m R)$ satisfies 
\[
\forall \xi< 2,\;\;\chi(\xi) =0\mbox{ and }\forall \xi>-1,\;\;\chi(\xi)=1. 
\]
Introducing the localizing cut-off outside the solitons
\begin{equation}\label{defc0}
\chi_0=\sqrt{1-\sum_\ind \chi_{i,\theta}^2},
\end{equation}
we write
\begin{align}
\bar \varphi(q, q) & = P-S-U+\sum_\ind P_{i,\theta}+Q_{i,\theta}-S_{i,\theta}-U_{i,\theta},\label{phichi}\\
\|q\|_{\q H}^2 &=\|q \chi_0\|_{\q H}^2-S-U + \sum_\ind \|q \chi_{i,\theta}\|_{\q H}^2-S_{i,\theta}-U_{i,\theta},\label{nchi}
\end{align}
where
\begin{align*}
P_0&= \bar \varphi\left(q\chi_0, q\chi_0\right),\;\;
P_{i,\theta} = \varphi_{i,\theta}\left(q\chi_{i,\theta},  q\chi_{i,\theta}\right),\\
Q_{i,\theta}&=p\iint \chi_{i,\theta}^2q_1^2\left(\kappa(\theta d(s) e_i,\nu)^{p-1}-|K^*_1(y,s)|^{p-1}\right)\rho dy,\\
S_0&=\frac 12 \iint q_1^2\left(|\nabla \chi_0|^2-(y\cdot \nabla \chi_0)^2\right)\rho dy,\\
S_{i,\theta}&=\frac 12 \iint q_1^2\left(|\nabla \chi_{i,\theta}|^2-(y\cdot \nabla \chi_{i,\theta})^2\right)\rho dy\\
U_0&=\iint q_1 \chi_0\left(\nabla q_1 \cdot \nabla \chi_0 - (y\cdot \nabla q_1)(y \cdot \nabla \chi_0)\right)\rho dy,\\
U_{i,\theta}&=\iint q_1\chi_{i,\theta}\left(\nabla q_1 \cdot \nabla \chi_{i,\theta} - (y\cdot \nabla q_1)(y \cdot \nabla \chi_{i,\theta})\right)\rho dy.
\end{align*}
Let us first estimate the main terms, namely $P_{i,\theta}$ localized near the ``center'' of the soliton $\kappa^*(\theta d,\nu)$, and $P_0$, localized in the middle, where no soliton is important. Then, we will estimate the perturbation term $Q_{i,\theta}$ and finally the cut-off terms supported where the truncation is not zero, namely, $S_0$, $U_0$, $S_{i,\theta}$ and $U_{i,\theta}$.

\medskip

- {\it Estimate of $P_0$}:\\
In the middle, no soliton is important, and the quadratic form is equivalent the the square of the norm of $\q H$. More precisely,
 using the definition \eqref{defk*1} of $K^*_1(y,s)$ and Lemma \ref{lemkd} for $s_0$ large enough, we see that for $\max(|\xi_1|, |\xi_2|) \le \zeta^*(s) -A$ where $(\xi_1,\xi_2)$ is given in \eqref{defxi}, we have
\begin{align*}
(1-|y|^2)|K^*_1(y,s)|^{p-1}&\le 4^p \sum_\ind (1-|y_i|^2)\kappa(\theta d^*(s) e_i,y)^{p-1}\\
&= 4^p \kappa_0 \sum_\ind\cosh^{-2}(\theta\xi_i-\zeta^*(s))\le 4^{p+1}\cosh^{-2}A\le C e^{-2A}.
\end{align*}
Therefore, from \eqref{defc0} and the Hardy-Sobolev estimate of Lemma \ref{lemhs}, we have
\[
\int \chi_0(y)^2 |K^*_1(y,s)|^{p-1} \rho dy \le C e^{-2A} \int q_1^2 \frac \rho{1-|y|^2} dy \le C e^{-2A} \|q\|_{\q H}^2.
\]
Using the definition \eqref{defphib} of $\bar\varphi$, it follows that
\[
P_0
\ge \frac{\|\chi_0q\|_{\q H}^2}{C_1} -C e^{-2A} \|q\|_{\q H}^2,
\]
for some $C_1>0$.

\medskip

- {\it Estimate of $P_{i,\theta}$}:

We use the one-soliton estimate given in \eqref{coer1c}. Let us first estimate the projections involved in that estimate. If $l=2$, then the projection is zero from the symmetries of $q$. If $l=0,1$, we know from the modulation technique \eqref{mod} that $\pp_l(\theta d^*(s) e_i,q(s))=0$. Using Lemma \ref{lemkd} and \eqref{conseq}, 
we see that
\begin{align*}
&|\pp_l(\theta d^*(s) e_i,q(s)\chi_{i,\theta})|=
|\pp_l(\theta d^*(s) e_i,q(s) (1-\chi_{i,\theta}))|\\
\le& C\|q(s)\|_{\q H} \left(\int_{-1}^{\tanh (\zeta^*(s)-A)}\kappa(d^*(s), z)^2 (1-z^2)^{\frac 2{p-1}-1} dz\right)^{\frac 12}.
\end{align*}
Performing the change of variables $Z = \frac{z +d^*}{1+d^*z}$, we see that the integral is equal to $\ds\int_{-1}^{-\tanh A} (1-Z^2)^{\frac 2{p-1}-1}dZ\le Ce^{-2A}$.
Therefore, from the one-soliton estimate given in \eqref{coer1c}, we see that
\[
P_{i,\theta}\ge \frac{\|q\chi_{i,\theta}\|_{\q H}^2}{C_0} -
C e^{-2A}\|q\|_{\q H}^2.
\]

- {\it Estimate of $Q_{i,\theta}$}: From symmetry, we take $(i,\theta)=(1,1)$. Since the support of $\chi_{1,1}$ is included in $Q_{1,1}$ \eqref{defqit}, where the soliton $\kappa^*(d(s) e_1)$ is dominant, making a Taylor expansion and using Lemma \ref{lemkd}, 
we see that
\begin{align*}
|Q_{1,1}|\le & C\sum_{(j,\eta)\neq(1,1)}
\int_{\xi_1 \ge \zeta^*(s) -2A}
q_1^2\kappa(d^*(s)e_1,y)^{p-2}\kappa(\eta d^*(s) e_j,y) \rho dy\\
\le &C(1-|d^*|)
\iint
q_1^2 \frac \rho{(1+ d^* y_1)^{\frac {2(p-2)}{p-1}}} dy\\
\le & C(1-|d^*|)\iint q_1^2 \frac \rho{1-|y|^2}(1+d^*y_1)^{\frac{3-p}{p-1}}dy,
\end{align*}
because we have $1-|y|^2 \le 2(1+d^*y_1)$. Since $(1+d^* y_1)^{\frac{3-p}{p-1}}\le C\max (1,(1-|d^*|)^{\frac{3-p}{p-1}})$, 
using the Hardy-Sobolev estimate of Lemma \ref{lemhs}, we see from \eqref{conseq} that
\[
|Q_{i,\theta}|\le C(1-|d^*|)^{\min (\frac 2{p-1}, 1)}\|q\|_{\q H}^2\le C e^{-\min (\frac 4{p-1}, 2)\zeta(s)}\|q\|_{\q H}^2.
\]

- {\it Estimate of the cut-off terms $S_0$ and $S_{i,\theta}$}:

Since we have by definitions
\eqref{defcit} and 
\eqref{defc0},
\[
0\le |\nabla \chi_{i,\theta}|^2 -(y\cdot \nabla \chi_{i,\theta})^2 \le \frac C{A^2}
\mbox{ and }
0\le |\nabla \chi_0|^2 -(y\cdot \nabla \chi_0)^2 \le \frac C{A^2},
\]
it follows that
\[
|S_{i,\theta}|+|S_0|\le  \frac C{A^2}\iint q_1^2 \rho dy \le  \frac C{A^2}\|q\|_{\q H}^2.
\]
 
- {\it Estimate of the cut-off terms $U_0$ and $U_{i,\theta}$}:

Using integration by parts, we write from the definition \eqref{defro} of the operator $\q L$:
\[
U_{i,\theta}= \frac 14 \iint \left(\nabla q_1^2\cdot \nabla \chi_{i,\theta}^2-(y\cdot\nabla q_1^2)(y\cdot \nabla \chi_{i,\theta}^2)\right)\rho dy
=-\frac 14 \iint q_1^2 \q L \chi_{i,\theta}^2 \rho dy
\]
and 
\[
U_0 = -\frac 14 \iint q_1^2 \q L \chi_0^2 \rho dy.
\]
Expanding the expression \eqref{defro} as follows:
\[
\q L v = \Delta v - \sum_{i,j}y_iy_j \partial^2_{y_i,y_j}v - \frac{2(p+1)}{p-1} y\cdot \nabla v,
\]
we see that
\[
|\q L \chi_{i,\theta}^2|+|\q L\chi_0^2|\le \frac C{A(1-|y|^2)}.
\]
Therefore, using the Hardy-Sobolev estimate of Lemma \ref{lemhs}, we see that
\[
|U_{i,\theta}|+|U_0|\le \frac CA \iint q_1^2\frac \rho{1-|y|^2}dy. 
\]
 - {\it Conclusion of the proof of \eqref{eq2}}: From the expansions \eqref{phichi} and \eqref{nchi}, we see from the above estimates that
\begin{align*}
\bar \varphi(q,q) &\ge \frac 1{C_2}\left(\|q\chi_0\|_{\q H}^2 + \sum_\ind \|q\chi_{i,\theta}\|_{\q H}^2\right)-\|q\|_{\q H}^2\left(\frac CA+Ce^{-\min(\frac 4{p-1}, 2)\zeta(s)}\right),\\
\|q\|_{\q H}^2 &\le \|q\chi_0\|_{\q H}^2 + \sum_\ind \|q\chi_{i,\theta}\|_{\q H}^2
+\|q\|_{\q H}^2\left(\frac CA+Ce^{-\min(\frac 4{p-1}, 2)\zeta(s)}\right),
\end{align*}
for some constant $C_2>0$. Fixing $A>0$ large enough, then assuming that $s_0$ is large enough, we see from \eqref{conmod} that \eqref{eq2}. This concludes the proof of Lemma \ref{lemquad}.
\end{proof}

\bigskip

{\bf Step 2: A differential inequality satisfied by $h_2(s)$}.

This step is dedicated to the proof of the second part of \eqref{lemlyap}.
The proof follows the proof of Claim 4.8 page 2898 in \cite{MZdmj12} (for the multi-soliton aspect) and the proof of Proposition 3.6 in \cite{MZtams14} (for the multi-dimensional aspect). For that reason, we will recall estimates from those papers, and only stress the novelties.
 We claim that estimate \eqref{lemlyap} follows from the following statement, together with Lemma \ref{lemquad}:
%%%%%%%%%%%%%%%%%%%%%%%%%%%%%%%%%%%%%%%%%%%%%
%%%%%%%%%%%%%%%%%%%%%%%%%%%%%%%%%%%%%%%%%%%%%
\begin{lem}\label{lemproj*} There exists $s_6\ge 0$ such that if $s_0\ge s_6$, 
then for all $s\in [s_0, \bar s)$, we have:
\begin{align}
\frac 12 \frac d{ds}\bar\varphi(q,q)\le &- \alpha\iint q_{2}^2 \frac \rho{1-|y|^2} dy+\iint q_2f(q_1) \rho dy
+\frac C{\eta}e^{-\frac{4\bar p\zeta}{p-1}}\nonumber\\
&+\bar \varphi(q,q) \left(\frac{\eta}{10}+Cs_0^{-\frac 14}\right),\label{eq1}\\
- \frac d{ds} \iint \q F(q_1)\rho dy\le&-\iint q_2f(q_1) \rho dy 
+C{s_0}^{-\frac 14}\bar \varphi(q,q),\label{eq0}\\
\frac d{ds}\iint q_1q_2 \rho \le&
-\frac 7{10}\bar\varphi(q,q) +C \iint q_{2}^2 \frac \rho{1-|y|^2}dy+ Ce^{-\frac{4\bar p\zeta}{p-1}},\label{eq3}\\
-\frac d{ds} e^{-\frac{4\bar p \zeta}{p-1}} \le& C s_0^{-\frac 14}e^{-\frac{4\bar p \zeta}{p-1}},\label{eqexp}
\end{align}
where $\alpha>0$ is defined in \eqref{defro}.
\end{lem}
%%%%%%%%%%%%%%%%%%%%%%%%%%%%%%%%%%%%%%%%%%%%%
%%%%%%%%%%%%%%%%%%%%%%%%%%%%%%%%%%%%%%%%%%%%%
\noindent Indeed,
 if $s_0$ is taken larger enough so that Lemmas \ref{lemquad} and \ref{lemproj*} hold, together with the two first estimates in Lemma \ref{propdyn}, then 
we see by definition \eqref{defh} of $h_1(s)$ and $h_2(s)$  
that for all $s\in [s_0, \bar s)$,
\begin{align*}
|2h_1(s) -  \bar \varphi(q,q)|\le &
C(s_0^{(1-\min(p,2))/2}+\eta)\|q\|_{\q H}^2
\le C(s_0^{(1-\min(p,2))/2}+\eta) \bar \varphi(q,q) ,\\
h_2'(s)\le& \left(-\alpha+C\eta\right)\iint q_2^2 \frac \rho{1-|y|^2} dy 
+\left(Cs_0^{-\frac 14}-\frac 6{10}\eta\right)\bar \varphi(q,q)\\
&+C(\eta^{-1}+s_0^{-\frac 14} \eta^{-2}) e^{-\frac{4\bar p\zeta}{p-1}}.
\end{align*}
If $\eta$ is small enough and $s_0$ large enough, then we may make the multiplying factor in front of $\iint q_2^2 \frac \rho{1-|y|^2} dy $ in the above inequality negative, and derive by definition \eqref{defh} of $h_2$:
\begin{align*}
3h_1(s) \ge& \bar \varphi(q,q) \ge h_1(s) = h_2(s) + \frac {\ds e^{-\frac{4\bar p \zeta(s)}{p-1}}}{\eta^2}, \\
h_2'(s)\le& \left(Cs_0^{-\frac 14}-\frac 6{10}\eta\right)h_2(s)
+C(\eta^{-1}+s_0^{-\frac 14} \eta^{-2}-2\eta^{-2}\left(\frac 6{10}\eta -Cs_0^{-\frac 14}\right)) e^{-\frac{4\bar p\zeta}{p-1}}.
\end{align*}
Since $\eta^{-1}-2\eta^{-2}\frac 6{10}\eta=-\frac 1{5\eta}$,
fixing first $\eta=\eta(p)>0$ small enough,
then
fixing $s_0$ large enough,
we see from \eqref{eq2} that \eqref{lemlyap} holds, and so does \eqref{est:q}
in Lemma \ref{propdyn}. It remains then to justify Lemma \ref{lemproj*} in order to conclude the proof of \eqref{lemlyap} and Lemma \ref{propdyn} too.

\bigskip

\begin{proof}[Proof of Lemma \ref{lemproj*}] 
  Following our techniques performed for the proof of Lemma C.2 page 2896 in \cite{MZdmj12}, 
using \eqref{conmod}, \eqref{conseq}, item (iii) in Lemma \ref{lemkd}, Lemma \ref{lemquad} and the Hardy-Sobolev estimate of Lemma \ref{lemhs}, we see that for all $s\in [s_0, \bar s)$, 
\begin{align}
\frac 12 \frac d{ds} \bar\varphi(q,q)\le& -\alpha \iint q_2^2 \frac \rho{1-|y|^2} dy%\nonumber\\
+
\iint q_2f(q_1) \rho dy
+C(R_1+R_2+R_3),\label{eqa-}\\
- \frac d{ds}\iint \q F(q_1) \rho dy \le& -\iint q_2f(q_1) \rho dy
+
 \frac{C|\nu|}{1-|d|^2}\iint \kappa(d^*e_1,y)|f(q_1)|\rho dy+CR_3,\label{eqnl}\\
\frac d{ds} \iint q_1 q_2 \rho dy \le& -\frac 9{10}\bar \varphi(q,q)+C\iint q_2^2 \frac \rho{1-|y|^2}+\iint q_1 f(q_1)\rho dy\nonumber\\
&+C(R_1+R_2),\label{q1q2}
\end{align}
where $\alpha>0$ is defined in \eqref{defro}, 
\begin{align}
R_1(s)&=\iint R(y,s)^2\rho(1-|y|^2) dy,\;\;
 R_2(s) =\frac{\|q\|_{\q H}}{1-|d|^2}\left(|\nu'-\nu|+ |d'|\right),\nonumber\\
R_3(s)&=\frac{\left(|\nu'|+ |d'|\right)}{1-|d|^2}\iint \kappa(d^* e_1,y)|K^*_1|^{p-2}q_1^2 \rho dy.\label{defR1}
\end{align}
In the following, we further estimate some terms appearing in the right-hand side of the previous equations, in order to derive Lemma \ref{lemproj*}.

\medskip

{\it - Estimate of $R_1=\iint R(y,s)^2\rho(1-|y|^2) dy$}: Using the defintion \eqref{defk*1} of $R(y,s)$ and proceeding as for \eqref{estR}, we obtain the following less refined estimate:
\[
|R(y,s)|\le C \sum_\ind (1-1_{Q_{i,\theta}})\kappa(\theta d^*(s)e_i,y)^p
+1_{Q_{i,\theta}}\kappa(\theta d^*(s) e_i,y)^{p-1}\sum_{(j,\eta)\neq (i,\theta)}\kappa(\eta d^*(s) e_j,y).
\]
Arguing as for the various integrals involved in \eqref{p0r}, using in particular the integral table given in item (i) of Lemma \ref{cltech0}, we see that
\begin{equation}\label{R1}
|R_1|=\iint R(y,s)^2\rho(1-|y|^2) dy\le C e^{-\frac{4\bar p}{p-1}\zeta(s)},
\end{equation}
where $\bar p$ is introduced in \eqref{defpb}.

\medskip

- {\it Estimate of $R_2$ and $R_3$}: Using the equations satisfied by $\zeta(s) = -\arg \tanh d(s)$ and $\nu(s)$ already proved in Lemma \ref{propdyn}, together with \eqref{estnl}, we see that
\begin{align}
|R_2|&\le C\|q\|_{\q H}\left(\|q\|_{\q H}^2 + e^{-\frac{4\zeta}{p-1}}+\frac{|\nu|}{1-|d|^2}\|q\|_{\q H}\right),\nonumber\\
|R_3|&\le \|q\|_{\q H}^2\left(\|q\|_{\q H}^2 + e^{-\frac{4\zeta}{p-1}}+\frac{|\nu|}{1-|d|^2}\right). \label{R3}
\end{align}

\medskip

- {\it Estimate of $\iint q_1 f(q_1)\rho dy$}: It happens that $q_1 f(q_1)$ satisfies the same bound as $\q F(q_1)$ in \eqref{boundFq1}. Therefore, by the same argument one uses for \eqref{eq2a} (which is in fact a simplification of the argument we used for \eqref{estnl}), we have the same bound as in \eqref{eq2a}, namely that 
\begin{equation}\label{q1fq1}
\left|\iint q_1 f(q_1)\rho dy\right|\le C\|q\|_{\q H}^{1+\min(p,2)}\le s_0^{(1-\min(p,2))/2}\|q\|_{\q H}^2.
\end{equation}
 
\medskip

- {\it Conclusion of the proof of Lemma \ref{lemproj*}}: 
Since $\|q\|_{\q H} e^{-\frac{4\zeta}{p-1}}\le \frac \eta{10} \bar \varphi(q,q) + \frac C\eta e^{-\frac{8\zeta}{p-1}}$ from Lemma \ref{lemquad}, 
using \eqref{conmod}, \eqref{R1} and \eqref{R3}, we see that \eqref{eq1} follows from \eqref{eqa-}. As for \eqref{eq0}, one needs in addition to use \eqref{estnl} to derive it from \eqref{eqnl}. Using \eqref{q1fq1} and Lemma \ref{lemquad}, we also get \eqref{eq3}. Finally, identity \eqref{eqexp} follows from the second estimate (already proved) in Lemma \ref{propdyn}. This concludes the proof of Lemma \ref{lemproj*}.
\end{proof}
Since we have explained right after the statement of Lemma \ref{lemproj*} how to derive \eqref{est:q}, which is the last estimate in Lemma \ref{propdyn}, this concludes the proof of Lemma \ref{propdyn} too.
\end{proof}

%%%%%%%%%%%%%%%%%%%%%%%%%%%%%%%%%%%%
%%%%%%%%%%%%%%%%%%%%%%%%%%%%%%%%%%%%
\section{Computation tool box}
%%%%%%%%%%%%%%%%%%%%%%%%%%%%%%%%%%%%
%%%%%%%%%%%%%%%%%%%%%%%%%%%%%%%%%%%%
In this section, we recall several estimates from our earlier work, and use them to prove some estimates related to the solitons $\kappa(d,y)$ \eqref{defkd} and $\kappa^*(d,\nu,y)$ \eqref{defk*}.

\medskip

We first recall the following Hardy-Sobolev inequality:
%%%%%%%%%%%%%%%%%%%%%%%%%%%%%%%%%%%%
%%%%%%%%%%%%%%%%%%%%%%%%%%%%%%%%%%%%%
\begin{lem}[A Hardy-Sobolev inequality]\label{lemhs} For any $v\in \q H_0$ and $\gamma \ge 2$, we have
\begin{equation*}%\label{hs}
\|v\|_{L^2_{\frac \rho{1-|y|^2}}}+\|v\|_{L^{p+1}_\rho}
+\|v\|_{L^\gamma_{ (1-|y|^2)^{\frac \gamma{p-1} - \frac 12}}}
\le C  \|v\|_{\q H_0}.
\end{equation*}
In one space dimension, we further have $\|v(1-|y|^2)^{\frac 1{p-1}}\|_{L^\infty}\le C  \|v\|_{\q H_0}.$
\end{lem}
%%%%%%%%%%%%%%%%%%%%%%%%%%%%%%%%%%%%
%%%%%%%%%%%%%%%%%%%%%%%%%%%%%%%%%%%
\begin{proof} The first term is bounded by the left-hand side thanks to Appendix B in \cite {MZajm03}, and the second thanks to Lemma E.1 in \cite{MZtams14}. As for the third, it follows from an intermediate estimate in that lemma. More precisely,

making the following change of variables:
\[
h(y) = v(z) \mbox{ with }\frac y{|y|}= \frac z{|z|}\mbox{ and }|y| = \psi(|z|)=1-\sqrt{1-|z|},
\]
then, introducing $a=2\alpha+1=\frac 4{p-1}$ (note that $a\neq 1$ and $a\ge 0$
since $1<p<5$ and $\frac sC\le \psi(s) \le Cs$, for all $s\in (0,1)$), we reduce the aimed identity to the following
 \begin{equation}\label{inter}
\left(\iint |h|^\gamma(1-|y|^2)^{\gamma a/2} dy \right)^{2/\gamma}\le C \iint \left(|\nabla h|^2+ h^2\right) (1-|y|^2)^ady
\end{equation}
(see the proof of Lemma E.1 in \cite{MZtams14} for details in a similar computation). It happens that identity \eqref{inter} is true thanks to Lemma E.4 in that paper. This concludes the proof of Lemma \ref{lemhs}.
\end{proof}

In the following lemma, we recall some properties of the solitons $\kappa(\bs d,y)$ and $\kappa^*(\bs d, \nu,y)$ defined in \eqref{defkd} and \eqref{defk*}, in particular some properties related to the Lyapunov functional $E(w, \partial_s w)$ defined above in \eqref{defenergy}. This is the statement:
%%%%%%%%%%%%%%%%%%%%%%%%%%%%%%%%%%%%
%%%%%%%%%%%%%%%%%%%%%%%%%%%%%%%%%%%%
\begin{lem}[Properties of the solitons] \label{lemkd} For any $|\bs d|<1$, we have:\\
(i) $\frac 1C\le \|\kappa(\bs d)\|_{\q H_0}\le C$.\\
(ii) For $l=0,1$ and any $|y|<1$, $|W_{l,2}(\bs d,y)|\le C \kappa(\bs d,y)\le C(1-|y|^2)^{-\frac 1{p-1}}$ and $|\q L W_{l,1}(\bs d,y) - W_{l,1}(\bs d,y)|\le C\frac {\kappa(\bs d,y)}{1-|y|^2}$.\\
(iii) For all $\nu>-1+|\bs d|$ and $|y|<1$, we have
\begin{align*}
\kappa^*_1(\bs d, \nu, y)&= \lambda(|\bs d|, \nu) \kappa(\bs d^*,y),\\
\ds\left\|\kappa^*\left(\bd,\nu\right)\right\|_{\H} &\le C\lambda(|\bs d|, \nu)  +C1_{\{\nu<0\}}\frac{|\nu|}{\sqrt{1-|\bd|^2}}\lambda(|\bs d|, \nu)^{\frac{p+1}2},
\end{align*}
 where $\lambda(|\bs d|, \nu)$ is defined in \eqref{deflambda}
and $\bs d^*= \frac{\bs d}{1+\nu}$. In addition, if $\frac{|\nu|}{1-| d|}\le \frac 12$, then we have
$\left|\lambda( d, \nu) -1\right|\le \frac{C|\nu|}{1-|d|}$. Moreover, if $\mu(d,\nu)$ is introduced in \eqref{defmu}, then we see that
\begin{equation}\label{compare0}
\min\left(\mu,\mu^2\right)\le \lambda^{p-1}\le \max\left(\mu,\mu^2\right).
\end{equation}
(iv) {\bf (Continuity of $\kappa^*$ along a given direction)} If $\omega\in \m R^2$ with $|\omega|=1$, and $(d_1,\nu_1)$ and $(d_2, \nu_2)$ satisfy 
\begin{equation}\label{condA}
\left|\frac {\nu_1}{1-|d_1|}\right|\le 1\mbox{ and }\left|\frac {\nu_2}{1-{|d_2|}}\right|\le 2,
\end{equation}
then
\begin{multline}\label{defze}
\|\kappa^*(d_1\omega,\nu_1)-\kappa^*(d_2\omega, \nu_2)\|_{\H} \\
\le C\left(\left| \frac {\nu_1}{1-|d_1|} -\frac {\nu_2}{1-{|d_2|}}\right| +\left|\arg\tanh d_1 - \arg\tanh d_2\right|\right).
\end{multline}
(v) For any $|\bd|<1$ and $\nu>-1+|\bd|$, we have
\[ 
E(\kappa_0,0) \ge E(\kappa^*(\bd,\nu))=E(\kappa_0,0)\left(\frac{p+1}{p-1}\lambda^2-\frac 2{p-1} \lambda^{p+1}+\frac 2{(p-1)}\frac{\nu^2}{(1-|\bd|^2)}\lambda^{p+1}\right),
\]
where $\lambda = \lambda(|\bd|,\nu)$ is defined in \eqref{deflambda}.\\
(vi)  For all $B\ge 2$, if $|d|<1$ and $-1+\frac 1B \le \frac{\nu}{1-|d|}\le B$, then $\|\partial_\nu \kappa^*(d,\nu)\|_{\q H}+\|\nabla_d \kappa^*(d,\nu)\|_{\q H}\le \frac {C(B)}{1-|d|}$.
\end{lem}
%%%%%%%%%%%%%%%%%%%%%%%%%%%%%%%%%%%%
%%%%%%%%%%%%%%%%%%%%%%%%%%%%%%%%%%%%
\begin{proof}$ $\\
(i) From item (iv) in Claim A.2 page 47 in \cite{MZtams14} and the definition of the similarity variables' version for the Lorentz transform in that paper, we see that $\frac 1C\|\kappa_0\|_{\q H_0}\le \|\kappa(\bs d)\|_{\q H_0}\le C\|\kappa_0\|_{\q H_0}$, and the result follows.\\
(ii) The first identity is straightforward from the definitions \eqref{defWl2} and \eqref{defkd}. The second identity follows from the one-dimensional case of the Hardy-Sobolev identity given in Lemma \ref{lemhs}.\\
(iii) For the beginning of the item, see item (i) in Lemma A.2 and identity (B.10) in \cite{MZdmj12}. For estimate \eqref{compare0}, first write by definitions \eqref{deflambda} and \eqref{defmu} of $\lambda$ and $\mu$ 
\begin{align*}
\lambda^{-(p-1)}= \frac{(1+\nu)^2 - |d|^2}{1-|d|^2}&=\left(1+\frac \nu{1-|d|}\right)\left(1+\frac \nu{1+|d|}\right)\\
&=\frac 1\mu\left(1+\left(\frac 1\mu-1\right) \frac{1-|d|}{1+|d|}\right).
\end{align*}
Then, according to the position of $\mu$ with respect to $1$, we may use the monotonicity of this expression with respect to $|d|$ in order to get \eqref{compare0}.\\
(iv) From the remark given around estimate \eqref{defnhz1}, this reduces to a one-dimensional estimate. Hence, Lemma A.2 in \cite{MZdmj12} applies and gives the result.\\
(v) By definitions \eqref{defk*} and \eqref{defenergy} of $\kappa^*(\bd, \nu)$ and $E(w,\partial_s w)$, this reduces to the one-dimensional case, already treated in item (i) in Lemma A.2 page 2878 in \cite{MZdmj12}.\\
(vi) See Claim F.2 page 81 in \cite{MZtams14}.
\end{proof}

We also recall the following computation table for integrals from \cite{MZjfa07} and \cite{MZajm12}:
%%%%%%%%%%%%%%%%%%%%%%%%%%%%%%%%%%%%%%%%%%%%%
%%%%%%%%%%%%%%%%%%%%%%%%%%%%%%%%%%%%%%%%%%%%%
\begin{lem}[Integral computation table]\label{cltech0}$ $\\
(i) Consider for some $\gamma>-1$ and $\beta\in \m R$ the following integral
\[
I(\delta)= \int_{-1}^1 \frac{(1-\xi^2)^\gamma}{(1+\delta \xi)^\beta}dy\mbox{ where }\delta\in(-1,1).
\]
Then, there exists $K(\gamma, \beta,N)>0$ such that the following limits hold as $|\delta|\to 1$:\\
- if $\gamma+1-\beta>0$, then $I(\delta)\to K$,\\
- if $\gamma+1-\beta=0$, then $I(\delta)|\log(1-|\delta|)|\to K$,\\
- if $\gamma+1-\beta<0$, then $I(\delta)(1-|\delta|)^{- (\gamma+1)+\beta}\to K$.\\
(ii) For all $d\in(-1,-\frac 12]$, we have 
$|e^{\frac 4{p-1}\zeta}\bar R_{1,1}(\delta)+2^{-\frac 2{p-1}}\barcc|
+|e^{\frac 4{p-1}\zeta}\bar R_{2,\pm 1}(\delta)+\barcc|
\le e^{-\bar \epsilon \zeta}$ for some $\barcc(p)>0$ and $\bar \epsilon>0$, where $\bar R_{i,\theta}$ is defined in \eqref{deftrit} and $\zeta=-\arg\tanh d$.\\
(iii) If $p<2$ and $\gamma>\frac 2{p-1}$, then for $|\delta+1|$ and $\eta$ small enough,  we have
$\bar B(\gamma, \eta, \delta)\le C$, where $\bar B$ is defined in \eqref{defb3}.
\end{lem}
%%%%%%%%%%%%%%%%%%%%%%%%%%%%%%%%%%%%%%%%%%%%%
%%%%%%%%%%%%%%%%%%%%%%%%%%%%%%%%%%%%%%%%%%%%%
\begin{proof}$ $\\
(i) See Claim 4.3 page 84 in \cite{MZjfa07}.\\
(ii) The estimate for $\bar R_{1,1}(\delta)$ follows with very minor changes from item (iii) of Lemma E.1 page 644 in \cite{MZajm12}. As for $\bar R_{2,\pm 1}(\delta)$,
using the change of variables 
\[
Y_1= \frac{\delta+y_1}{1+\delta y_1},\;\;Y_2 = \frac{\sqrt{1-\delta^2}}{1+\delta y_1}y_2,
\]
we see after sraightforward computations that
\begin{align}
\bar R_{1,1}(\delta)&=(1-\delta^2)^{\frac 2{p-1}}\int_{|Y|<1}Y_1(1+\delta^2-2\delta Y_1)^{-\frac 2{p-1}}\rho(Y) dY
\nonumber\\
\bar R_{2,\pm 1}(\delta)&=(1-\delta^2)^{\frac 2{p-1}}\int_{|Y|<1}Y_1(1-\delta Y_1+\delta\sqrt{1-\delta^2}Y_2)^{-\frac 2{p-1}}\rho(Y) dY.\nonumber
\end{align}
 Introducing the rotation of coordinates 
\[
Z_1=\frac{Y_1}{\sqrt{2-\delta^2}}-\sqrt{\frac{1-\delta^2}{2-\delta^2}}Y_2,\;\;
Z_2=\sqrt{\frac{1-\delta^2}{2-\delta^2}}Y_1+\frac{Y_2}{\sqrt{2-\delta^2}},
\]
we see that
\begin{align*}
\bar R_{2,\pm 1}(\delta)&=\frac {(1-\delta^2)^{\frac 2{p-1}}}{\sqrt{2-\delta^2}}\int_{|Z|<1}\frac{(Z_1+\sqrt{1-\delta^2}Z_2)\rho(Z)}{(1-\delta \sqrt{2-\delta^2} Z_1)^{\frac 2{p-1}}} dZ\\
&=\frac{(1+{\bar \delta}^2)^{\frac 2{p-1}}}{\sqrt{2-\delta^2}}\left(\frac{1-\delta^2}{1-\bar \delta^2}\right)^{\frac 2{p-1}}\bar R_{1,1}(\bar \delta),
\end{align*}
where $\bar \delta>-1$ is such that $\frac {2\bar \delta}{1+{\bar \delta}^2} = \delta \sqrt{2-\delta^2}$.
This gives $1+\bar \delta =2(1+\delta) +O(1+\delta)^2$ as $\delta \to -1$. Since $e^{-2\zeta} = \frac{1+\delta}2 + O(1+\delta)^2$ and the same holds between $\bar \delta$ and $\bar \zeta = -\arg \tanh \bar \delta$, the estimate for $\bar R_{2,\pm 1}$ follows from the estimate on $\bar R_{1,1}$.\\
(iii) From the definitions \eqref{defd3} and \eqref{defb3} of $\bar D$, $\bar B$ and $\bar K(y)$, it is easy to see that $\bar D \subset \{\max(|y_1|, |y_2|)\le \frac 9{10}\}$, hence, using the common denominator for $\bar K(y)$, we write
\begin{equation}\label{boundbb}
\bar B(\gamma, \eta, \delta) \le C \int_{\bar D(\eta)}\frac{dy}{|K(\delta y)|^{\frac{\gamma(2-p)}{\gamma-2}}\sqrt{1-|y|^2}},
\end{equation}
where 
\[
K(y) = \bar g(y_2)h(y_1^2) - \bar g(y_1)h(y_2^2),\;
\bar g(\xi) = (1+\xi)^{\frac 2{p-1}}+ (1-\xi)^{\frac 2{p-1}}\mbox{ and }
h(\xi) = (1-\xi)^{\frac 2{p-1}}.
\]
Expanding $\bar g$ in powers of $\xi$, it is easy to see that
\[
\bar g(\xi) = g(\xi^2),
\]
where $g$ is a convergent power series on the unit disc. Therefore, 
\begin{equation}\label{defke}
K(y)= K^*(y_1^2, y_2^2),\mbox{ where }K^*(z_1,z_2)= g(z_2)h(z_1) - g(z_1)h(z_2)
\end{equation}
is $C^\infty$. Given $z_1$ and $z_2$ in $(0,1)$, noting that $K^*(\frac{z_1+z_2}2,\frac{z_1+z_2}2)=0$ and applying the mean value theorem to the function 
$\theta \mapsto K^*\left((1-\theta)\left(\frac{z_1+z_2}2,\frac{z_1+z_2}2\right)
+\theta(z_1,z_2)\right)$, 
we see that
\begin{equation}\label{minor}
|K^*(z_1,z_2)| =|z_1-z_2|\left|\partial_{ e^*}K^*\left((1-\theta^*)\left(\frac{z_1+z_2}2,\frac{z_1+z_2}2\right)+\theta^*(z_1,z_2)\right)\right|
\end{equation}
for some $\theta^*\in [0,1]$, where $e^*=(-1,1)$. Since $p<2$, hence $\frac 2{p-1}>2$, and for all $z\in [0,\frac{99}{100}]$,
\begin{align*}
\partial_{e^*}K(z,z) &=-2 g(z) h'(z)+2g'(z)h(z)
=-2 \bar g(\sqrt z)h'(z) +\frac{\bar g'(\sqrt z)}{\sqrt z} h(z)\\
&=\frac 4{p-1}\left((1+\sqrt z)^{\frac 2{p-1}}+(1-\sqrt z)^{\frac 2{p-1}}\right)(1-z)^{\frac 2{p-1}-1}\\
&+\frac 2{(p-1)\sqrt z}\left((1+\sqrt z)^{\frac 2{p-1}-1}-(1-\sqrt z)^{\frac 2{p-1}-1}\right)(1-z)^{\frac 2{p-1}}\\
&\ge 2\beta_0
\end{align*}
for some $\beta_0>0$, we see from \eqref{minor}, \eqref{defke} and \eqref{defd3} that for $|\delta+1|$ and $\eta>0$ small enough, we have for all $y\in \bar D(\eta)$, 
\[
|K(\delta y)|\ge \beta_0|y_1^2-y_2^2|.
\]
Since $p<2$ and $\gamma> \frac 2{p-1}$, hence $\frac{\gamma(2-p)}{\gamma-2}<1$, using this bound and performing a change of variables $y\mapsto \bar y = \left(\frac{y_1+y_2}{\sqrt 2}, \frac{y_1-y_2}{\sqrt 2}\right)$, we see that the integral in \eqref{boundbb} is uniformly bounded for $|\delta+1|$ and $\eta>0$ small enough.
\end{proof}

\bigskip

Now, we give some calculations related to the solitons $\kappa(d,y)$ \eqref{defkd}. More precisely, given that $\|\kappa(d)\|_{L^{p+1}_\rho}\le C$, for all $|d|<1$, thanks to item (i) of Lemma \ref{lemkd} and the Hardy-Sobolev inequality of Lemma \ref{lemhs}, it is natural to look for some uniform lower bound. In the following lemma, we do better, and find an ellipse whose area shrinks to zero as $|d|\to1$, and where the $L^{p+1}_\rho$ of $\kappa(d)$ is constant. This is our statement (from rotation invariance, we state the result only when the soliton parameter is of the form $(d,0)$ with $d\in (-1,1)$):
%%%%%%%%%%%%%%%%%%%%%%%%%%%%%%%%%%
%%%%%%%%%%%%%%%%%%%%%%%%%%%%%%%%%%
\begin{lem}[A small domain supporting the $L^{p+1}_\rho$ norm of $\kappa(d)$]\label{lemlorint}
For all $d\in (-1,1)$, we have 
\begin{equation*}%\label{lorint}
\int_{\EE(d)}\kappa(de_1,Y)^{p+1} \rho(Y)dY=\int_{|y|<\frac 12} \kappa_0^{p+1}\rho(y)dy
\end{equation*}
where the integration domain $\EE(d)\subset B(0,1)$ is the ellipse centered at $(c(d),0)=(-\frac{3d}{4-d^2},0)$ whose horizontal axis is $2a(d)=\frac{4(1-d^2)}{4-d^2}$ and its vertical axis is $2b(d)=2\sqrt{\frac{1-d^2}{4-d^2}}$
\end{lem}
%%%%%%%%%%%%%%%%%%%%%%%%%%%%%%%%%%
%%%%%%%%%%%%%%%%%%%%%%%%%%%%%%%%%%
\begin{proof}The result comes from sraightforward calculations based on the following similarity variables' version of the Lorentz change of coordinates, already introduced in our earlier work \cite{MZtams14}:
\[
y_1=\frac{Y_1+d}{1+dY_1} \mbox{ and }y_2=Y_2\frac{\sqrt{1-d^2}}{1+dY_1}.
\]
\end{proof}

%%%%%%%%%%%%%%%%%%%%%%%%%%%%%%%%%%
%%%%%%%%%%%%%%%%%%%%%%%%%%%%%%%%%%
\section{Details for the soliton-loosing mechanism and a generalized trapping result}
\label{apptx}
%%%%%%%%%%%%%%%%%%%%%%%%%%%%%%%%%%
%%%%%%%%%%%%%%%%%%%%%%%%%%%%%%%%%%
Many of our arguments 
dedicated to the geometry of the blow-up set and the local behavior come from our earlier work performed in \cite{MZdmj12} in the one-dimensional case, or in the higher-dimensional case in \cite{MZtams14} (specifically, the trapping result). In order to be nice with the expert reader, we removed them from
the main part of the paper, 
and give them here in the Appendix, for the non-expert reader convenience. In fact, in order to keep the paper into reasonable limits, we will only give the key steps and refere the reader to our earlier papers for details.

\medskip

In particular, we give here the proof of Proposition \ref{prop422} and Lemma \ref{proptrap}.

%%%%%%%%%%%%%%%%%%%%%%%%%%%%%%%%%%%
%%%%%%%%%%%%%%%%%%%%%%%%%%%%%%%%%%%
\subsection{Estimates on $w_x$}
%%%%%%%%%%%%%%%%%%%%%%%%%%%%%%%%%%%
%%%%%%%%%%%%%%%%%%%%%%%%%%%%%%%%%%%
 In this section, we prove Proposition \ref{prop422}.

\begin{proof}[Proof of Proposition \ref{prop422}]
As we did in the one dimensional case treated in \cite{MZdmj12}, the proof is done in two steps:\\
- Using the transformation $\q T_x$ introduced in \eqref{defw**}, we bound the $\q H$ norm of $r$ when the unit ball is replaced by a smaller set; more precisely,
we show the following:
\begin{equation}\label{ident*}
\forall s\in [s_3, \sl(x,A)],\;\;\left\|r(x,\ts)\right\|_{\H(\ty \in B_-(\ts))} \le C \left\|R(\cs(x,s))\right\|_{\H}+\frac CS,
\end{equation}
where $\sl(x,A)$ is defined in \eqref{defslsd3} (depending on $A>0$ large enough),  $B_-(\ts)$ is some subset of $B(0,1)$, 
$R(Y,S)$ is defined by 
the following:
\begin{equation}\label{defg*}
R(Y,S)=\vc{W(\cy,\cs)}{\ps W(\cy,\cs)}-\sum_{dir \in \q D} \tdir\vc{\kappa(\bar d(\cs)\edir,\cy)}{0},
\end{equation}
 and $r(x,y,s)$ by \eqref{defh*};\\
- Then, using our fine knowledge of the the behavior of solutions to equation \eqref{eqw} near a decoupled sum of solitons (see \cite{MZdmj12}), we finish the proof.

\bigskip

{\bf Step 1: Algebraic identities for a partial result}

We justify estimate \eqref{ident*} here. Let us first derive from the transformation \eqref{defw**} some algebraic identities linking $r$, $R$ and their derivatives. Then, we will see that the set $B_-(s)$ will naturally emerge from that. More precisely, we claim the following: 
%%%%%%%%%%%%%%%%%%%%%%%%%%%%%%%%%%%%
%%%%%%%%%%%%%%%%%%%%%%%%%%%%%%%%%%%%
\begin{cl}[Transformation of the error terms]\label{clalgebre}
 It holds that
\begin{align}
r_1&=(1-T(\tx) e^{\ts})^{-\frac 2{p-1}} \chr_1,\label{faou1}\\
|\nabla_\ty r_1|^2-(y\cdot \nabla_y r_1)^2&=(1-T(\tx) e^{\ts})^{-\frac {2(p+1)}{p-1}}\left[|\nabla_\cy \chr_1|^2-(Y\cdot \nabla_Y R_1)^2
\right.\label{faou2}\\
&\left.
-\left(\frac{e^s(\tx+yT(\tx))}{1-T(\tx) e^{\ts}}\cdot\nabla_Y \chr_1\right)
\left(\frac{2y+e^s(\tx-yT(\tx))}{1-T(\tx) e^{\ts}}\cdot\nabla_Y \chr_1\right)
\right],\nonumber\\
r_2&= (1-T(\tx) e^{\ts})^{-\frac {p+1}{p-1}}\left[\chr_2+\frac {2T(\tx) e^{\ts}}{p-1}\chr_1\right.\label{faou3}\\
&\left.
+\frac{\tx e^{\ts}+yT(\tx)e^{\ts}}{1-T(\tx) e^{\ts}}\cdot\nabla_Y \chr_1\right]
+\frac{\bar d'(S)}{1-T(x)e^s}\sum_{dir \in \q D}F_{dir}
\end{align}
where $r=r(x,y,s)$ defined by \eqref{defh*}, $R=R(Y,S)$ defined by \eqref{defg**}, $(\cy,\cs)$ and $(\ty, \ts)$ are linked by \eqref{defw**}, and 
\begin{align}
F_{dir} = \tdir\edir\cdot
 &\left\{(\nabla_d\kappa^*_1+xe^s \partial_\nu \kappa^*_1)(\bar d(S) \edir, \ndir(x,s),y)\right.\label{defFdir}\\
&\left.-(1-T(x)e^s)^{-\frac 2{p-1}}\nabla_d\kappa(\bar d(S) \edir,Y)\right\}.
\end{align}
\end{cl}
%%%%%%%%%%%%%%%%%%%%%%%%%%%%%%%%%%%%
%%%%%%%%%%%%%%%%%%%%%%%%%%%%%%%%%%%%
\begin{proof}
See below.
\end{proof}
Let us first use this claim to derive estimate \eqref{ident*}.

\medskip

\begin{proof}[Proof of estimate \eqref{ident*}, assuming that Claim \ref{clalgebre} holds]
From Claim \ref{clalgebre}, we see that a good definition for the set $B_-(s)$ needed in \eqref{ident*} would be the following:
\begin{equation}\label{defb-s}
B_-(s) = \{|y|<1\}\cap \{1-|y|^2\le 2(1-|Y|^2)\}\cap \{x_1e^s \le 2(1-|Y|^2)\},
\end{equation}
where $Y$, $y$ and $s$ are linked by \eqref{defw**}.\\
Before proving that, let us just remark that by definition \eqref{defslsd3} of $\sl(x,A)$ and the transformation \eqref{defw**}, it is easy to see that $Y\sim y$ and $x_1e^s\to 0$ for $s\le \sl(x,A)$ and $|x|$ small enough, hence $B_-(s)$ is a non-empty set which is ``close'' to the unit ball. 

\medskip

Consider $x$ small enough satisfying \eqref{portion3} and \eqref{pyramid3} with $\epsilon = \frac 12$ (see Proposition \ref{prop1}). 
If we consider  $s\le \sl(x,A)$ defined in \eqref{defslsd3}, then we see that
 $e^s\le e^{\sl(x,A)}\le 1/(|T(x)|(\frac{l^{\frac{p-1}2}}{C_0A}-1)\le CA l^{-\frac{p-1}2}/x_1$.
 Using \eqref{pyramid3}, we see that
\begin{equation}\label{txes}
|xe^s|+|T(x) e^s|\le CA l^{-\frac{p-1}2}.
\end{equation}
 Taking $x$ small enough, we see by definition \eqref{defb-s} of $B_-(s)$ that for all $s\in [s_3, \sl(x,A)]$,
\[
\left\|r(x,\ts)\right\|_{\H(\ty \in B_-(\ts))}^2 \le C \left\|R(\cs(x,s))\right\|_{\H}^2+C|\bar d'(S)|^2\sum_{dir \in \q D}\|F_{dir}\|_{\q H(y\in B_-(s))}^2,
\]
where $F_{dir}$ is defined in \eqref{defFdir}. 
Using Lemma \ref{lemchrono}, we see that we can use \eqref{txes} and apply item (vi) of Lemma \ref{lemkd} to derive that for any direction $dir$, we have
\[
\|F_{dir}\|_{\q H(y\in B_-(s))} \le \|F_{dir}\|_{\q H}\le \frac{C(A)}{1-|\bar d(S)|}
\]
Since 
\[
\frac{|\bar d'(S)|}{1-|\bar d(S)|}\le \frac CS
\]
from \eqref{devdbar}, this concludes the proof of estimate \eqref{ident*}.
\end{proof}

Now, it remains to prove Claim \ref{clalgebre}.

\medskip

\begin{proof}[Proof of Claim \ref{clalgebre}] The proof is the same as in the one-dimensional case given in Claim B.1 page 2884 in \cite{MZdmj12}. However, as we mistakenly assumed that the parameter $\bar D(S)$ was constant in the proof of that claim in that paper, we take the opportunity to correct that mistake here, giving all the details of the proof here.\\  
- {\it Proof of \eqref{faou1}}:  If one uses the definitions \eqref{defw**}, \eqref{defg*} and \eqref{defh*} of ${\q T}_x$, $R(Y,S)$ and $r(x,y,s)$, this simply follows by linearity from the identities \eqref{rach1} and \eqref{rach2}.\\
- {\it Proof of \eqref{faou2}}: It follows from the differentiation of \eqref{faou1}.\\
- {\it Proof of \eqref{faou3}}: If $\bar d(S)$ was independent from $S$ (as we mistakenly assumed in \cite{MZdmj12}), then one would see by definitions \eqref{defg**} and \eqref{defh**} of $R_2$ and $r_2$, linearity together with \eqref{rach1}, \eqref{rach2}, \eqref{foursol} and 
\eqref{faou1}
that $R_2=\partial_S R_1$ and $r_2 = \partial_s r_1$. Since this is not true, there will be a defect related to $\bar d'(S)$, which happens to be small, fortunately. More precisely, using the above-mentioned identities and recalling that $\kappa^*_2(d,\nu,y) = \nu \partial_\nu \kappa^*_1(d,\nu,y)$ by definition \eqref{defk*}, we see that 
\begin{align*}
R_2(Y,S)&=\partial_S R_1(Y,S) +\bar d'(S)\sum_{dir \in \q D}\tdir\edir\cdot\nabla_d\kappa(\bar d(S) \edir,Y),\\
r_2(y,s) =& \partial_s r_1(y,s)\\
& +\frac{\bar d'(S)}{1-T(x) e^s}
\sum_{dir \in \q D}\tdir 
\edir\cdot(\nabla_d\kappa^*_1+xe^s \partial_\nu \kappa^*_1)(\bar d(S) \edir, \ndir(x,s),y).
\end{align*}
Differentiating identity \eqref{faou1}, we see that
\begin{align*}
\partial_s r_1(y,s) =&
(1-T(\tx) e^{\ts})^{-\frac {p+1}{p-1}}\left[\partial_S \chr_1(\cy,\cs)+\frac {2T(\tx) e^{\ts}}{p-1}\chr_1(\cy,\cs)
\right.\nonumber\\
&\left.
+\frac{\tx e^{\ts}+yT(\tx)e^{\ts}}{1-T(\tx) e^{\ts}}\nabla_Y \chr_1(\cy,\cs)\right].
\end{align*}
Using these three identities, we see obtain the result.

This concludes the proof of  Claim \ref{clalgebre} and the proof of estimate \eqref{ident*} too.
\end{proof}

\bigskip

{\bf Step 2: Analytic procedure for the full estimate}

This analytical procedure follows in all points what we did with multi-solitons in one space dimension. See Lemma 4.6 page 2864 in \cite{MZdmj12} for details.\\
This concludes the proof of Proposition \ref{prop422}.
\end{proof}

%%%%%%%%%%%%%%%%%%%%%%%%%%%%%%%%%%%
%%%%%%%%%%%%%%%%%%%%%%%%%%%%%%%%%%%
\subsection{A generalized trapping result}
%%%%%%%%%%%%%%%%%%%%%%%%%%%%%%%%%%%
%%%%%%%%%%%%%%%%%%%%%%%%%%%%%%%%%%%
In this section, we prove Lemma \ref{proptrap}. 

\begin{proof}[Proof of Lemma \ref{proptrap}]
We will handle separately the two cases $\nu^*=0$ and $\nu^*\neq 0$ below.

\medskip

{\bf Case $\nu^*=0$}. This is exactly the case treated in \cite{MZtams14} and \cite{MZcmp14}, if we remove the first term from estimate \eqref{proxi0}. Fortunately, the estimate on that term comes from the proof in \cite{MZtams14}, as we explain below. Indeed, using the parameter $A$ visible in \eqref{bornesA}, we make the following 2 observations from the proofs and not just from the statements:\\
1- From the modulation technique in Appendix D of \cite{MZtams14}, we define a modulation parameter $d(s)$ for $s\ge s^*$ such that
\[
\frac{|l_{d^*}(d(s^*)-d^*)|}{1-|d^*|}+
\frac{|d(s^*)-d^*|}{\sqrt{1-|d^*|}}
\le C(A) \bar \epsilon \mbox{ and }
\frac 1{C(A)}\le \frac{1-|d(s^*)|}{1-|d^*|}\le C(A).
\]
We claim that
\begin{equation}\label{clmod}
\frac{|l_{d(s^*)}(d(s^*)-d^*)|}{1-|d(s^*)|}\le C(A) \bar \epsilon.
\end{equation}
If $|d^*|\le \frac 12$ and $\bar \epsilon$ is small, then this is clear, since we are far from $1$.\\
If $\frac 12 < |d^*|<1$ and $\bar e_1(d) = \frac d{|d|}$, then we write 
\[
|\bar e_1(d(s^*))-\bar e_1(d^*)|\le C|d(s^*)-d^*|\le C \bar \epsilon \sqrt{1-|d^*}.
\]
Therefore, 
\begin{align*}
&|\bar e_1(d(s^*))\cdot (d(s^*)-s^*)|\le |\bar e_1(d^*)\cdot(d(s^*)-d^*)|+
|(\bar e_1(d(s^*))-\bar e_1(d^*))\cdot  (d(s^*)-s^*)|\\
\le &C(A) \bar \epsilon(1-|d^*|)+C(A) \bar \epsilon^2 (1-|d^*|),
\end{align*}
and \eqref{clmod} follows.\\
2- Using the proof of the trapping result in Section 3.2 of \cite{MZtams14}, we see that 
\[
\frac{|l_{d_{\infty}}(d(s^*)-d_\infty)|}{1-|d_{\infty}|} + \frac{|d_{\infty}-d(s^*)|}{\sqrt{1-|d_{\infty}|}}\le C(A) \bar \epsilon \mbox{ and }
\frac 1{C(A)}\le \frac{1-|d(s^*)|}{1-|d_{\infty}|}\le C(A).
\]
Arguing as for \eqref{clmod}, we see that
\begin{equation}\label{cldyn}
\frac{|l_{d(s^*)}(d(s^*)-d_{\infty})|}{1-|d(s^*)|}\le C(A) \bar \epsilon.
\end{equation}
Using \eqref{clmod} and \eqref{cldyn}, we see that 
\begin{equation*}
\frac{|l_{d(s^*)}(d^*-d_{\infty})|}{1-|d(s^*)|}\le C(A) \bar \epsilon,
\end{equation*}
which yields by the same trick as before, 
\begin{equation*}
\frac{|l_{d_{\infty}}(d^*-d_{\infty})|}{1-|d_{\infty}|}\le C(A) \bar \epsilon,
\end{equation*}
and this is precisely the missing estimate in \eqref{proxi0}.

\medskip

{\bf Case $\nu^*\neq 0$}. Simply use the similarity variables' transformation \eqref{defw} back and forth in order to reduce to the case $\nu^*=0$.  
For the proof of \eqref{iii}, simply note from the monotonicity of the Lyapunov functional that $\lambda(d^*,\nu^*)$ defined in \eqref{deflambda} is close to $1$, 
 otherwise, the energy at time $s^*$ will be less than $E(\kappa_0)$. This is in fact the one-dimensional argument of \cite{MZdmj12}. Using the link between the quantity $\lambda(d^*, \su^*)$ and the size variable given in \eqref{compare}, we see that $\mu(d^*, \nu^*)\sim 1$, which yields that $\frac{\nu^*}{1-|d^*|}$ is close to $0$. Since the error is propotional to $\ee$ in all our estimates, this leads exactly to the estimate in item (iii).
\end{proof}

\def\cprime{$'$} \def\cprime{$'$}

\noindent{\bf Address}:\\
Universit\'e de Cergy Pontoise, D\'epartement de math\'ematiques, 
2 avenue Adolphe Chauvin, BP 222, 95302 Cergy Pontoise cedex, France.\\
\vspace{-7mm}
\begin{verbatim}
e-mail: merle@math.u-cergy.fr
\end{verbatim}
Universit\'e Paris 13, Institut Galil\'ee, 
Laboratoire Analyse, G\'eom\'etrie et Applications, CNRS UMR 7539,
99 avenue J.B. Cl\'ement, 93430 Villetaneuse, France.\\
\vspace{-7mm}
\begin{verbatim}
e-mail: Hatem.Zaag@univ-paris13.fr
\end{verbatim}
\end{document}